%% file: main.tex
\documentclass[10pt,oneside,a4paper]{amsart}

\usepackage{Preamble} 
\usepackage{import} 
\usepackage{standalone}

\title{A minimal model for prestacks and morphisms of operadic algebras via Koszul duality for box operads}
\author{Lander Hermans}
\address[Lander Hermans]{Universiteit Antwerpen, Departement Wiskunde, Middelheimcampus,
	Middelheimlaan 1,
	2020 Antwerp, Belgium}
	\email{lander.hermans@uantwerpen.be}

\thanks{The author is a predoctoral fellow of the Research Foundation - Flanders (FWO),
file number 1194422N. \\
MSC2020: 18M70, 14A22 (Primary) 18M65, 14A20 (Secondary) \\
Keywords: operad, Koszul duality, prestack, cooperad, twisting morphism, $L_{\infty}$-algebra}

\begin{document}

\maketitle

\begin{abstract}
Prestacks are algebro-geometric objects whose defining relations are far from quadratic. Indeed, they are cubic and quartic, and moreover inhomogeneous. Similarly, a morphism of $\ppp$-algebras for a (nonsymmetric) Koszul operad $\mathcal{P}$ has inhomogeneous relations possibly of any arity. We show that box operads, a rectangular type of operads introduced in \cite{dinhvanhermanslowen2023}, constitute the correct framework to encode them and resolve their relations up to homotopy. Our first main result is a Koszul duality theory for box operads, extending the duality for (nonsymmetric) operads. In this new theory, the classical restriction of being quadratic is replaced by the notion of being \emph{thin-quadratic}, a condition referring to a particular class of ``thin'' operations. Our main cases of interest are the box operad $\Mor(\ppp)$ encoding morphisms of $\ppp$-algebras and the box operad $\Lax$ encoding lax prestacks as algebras. We show that both $\Mor(\ppp)$ and $\Lax$ are not Koszul by explicitly computing their Koszul complex. We then go on to remedy the situation by suitably restricting the respective Koszul dual box cooperads $\Mor(\ppp)^{\antishriek}$ and $\Lax^{\antishriek}$ to obtain our two main applications. As our second main result we establish a minimal (in particular cofibrant) model $\Mor(\ppp)_\infty$ for the box operad $\Mor(\ppp)$ as the cobar construction on a box subcooperad $\Mor(\ppp)^{\antishriek}_{p\leq 1}\subseteq \Mor(\ppp)^{\antishriek}$, hereby answering an open question by Markl \cite[Problem 9]{markl2002}. As expected, it encodes $\infty$-morphisms of $\ppp_\infty$-algebras. As our third main result we establish a minimal model $\Laxinf$ for the box operad $\Lax$ encoding lax prestacks. This sheds new light on Markl's question on the existence of an explicit cofibrant model for the operad encoding presheaves of algebras from \cite[Conj. 31]{markl2002}. Indeed, we answer the parallel question with presheaves viewed as prestacks, and prestacks considered as algebras over a box operad, in the positive.
\end{abstract}

\tableofcontents

\import{}{KD_intro_morph}

\import{}{KD_boxoperads_v2}

\import{}{KD_modules}

\import{}{KD_MorP}

\import{}{KD_Lax}

\def\cprime{$'$}
\providecommand{\bysame}{\leavevmode\hbox to3em{\hrulefill}\thinspace}
\bibliography{Bibfile}
\bibliographystyle{amsplain}

\end{document}

%% file: KD_intro_morph.tex
\section{Introduction}

\subsubsection*{Motivation}
Koszul duality for operads has been proven to be a powerful tool to treat the deformation and homotopy theory of the algebras they encode \cite{ginzburgkapranov1994,  lodayvallette, markl2004, vallette2020, getzlerjones, markl2010, milles2011, vanderlaan2004}. When an operad is Koszul,
it provides an explicit minimal model, resolving the operad, and thus its algebras, up to homotopy. A highly successful application is the operad $\Assoc$ encoding associative algebras, whose minimal model is the operad $\Ainf$.
The main goal of this paper is to develop a Koszul duality theory for box operads and apply it to both the box operad encoding (lax) prestacks and the box operad encoding a morphism of $\ppp$-algebras for a nonsymmetric Koszul operad $\ppp$. In particular, the duality theory presented in this paper extends the one for operads past the standard restriction of being quadratic.

Prestacks naturally appear as models for ``noncommutative spaces'' in the sense of Van den Bergh, where they show up as deformations of the structure sheaf of a scheme and can thus be described as noncommutative deformations thereof \cite{artintatevandenbergh, lowenvandenberghhoch, vandenbergh2, dinhvanliulowen2017}. In a more global picture, they have become part of homological mirror symmetry as originally proposed by Kontsevich in his 1994 ICM address \cite{kontsevich2, aurouxorlov2008}. Meanwhile, prestacks emerged, as twisted presheaves, in the study of the Gerstenhaber-Schack (GS) complex for presheaves of algebras. Gerstenhaber and Schack introduced and studied their namesake complex in the 1980ties with an eye on developing an algebraic model for the Hochschild cohomology for schemes \cite{gerstenhaberschack, gerstenhaberschack1, gerstenhaberschack2}. However this cohomology, which computes bimodule $\Ext$-groups, does not control first order deformations as presheaves, but as prestacks. Prestacks generalise presheaves of categories, which are functors, by relaxing their functoriality up to natural isomorphisms called \emph{twists}. 
The Cohomology Comparison Theorem and the GS complex for presheaves were subsequently generalized to prestacks by Lowen and Van den Bergh \cite{lowenvandenberghCCT} and Dinh Van and Lowen \cite{DVL} respectively.

For associative algebras, operads have proven especially insightful and efficient for studying their deformation and homotopy theory. For instance, they furnish a deformation complex with the structure of a dg Lie algebra \cite{markl2010} which coincides for associative algebras with the (truncated) Hochschild complex \cite{foxmarkl1997}. Furthermore, an array of cohomology comparison results are made available, for instance involving André-Quillen cohomology  and multiple $\Ext$-interpretations \cite{milles2011, doubek2012}. 

As associative dg algebras are not homotopically invariant, resolving the encoding operad also provides a homotopically invariant, yet equivalent, notion, in casu $\Ainf$-algebras. Although they first appeared in topology \cite{stasheff1963}, $\Ainf$-structures became ubiquitous in geometry and physics since the 1990ties \cite{getzlerjones, fukaya1993, kontsevich2, kontsevichsoibelman2001, seidel2002}. Currently, their theory is well-established with applications ranging from homological mirror symmetry \cite{kontsevich2} to the study of Fourier-Mukai functors \cite{rizzardovandenbergh, rizzardovandenberghneeman2019}. Although $\Ainf$-algebras, and more generally $\Ainf$-categories,	 are more complex objects, their advantage stems from having more explicit tools available to handle their homotopy theory, for instance an explicit description of their mapping spaces, a simple description of their model category and a homotopy transfer theorem. In particular, the notion of $\infty$-morphisms turns out to be especially useful, as famously illustrated by Kontsevich formality \cite{kontsevich}. In \cite{markl2004, vallette2020, fresse1967, fresse2010}, this is shown to hold true for algebras over the minimal model of any Koszul operad. A key ingredient in these operadic approaches is a suitably explicit cofibrant replacement of the encoding operad, in casu $\Ainf$. 



In contrast, application of the powerful operadic machinery to prestacks has met obstacles early on. The restricted case of presheaves of associative algebras over a small category $\uuu$ was first considered in \cite{markl2002} by Markl, in which he defined a suitable operad and formulated the open question of finding an explicit minimal, or at least cofibrant, model. In fact, Markl posed this question in the generality of presheaves of $\ppp$-algebras for a (symmetric) operad $\ppp$. For the simplest case $\uuu=(\bullet \rightarrow \bullet)$, the single arrow category, where presheaves correspond to morphisms of $\ppp$-algebras, he obtains results by ad hoc techniques and puts forward the additional question whether these can be obtained through a suitable Koszul duality \cite[Problem 9]{markl2002}. For the particular case $\ppp= \Assoc$, the operad encoding associative algebras, he shows that the minimal model encodes exactly $\infty$-morphisms of $\Ainf$-algebras. Later, relevant results for presheaves of $\ppp$-algebras were obtained by Berger and Moerdijk  \cite{bergermoerdijk2006} \cite{bergermoerdijk2007} in restricted cases and notably by Doubek \cite{doubek2011} \cite{doubek2012} for the general case, but to the best of our knowledge, both questions as posed remain open.

In particular, these general approaches \cite{markl2002, bergermoerdijk2006, bergermoerdijk2007, doubek2011, doubek2012} seem to lack a solution that is both manageable and explicit. 
On the other hand, a large part of the success of Koszul duality can be attributed to providing such resolutions. A necessary requirement for an operad to be Koszul is that its relations are quadratic. In the present paper, we however consider two cases with highly nonquadratic and moreover inhomogeneous relations. Indeed, already encoding a single morphism of associative algebras poses a problem: its algebraic structure consists of multiplications $m$ and restrictions $f$ (encoding the morphism) satisfying the following inhomogeneous quadratic-cubic relation
\begin{equation} \label{functrel}
 m \circ (  f \otimes f ) = f \circ m.
 \end{equation}
 For the case of a morphism of $\ppp$-algebras, there is moreover a quadratic-$(n+1)$-ary relation for every generator of $\ppp$ of arity $n$. More generally, for presheaves the restrictions $f$ additionally satisfy a quadratic-linear relation. 

For prestacks, whose algebraic structure in addition to the components $m$ and $f$ involves twists $c$, the situation is even more dire: 
%
 the relations involving the twists are inhomogeneous cubic and quartic. To remedy this, in the present paper we put forward a novel Koszul duality theory for box operads. This theory generalizes the classical one for (nonsymmetric) operads, while enabling us to resolve relations of higher nature, such as the ones appearing above.

\subsubsection*{Thin-quadratic relations}

Let us start by recalling box operads and present our key insight. In their present formulation, box operads were introduced in \cite{dinhvanhermanslowen2023} by Dinh Van, Lowen and the author in order to study the higher structure on the Gerstenhaber-Schack complex. Box operads are also an enriched instance of a special type of virtual double categories \cite{CruttwellSchulman2010}, \cite{Koudenburg2020}. The latter have been consider by Leinster under the name of fc multicategories \cite{leinsterFc} \cite{Leinster2004}  and originated as instances of multicategories over a monad in the sense of Burroni \cite{burroni1971}.

In $\S \ref{parboxoperads}$, following \cite{dinhvanhermanslowen2023}, we introduce box operads as algebras over a $\N^3$-coloured symmetric operad $\boxop$. Instead of being governed by trees, they are governed by stackings of labeled boxes, as we now will explain. A box operad $\bbb$ consists of a collection of dg modules $(\bbb(p,q,r))_{p,q,r \in \N}$, whose elements $x$ for each label $(p,q,r)$ are pictured as
$$ \scalebox{0.8}{$\input{fcRectangle_x.tikz}$} \quad$$
Note the bottom label is always $1$, so we omit it. Composition is given by stacking of labeled boxes
$$ \scalebox{0.8}{$\input{2Level_stack_x.tikz}$} \quad \longmapsto \quad  \scalebox{0.7}{$\input{2Level_stack_result_x.tikz}$}  .$$
combining both a vertical and horizontal composition simultaneously. 
Further, we show that operads form a full subcategory of box operads by considering a corolla of arity $q$ as a box with label $(0,q,0)$ (Proposition \ref{propthinoperads}). We call these boxes \emph{thin}, depicting them as 
$$ \scalebox{0.9}{$\input{thinbox_intro.tikz}$} \quad $$
with two parallel lines (reminiscent of the equality sign) for the sides labeled $0$.

Thin boxes play a key role in our new Koszul duality as follows. For operads, a quadratic relation consists of stackings of exactly two thin boxes, resulting in a single thin box. For box operads, we define \emph{thin-quadratic} stackings consisting of $n$ boxes such that the number of thin boxes (also counting the resulting box) decreases by exactly $1$. Concretely, thin-quadratic stackings appear in three forms
$$ \scalebox{0.8}{$\input{P_2_1.tikz}$} \quad, \quad \scalebox{0.7}{$\input{partial_comp.tikz}$} \quad \text{ and } \quad \scalebox{0.8}{$\input{P_n.tikz}$} $$
where $i$ denotes the vertical input at which the upper box is composed. The first class corresponds to classical trees with two vertices (represented by two thin boxes). The second class consists of stackings of two boxes for which the bottom box is not thin and the last class consists of stackings with only a single thin stacking at the bottom and an arbitrary number of non-thin boxes on top. Note that their definition (Definition \ref{defthinquadratic}) also involves a necessary topological condition. Our new Koszul duality theory will apply to box operads with thin-quadratic relations, thus extending the classical quadratic setting.

\subsubsection*{Koszul duality for box operads}
Let us now explain the main ingredients of Koszul duality for box operads and their novelties. These results constitute the technical heart of the present paper. Besides our application to prestacks and morphisms of operadic algebras, we consider these results of independent interest.

It is important to mention that we generally follow the structure, notations and naming put forth in \cite[Ch. 6 \& 7]{lodayvallette}. 

In $\S \ref{parbarcobar}$, we establish a bar-cobar adjunction between the categories of \emph{augmented} dg box operads and \emph{conilpotent} dg box cooperads. The cobar functor is defined by endowing the free dg box operad on the coaugmentation coideal of a conilpotent dg box cooperad with a differential induced by its comultiplication. Our construction differs from the usual cobar functor in two ways: it desuspends the cooperad only in thin arities and its differential factors through the space of thin-quadratic stackings as opposed to quadratic trees. Interestingly, the proof of the differential coincides on the level of stackings with the proof of the $\Linf$-structure from \cite{dinhvanhermanslowen2023}. Dually, we obtain the bar functor via the cofree conilpotent dg box cooperad.

Next, in $\S \ref{partwisting}$, we establish an incarnation of the aforementioned $\Linf$-structure on the convolution box operad and obtain a notion of twisting morphism as Maurer-Cartan element. The space of twisting morphisms in turn realises the bar-cobar adjunction (forming a so-called ``Rosetta stone" \cite[Thm. 6.5.10]{lodayvallette}). 

In $\S \ref{parcomplex}$, for a dg box operad $\ccc$ and a conilpotent dg box cooperad $\ccc$, we associate to a twisting morphism $\ka: \ccc \longrightarrow \bbb$ a (right) \emph{twisted complex} $\ccc \smsquaresub{\ka} \bbb$. Here, an essential result is that dg box operads can equivalently be defined as monoids in a skew monoidal category $(\dgMod(k)^{\N^3},\smsquare,k)$, that is, a monoidal category for which the associator and unitors are not necessarily invertible (see for instance \cite{szlachanyi2012, lackstreet2014}).
 The twisted complex $\ccc \smsquaresub{\ka} \bbb$ corresponds to endowing the free right $\bbb$-module on $\ccc$ with a differential by factoring through thin-quadratic stackings and applying $\ka$. In contrast to the situation for operads, the left twisted complex seems to not be available, which we attribute to the lack of free left $\bbb$-modules due to the skewness. On the other hand, the twisted complex for box operads contains a useful subcomplex $\ccc \smsquareconnsub{\ka} \bbb \sub \ccc \smsquaresub{\ka}\bbb$, called the \emph{connected} twisted complex. Indeed, we show the connected bar complex $\B \bbb \smsquareconnsub{\pi} \bbb$ is acyclic (Proposition \ref{acyclic}), but lack a similar result for the full bar complex $\B \bbb \smsquaresub{\pi} \bbb$ (Remark \ref{rmkacyclic}). Conversely, we are able to show that $\ccc \smsquaresub{\iota}\Om \ccc$ is acyclic (Proposition \ref{acyclic}), but not its connected subcomplex. 

In $\S \ref{parfundamental}$, we establish our first main comparison result, a partial fundamental theorem for box operads (Theorem \ref{partialfundamental}). It relates the homology of the connected twisted complex to the question whether the bar construction is a fibrant resolution of the box cooperad. Unfortunately, we are more interested in the dual statement: whether the cobar construction defines a cofibrant resolution of the box operad under consideration. We identify two problems: for operads, this is obtained via the operadic comparison lemma \cite[Lem. 6.4.13]{lodayvallette}, whose proof only extends to box operads for the connected twisted complex (Proposition \ref{comparisonlemma}). On the other hand, the connected cobar complex $\ccc \smsquareconnsub{\iota} \Om \ccc$ is not necessarily acyclic. We solve this problem by considering a smaller subcategory of \emph{inclined} box operads. The notion of being inclined solely pertains to the underlying $\N^3$-collection of dg modules, that is, inclined box operads do not contain any element in arity $(p,q,r)$ for which $p < r$. Our main result is that we can salvage enough of the operadic comparison lemma (Proposition \ref{operadiclemmainclined}) to establish a fundamental theorem for inclined box operads (Theorem \ref{thminclinedfundamental}). A key point is that for inclined box operads the grading $p+q$ for the label $(p,q,r)$ defines an efficient filtration. We end this section by defining thin-quadratic box operads, their Koszul dual box cooperad and the associated Koszul complex.

\subsubsection*{Application $1$: morphisms of $\ppp$-algebras}

Let us now outline our approach for our first application: morphisms of $\ppp$-algebras for a (nonsymmetric) Koszul operad $\ppp$. In \S \ref{section:applicationmorphismPalgebras}, we describe a box operad $\Mor(\ppp)$ encoding such a morphism, to which our Koszul duality applies. For a quadratic presentation of $\ppp$ with module of generators $E$ and relations $R$, the box operad $\Mor(\ppp)$ is generated by elements 
$$  \scalebox{1}{$\input{morP_e.tikz}$} \quad \text{ for } e \in E(q) \text{ and } \quad \scalebox{1}{$\input{lax_f.tikz}$} \quad$$
with labels $(0,q,0)$ and $(1,1,1)$, and besides the quadratic relations $R$, also the thin-quadratic relations
$$\quad \scalebox{1}{$\input{morP_funct1.tikz}$} \quad  =  \quad \scalebox{1}{$\input{morP_funct2.tikz}$}\quad. $$
For $\ppp= \Assoc$, the above equation for the multiplication $m$ is the box operadic version of \eqref{functrel}. 
However, we prove the following theorem.
\begin{theorem}[Proposition \ref{laxnotkoszul}]
The box operad $\Mor(\ppp)$ is not Koszul.
\end{theorem}
More precisely, the cobar construction $\Om(\Mor(\ppp)^{\antishriek})$ on the Koszul dual box cooperad $\Mor(\ppp)^{\antishriek}$ of $\Mor(\ppp)$ turns out to contain too many generators. We remedy the situation by restricting to a box subcooperad $\Mor(\ppp)^{\antishriek}_{r \leq 1}$ of $\Mor(\ppp)^{\antishriek}$ and proving that its associated twisted complex is acyclic. The fundamental theorem for inclined box operads (Theorem \ref{thminclinedfundamental}) then implies the following main theorem.
\begin{theorem}[Theorem \ref{thmlaxinf}]
The dg box operad $\Mor(\ppp)_{\infty}:= \Om \Mor(\ppp)^{\antishriek}_{p \leq 1}$ is a minimal model of $\Mor(\ppp)$ and encodes $\infty$-morphisms of $\ppp_{\infty}$-algebras. 
\end{theorem}
 This answers the open question by Markl \cite[Problem 9]{markl2002} for the case of nonsymmetric operads.
Notice that in $\S 2.4$ we obtain a model structure for box operads (Proposition \ref{propmodel}), providing a suitable framework encompassing minimal models.

\subsubsection*{Application $2$: prestacks}
Let us now outline our approach to prestacks using box operads, which motivated the study of the latter. It turns out that this case is considerably more difficult than Application $1$, thus requiring more intricate arguments.  First, we remark that from the deformation perspective it suffices to study \emph{lax} prestacks, that is, lax functors taking values in the $2$-category of linear categories. Indeed, the invertibility of the twists $c$ of prestacks may be disregarded as it is stable under deformation. In $\S \ref{parlax}$, we first describe an inclined box operad $\Lax$ encoding lax prestacks (Proposition \ref{proplaxalgebra}) to which our Koszul duality applies. It is generated by three elements 
$$\scalebox{1}{$\input{lax_m.tikz}$} \quad,  \quad \scalebox{1}{$\input{lax_f.tikz}$} \quad \text{ and } \quad \scalebox{1}{$\input{lax_c.tikz}$}\quad, $$
with labels $(0,2,0), (1,1,1)$ and $(2,0,1)$ and thin-quadratic relations. For instance, the quadratic-cubic relation \eqref{functrel} for the multiplication $m$ is part of the relations, as is the following cubic-quartic relation
$$ \scalebox{1}{$\input{intro_lax_nat1.tikz}$} \quad  =  \quad \scalebox{1}{$\input{intro_lax_nat2.tikz}$}\quad$$
which encodes the `naturality' of the twists $c$. 
However, we prove the following theorem.
\begin{theorem}[Theorem \ref{laxnotkoszul}]
The box operad $\Lax$ is not Koszul.
\end{theorem}
More precisely, the cobar construction $\Om(\Lax^{\antishriek})$ on the Koszul dual box cooperad $\Lax^{\antishriek}$ of $\Lax$ turns out to contain too many generators again. We remedy the situation by restricting to a subspace $\Lax^{\antishriek}_{r \leq 1}$ and endowing the graded box suboperad 
$$\Laxinf := \Tboxop(\Lax^{\antishriek}_{r \leq 1}) \sub \Om(\Lax^{\antishriek})$$
 with a differential through the projection $\Om(\Lax^{\antishriek}) \twoheadrightarrow \Laxinf$. In contrast with the first application, $\Lax_{r\leq 1}^{\antishriek}$ is no longer a box cooperad. Nonetheless, the proofs of our new Koszul duality restrict well to this setting, resulting in the following main theorem.
\begin{theorem}[Theorem \ref{thmlaxinf}]
The dg box operad $\Laxinf$ is a minimal model of $\Lax$. 
\end{theorem}
This relates to Application $1$ as follows: we have a commuting diagram of dg box operads
$$ 
\begin{tikzcd}
\Laxinf \arrow[r, "\sim"]                             & \Lax                         \\
\Mor(\Assoc)_\infty \arrow[r, "\sim"] \arrow[u, hook] & \Mor(\Assoc) \arrow[u, hook]
\end{tikzcd}$$
where the vertical maps are arity-wise injections and the horizontal maps are quasi-isomorphisms.

Following \cite{markl2004}, we call $\Laxinf$-algebras \emph{strongly homotopy lax prestacks} or \emph{lax prestacks up to homotopy}. 

Let us elaborate on the concept of $\Laxinf$-algebras a little further. $\Laxinf$ admits a simple description: it is freely generated by elements $m_{p,q}$ in degree $p+q-2$ for every $p+q \geq 2$, with an explicit differential (see \eqref{difflaxinf}). The generators $(m_{0,2}, m_{1,1},m_{0,2})$ correspond to the generators $(m,f,c)$ of $\Lax$. Taking the differential of the four generators $(m_{0,3},m_{1,2},m_{2,1},m_{3,0})$ yields the four defining relations of $\Lax$ respectively, or differently said, yields homotopies witnessing them. As a result, the generators $(m_{p,q})_{p+q \geq 2}$ form a coherent set of higher homotopies witnessing higher prestack-type relations. In particular, the subset $(m_{0,q})_{q\geq 2}$ defines individual $\Ainf$-structures and $(m_{1,q})_{q\geq 1}$ encodes $\Ainf$-morphisms between them. As the diagram above shows, this corresponds exactly to the structure encoded by the box suboperad $\Mor(\Assoc)_\infty$.


\subsubsection*{Comparison to existing frameworks}

We now situate our results in the greater scheme of Koszul duality theories for existing operadic frameworks encoding operadic structures. In the past decades, a plethora of operad-related structures have appeared and found applications in a variety of fields, for instance properads \cite{vallette2007, hoffbecklerayvallette2021, merkulovwillwacher2018}, modular operads \cite{getzlerkapranov1998, barannikov2007}, cyclic operads \cite{getzlerkapranov1995}, globular operads \cite{bataninweber2011, batanin2008, batanin2007}, wheeled properads \cite{marklmerkulovshadrin2009} and dioperads \cite{gan2003}. To unify these theories multiple frameworks have been developed, notably groupoid-coloured operads \cite{petersen2013}, multicategories over a monad by Burroni \cite{burroni1971}, see also Leinster \cite{leinsterFc}, Feynmann categories by Kaufmann and Ward \cite{kaufmannward2017}, patterns by Getzler \cite{getzler2009} and operadic categories by Batanin and Markl \cite{bataninmarkl2015}.

These frameworks pin down the necessary information $\mathcal{I}$ to define operads of type $\mathcal{I}$ and a category of algebras over them. Currently, Koszul duality is available in different forms for some of them. 

A natural framework for box operads turns out to be that of coloured symmetric operads, precisely, box operads are algebras over a coloured symmetric operad $\boxop$. A Koszul duality theory for both their operads \cite{vanderlaan2004, ward2022} and algebras \cite[Ch. 11]{lodayvallette} exists, yet it requires a Koszul twisting morphism $\ka: \ccc \longrightarrow \boxop$. Unfortunately, this framework does not apply to $\boxop$ as we do not have, nor expect to have, a quadratic presentation of $\boxop$. Indeed, the quadratic presentation of the coloured symmetric operad $\Op$ encoding (nonsymmetric) operads is given by the partial composition trees. These are not available for box operads (Remark \ref{nopartialstacking}).  

For Feynmann categories, Kaufmann and Ward developed an interesting bar-cobar adjunction and a notion of master equation and twisting morphism for so-called cubical Feynmann categories \cite{kaufmannward2023}. A key ingredient here is a degree-function \cite[Def. 7.2.1]{kaufmannward2017}, that is, a function assigning a nonnegative integer to every element satisfying certain conditions. As every coloured symmetric operad induces a Feynmann category, we obtain a Feynmann category $\boxop^{\mathcal{F}}$ whose operads correspond to box operads. Again, thin boxes inspire a natural degree function
$$\deg(S) := \begin{cases} \# \text{ thin boxes} - 1 & S \text{ consist of thin boxes } \\
\# \text{ thin boxes} & \text{ otherwise} \end{cases} $$
for any stacking $S \in \boxop$. However, this function is not \emph{proper} \cite[Def. 7.2.1]{kaufmannward2017}: it is not true that $\deg(S) =0$ if and only if $S$ is the identity stacking. For instance any stacking of only non-thin boxes has degree $0$. Interestingly, these stackings also appear as an obstruction for an efficient weight grading (Remark \ref{remweight}). 

A Koszul duality for operadic categories has also been developed \cite{bataninmarkl2023, bataninmarkl2023a}. However, here too one restricts to operadic categories having a quadratic presention. Nonetheless, an interesting feature of operadic categories, uncovered by Lack in \cite{lack2018}, is that their operads are equivalently defined as monoids in a skew monoidal category, analogous to box operads (Proposition \ref{operadasmonoid}). 

\subsubsection*{Work in progress and future prospects}
We finish by discussing some work in progress. The minimal model $\Laxinf$ opens up new research avenues and applications. 

In \cite{hermanslowenOC}, we will develop a box operadic deformation complex, which requires an explicit cofibrant replacement, and apply it to $\Lax$.
In parallel with the existing $\Linf$-structures obtained by Dinh Van, Lowen and the author employing respectively operads \cite{vanhermanslowen2022} and box operads \cite{dinhvanhermanslowen2023}, we will thus obtain another $\Linf$-structured complex governing the deformations of (lax) prestacks. Moreover, we will extend the André-Quillen- and $\Ext$-cohomology results for algebras over operads from \cite{milles2011}, to algebras over box operads. Applying these results to $\Lax$ should retrieve the results from \cite{gerstenhaberschack, gerstenhaberschack1, gerstenhaberschack2} and \cite{lowenvandenberghCCT}.

Further, we initiate the study of the homotopy theory of prestacks and presheaves of categories. On the one hand, we will utilise $\Laxinf$ to answer Markl's original question \cite{markl2002} by providing an explicit cofibrant resolution $\Fun_\infty$ of the operad $\Fun$ encoding presheaves (also called diagrams). This fits into a larger program of developing homotopy diagrams (notably started in Grothendieck's famous letter ``Pursuing Stacks").

In this regard, our two applications $\Mor(\ppp)_\infty$ and $\Lax_\infty$ prompt the question whether they can be simultaneously generalised. Conjecturally, for a nonsymmetric operad $\ppp$ this would be the box operad $\Lax(\ppp)$ encoding lax functors taking values in a certain $2$-category of $\ppp$-categories, functors and natural transformations. An interesting application of resolving $\Lax(\ppp)$, would be to obtain an explicit cofibrant resolution of the (box) operad encoding presheaves of $\ppp$-algebras, similarly as to the case of presheaves of associative algebras.

On the other hand, we will study relevant model structures on the category of prestacks and strongly homotopy prestacks. Our first goal will be to obtain a homotopy transfer theorem for algebras over cofibrant box operads, extending \cite{fresse2010}. Following \cite{bergermoerdijk2006} and specifically \cite{muro2014}, we will study model structures on categories of box operads and their algebras and look into comparison results for algebras over weakly equivalent box operads. In particular, we will obtain the especially well-behaved model structure on algebras over a cofibrant box operad, in casu $\Laxinf$.


Finally, comparing our results with the results that still hold true for Feynmann categories with non-proper degree function is an interesting avenue for future work. On the other hand, a promising future project is to define an operadic category encoding box operads and studying the associated skew monoidal category.

\medskip 

\noindent \emph{Acknowledgement.} 

I would like to thank Wendy Lowen for the numerous discussions, kind guidance and support during the writing of this paper. The work is strongly influenced by the book \cite{lodayvallette} by Jean-Louis Loday and Bruno Vallette, in particular its insightful and transparent presentation of Koszul duality. I have also had the opportunity to visit Bruno Vallette at Université Paris Sorbonne Nord, which led to many valuable discussions and explanations, especially on the wide variety of aspects involved in Koszul duality. I would also like to thank Bruno for his encouragement to additionally investigate the case of a morphism of operadic algebras. I also would like to thank Vladimir Dotsenko, whose comments with Bruno in Barcelona on an early part of the work helped me pin down the peculiarities of the proposed Koszul duality. I would also like to thank Hoang Dinh Van, Pedro Tamaroff, Martin Markl, Philip Hackney, Joan Bellier-Millès, Ricardo Campos and Ezra Getzler for the valuable discussions, both off- and online. I am also grateful to Ralph Kaufmann for bringing \cite{kaufmannward2023} to my attention.

\medskip

\noindent \emph{Conventions.}

We work over a ground field $k$ of characteristic $0$. However, similar to the case of nonsymmetric operads, we expect many of the results to hold for arbitrary characterisic.

We work enriched over the monoidal category of either $k$-modules $\Mod(k)$, graded $k$-modules $\grMod(k)$ or dg $k$-modules $\dgMod(k)$. In the case of the latter two, we add the prefix ``graded'' or ``dg'' respectively. Further, we will implicitly make use of the standard embeddings $\Mod(k)\hookrightarrow \grMod(k) \hookrightarrow \dgMod(k)$. Note that a large part of the theory of box operads can be developed over a general symmetric monoidal category.

Specifically in this paper, we will call an object of the category $\dgMod(k)^{\N^3}$ a \emph{$\N^{3}$-dg module}.

Operads can either have a single colour or a set of colours, in which case we call them ``coloured operads''. Furthermore, they can come equipped with a symmetric action, in which case we call them ``symmetric operads''. When such an action is not taken into account, we emphasize this by writing ``nonsymmetric''. In general, we omit these adjectives when the correct notion of ``operad'' is clear from context. We add the adjectives ``nonsymmetric'', ``symmetric'' and `` coloured'' when we want to emphasize the extra structure (or lack thereof).

For $C$ a dg module and $z\in \Z$, denote by $s^zC:= ks \otimes C$ its $z$th suspension, where $s$ is a formal element of degree $1$. Note that its differential alters by $(-1)^z$. To fix sign conventions, we work with the Koszul sign in the graded context and we make the following choice of identification
$$ s\inv s = \Id_C =-ss\inv.$$

For coalgebraic structures, we use Sweedler's notation omitting sums and indices.

%% file: KD_boxoperads_v2.tex
\section{Box (co)operads} \label{parboxoperads}

In this section, we introduce box (co)operads as (co)algebras over the $\N^3$-coloured symmetric operad $\boxop$ (pronounced as ``box-op") from \cite{dinhvanhermanslowen2023}. The operad $\boxop$ naturally models $(p,q,r)$-labeled stackings of boxes
$$ \scalebox{0.7}{$\input{2Level_stack.tikz}$} $$
Further, in $\S \ref{parparalgebras}$, we show that a box operad $\bbb$ naturally encodes operations on quivers over small categories, in which case we call them $\bbb$-categories (Definition \ref{defcategory}). In case each quiver has a single object, we call them $\bbb$-algebras (Definition \ref{defalgebra}. This can be understood as representing a $(p,q,r)$-labeled box as an operation from the $q$th tensor product $\A$ to $\A$, alongside a $p-$ and $r-$simplex of $\uuu$ acting on $\Ob(\A)$, the set of objects of $\A$,
$$ \scalebox{0.7}{$\input{boxAlgebra2.tikz}$} $$ 

The remaining goal of this section is to establish the necessary ingredients for the bar and cobar functors defined in $\S \ref{parbarcobar}$.

In $\S \ref{parparmonad}$, we make the monad $\Tboxop$ associated to the operad $\boxop$, which equivalently encodes box operads, explicit.

In $\S \ref{parparcomonad}$, we introduce \emph{conilpotent} box cooperads as coalgebras over a comonad $\Tboxcoopeta$. They form a subcategory of box cooperads, important for Koszul duality, which are often introduced via a converging filtration (see \cite[\S 5.7.6]{lodayvallette}). In contrast with the free box cooperad comonad, the conilpotent cofree box cooperad functor agrees with the underlying functor of $\Tboxop$, an essential feature used throughout the paper.

In $\S \ref{parparmodel}$, we show that box operads come equipped with a natural model structure, lifted from $\dgMod(k)^{\N^3}$.

Finally, in $\S \ref{parparinterplay}$, we realise (nonsymmetric) operads as \emph{thin} box operads, that is, box operads concentrated in arities $(0,q,0)$, constituting a coreflective subcategory of box operads. Further, thin, semi-thin and non-thin boxes are introduced and will play a key role in our Koszul duality.

\subsection{The dg operad $\boxop$} \label{parparboxop}

\begin{mydef}\label{defgenFc}
Let $\boxop$ be the $\N^3$-coloured dg operad generated by  
\begin{itemize}
\item for every $n,\nth{q}, p_0,\ldots,p_n,p ,r \in \N$ 
$$ C^{p_0,\ldots,p_n;p,r}_{q_1,\ldots,q_n} \in \boxop\left( (p,n-1,r),(p_0,q_1,p_1),\ldots,(p_{n-1},q_{n},p_n); (p+p_0,\sum_{i=1}^{n}q_i,r+p_n)\right)$$
\item a unit $\eta \in \boxop\left( ; \left(0,1,0\right) \right)$
\end{itemize}
and concentrated in degree $0$ with trivial differential. 
For \emph{matching} colours
$$(p,n,r),(p_0,m_1,p_1),\ldots,(p_{n-1},m_n,p_n),(p^1_0,q^1_1,p^1_1),\ldots,(p^1_{q_1-1},q^1_{m_1},p^1_{m_n}),\ldots,(p^n_0,q^n_1,p ^n_1),\ldots, (p^n_{m_n-1},q^n_{m_n},p^n_{m_n}),$$
that is, such that for $0 < i < n$ we have
$$p^i_{m_i} = p^{i+1}_0,$$
we require associativity and unit relations
\begin{enumerate}
\item $ C_{\underline{q}^1,\ldots,\underline{q}^n}^{\underline{p}^1,\underline{\hat{p}}^2,\ldots,\underline{\hat{p}}^n} \circ C_{\underline{m}}^{\underline{p};p,r} = \left( C^{p_0+p^1_0,\ldots,p_{n-1}+p^{n}_0,p_n+p^n_{m_n};p,r}_{\sum_{j=1}^{m_1} q^1_j,\ldots, \sum_{j=1}^{m_n} q^n_j} \circ ( - , C_{\underline{q}^1}^{\underline{p}^1;p_0,p_1} , \ldots, C_{\underline{q}^n}^{\underline{p}^n;p_{n-1},p_n}) \right)^{\si} $ \label{eqboxopAssoc}
\end{enumerate}
where $\si$ is the usual shuffle
$$( x; (x_1; \ofromto{x^1}{q_1}) \otimes \ldots \otimes (x_n; \ofromto{x^n}{q_n})) = ((x;x_1 \otimes \ldots \otimes x_n); x_1^1 \otimes \ldots \otimes x_{q_n}^n) $$
and $\underline{\hat{p}}^i := p^i_1,\ldots,p^i_{m_i}$ the (possibly empty) sequence omitting the first element, 
\begin{enumerate}[resume]
\item $C_n^{p,r;0,0} \circ_1 \eta = 1 =  C_{1,\ldots,1}^{0,\ldots,0;p,r} \circ (-,\eta,\ldots,\eta) $,
\end{enumerate}
where $1$ is the operadic identity.
\end{mydef}
\begin{opm}
$\boxop$ is in each arity the free dg $k$-module of the corresponding arity of the underlying operad in sets by the same generators and relations \cite[Def. 2.1]{dinhvanhermanslowen2023}.
\end{opm}

An $\boxop$-algebra is called a \emph{dg box operad}. Denote by $\dgBoxOperad$ the category of dg box operads and their morphisms. Let us spell this out.

\begin{mydef}\label{defboxoperad}
A \emph{dg box operad} $\bbb$ consists of a collection of dg modules $(\bbb(p,q,r))_{p,q,r\in \N}$, composition maps
$$\mu: \bbb(p,n,r) \otimes \bigotimes_{i=1}^n\bbb(p_{i-1},q_i,p_{i}) \longrightarrow \bbb\left(p+ p_0, \sum_{i=1}^n q_i, r + p_{n}\right),$$
and a unit
$$\eta: k \longrightarrow \bbb(0,1,0),$$ 
satisfying associativity and unit axioms: for $x,x_1,\ldots,x_n, x_1^1,\ldots,x^n_{k_n}\in \bbb$ in \emph{matching} arities (see Definition \ref{defgenFc}), we have
\begin{enumerate}
\item \label{boxop_assoc} $\mu( x; \mu(x_1; \ofromto{x^1}{q_1}) \otimes \ldots \otimes \mu(x_n; \ofromto{x^n}{q_n})) = (-1)^{ \si} \mu(\mu(x;x_1 \otimes \ldots \otimes x_n); x_1^1 \otimes \ldots \otimes x_{q_n}^n)$,
\item $\mu(x; 1 \otimes \ldots \otimes 1) = x = \mu(1;x)$,
\end{enumerate}
where $(-1)^\si$ is the Koszul sign associated to the usual shuffle. \newline
A \emph{morphism of dg box operads} $f:\bbb \longrightarrow\bbb'$ consists of a collection of dg module maps $f_{p,q,r}:\bbb(p,q,r) \longrightarrow \bbb'(p,q,r)$ commuting with the composition and unit maps.
\end{mydef}
\begin{opm}
The above definition makes sense in any symmetric monoidal category. 
\end{opm}

\subsubsection{dg box cooperads}

A $\boxop$-coalgebra is called a \emph{dg box cooperad}. Denote by $\dgBoxCooperad$ the category of dg box cooperads and their morphisms. The explicit definition of a dg box cooperad is obtained from Definition \ref{defboxoperad} by reversing the arrows.

\subsubsection{Stackings of boxes}\label{parstackings}

In \cite{dinhvanhermanslowen2023}, we provided a combinatorial description of $\boxop$ (or rather its underlying operad in sets) based on \emph{stackings of boxes}. Here, we review an intuitive definition of this notion of stacking and refer to \cite[\S 2.4]{dinhvanhermanslowen2023} for the details.

Picture a colour $(p,q,r)$ as a labeled box
$$\scalebox{0.7}{$\input{labelpqr.tikz}$}$$
having $q$ vertical inputs, a single vertical output, $p$ horizontal inputs and $r$ horizontal outputs.
A stacking $S \in \Stack((p_1,q_1,r_1),\ldots,(p_n,q_n,r_n);(p,q,r))$ consists of $(p_i,q_i,r_i)$-labeled boxes stacked by ``matching in- and outputs'' such that they form a $(p,q,r)$-labeled box, for instance
$$  \scalebox{0.85}{$\input{stacking_example.tikz}$} \quad . $$
Observe that when a side is labeled by $0$, we draw two parallel lines (reminiscent of an equality sign). In particular, when a top side has label $0$, this creates an open space that does not require to be filled in by a box. Hence, it is naturally `degenerate' and we color it grey.

$\Stack$ forms a $\N^3$-coloured symmetric operad \cite[Def. 2.10]{dinhvanhermanslowen2023}: its composition is given by substitution of stackings in boxes and the symmetric action is given by permuting vertices, i.e. $S^\si$ is the stacking obtained from $S$ by renaming box $i$ by box $\si\inv(i)$. For instance, we can compose the above stacking at box $4$ with another stacking and apply the permuration $(13)$
$$ \left(\quad\scalebox{0.85}{$\input{stacking_example.tikz}$} \quad \circ_4 \quad \scalebox{0.85}{$\input{stacking_example_2.tikz}$}  \quad \right)^{(13)} = \left(  \quad  \scalebox{0.85}{$\input{stacking_composed.tikz}$} \quad \right)^{(13)} = \quad   \scalebox{0.85}{$\input{stacking_composed_2.tikz}$}$$

By adding a unit $\eta$ of arity $(0,1,0)$ with suitable unital relations, we obtain a coloured symmetric operad $\Stack[\eta]$ isomorphic to $\boxop$ \cite[Prop. 2.17]{dinhvanhermanslowen2023}. Under this isomorphism the generator $C^{p_0,\ldots,p_n;p,r}_{q_1,\ldots,q_n}$ of $\boxop$ corresponds to the stacking
$$\scalebox{0.8}{$\input{genboxop.tikz}$}$$

We make fundamental use of two underlying structures of a stacking $S$: the \emph{vertical composite tree} $\V_S$ and the \emph{horizontal composite graph} $\Hh_S$, which we introduce next. 

Let $\Op$ be the $\N$-coloured symmetric operad encoding nonsymmetric operads \cite[\S 2.1]{dinhvanhermanslowen2023} \cite[Def. 4.5]{hawkins}, which can be defined using (full) planar trees and a unit. For a stacking $S$, its vertical composite tree $\V_S$ is built by identifying its vertices with the boxes and the edges by the vertical adjacencies. This constitutes a morphism of coloured symmetric operads
$$\V: \boxop \longrightarrow \Op$$
sending the colour $(p,q,r)$ to $q$. The vertical composite tree $\V_S$ equips the boxes of a stacking $S$ with a total order, i.e. $a \downarrow_S b$ if and only if $a$ is below $b$ or $a$ is left of $b$. We call $S$ \emph{in standard order} if $\downarrow_S$ agrees with the natural order on $\{1,\ldots,n \}$.

Let $\Pro$ be the $\N^2$-coloured symmetric operad encoding pros or asymmetric props \cite[\S 2.2]{dinhvanhermanslowen2023}, which can be defined using planar (directed) graphs. For a stacking $S$, its horizontal composite graph $\Hh_S$ is built by identifying its vertices with the boxes and the number of edges between two vertices by the number of horizontal adjacencies. This constitutes a morphism of coloured symmetric operads
$$\Hh: \boxop \longrightarrow \Pro$$
sending the colour $(p,q,r)$ to $(p,r)$.

For instance, we have 
$$S = \quad \scalebox{0.85}{$\input{example_stacking_simple.tikz}$} \quad ,\quad  \V_S = \quad \scalebox{0.85}{$\input{example_tree_simple.tikz}$}  \quad \text{ and } \quad  \Hh_S = \quad \scalebox{0.85}{$\input{example_graph_simple.tikz}$}\; .$$

\subsection{Algebras and categories over a box operad} \label{parparalgebras}

Let $\uuu$ be a small category.
\begin{mydef}\label{defQuiver}
A \emph{dg quiver $\A$ on $\uuu$} consists of the following
\begin{itemize}
\item for every object $U \in \uuu$ a dg quiver $\A(U)$, consisting of a set of objects $\Ob(\A(U))$ and a dg module of arrows $\A(U)(A,A')$ for every two objects $A,A'$,
\item for every morphism $u:U \longrightarrow V$ in $\uuu$, a function between the respective sets of objects $u\st:\Ob(\A(U)) \longrightarrow \Ob(\A(V))$. 
\end{itemize}
\end{mydef}

\medskip

\noindent \emph{Notations.}
Let $\si = (U_{0} \overset{u_{1}}{\rightarrow} U_{1} \rightarrow  \ldots   \overset{u_{p}}{\rightarrow}  U_{p} )$ be a $p$-simplex in the category $\uuu$. Let $\si_{< t}:= (u_1,\ldots,u_{t-1})$ be the first $t-1$ arrows of $\si$, and $\si_{\geq t} := (u_{t},\ldots,u_p)$ the last $p-t+1$ arrows of $\si$. 

Let $\ev: \nnn(\uuu) \longrightarrow \uuu$ be the evaluation function composing the arrows of a simplex $\si$, i.e. $\ev(\si) = u_p \circ \ldots \circ u_1$. 

Further, we have an induced function $\Ob(\A(U_{0})) \longrightarrow \Ob(\A(U_{p}))$ defined as
$$\si\st := u_p \st \circ \ldots \circ u_1 \st$$
Note that for a $0$-simplex $\si=U$, we set $\si\st$ as the identity function $\Id_{\Ob(\A(U))}$.

\medskip 
 
\begin{mydef}
For $\A$ a dg quiver over $\uuu$, we define the dg module $\End(\A)(p,q,r)$ as follows
$$\End(\A)(p,q,r) := \prod_{\substack{ \si:V \rightarrow U \; p\text{-simplex} \\ \tau:V \rightarrow U \; r\text{-simplex} \\ \ev(\si) = \ev(\tau) }}\prod_{A_0,\ldots,A_q \in \A(U)} \Hom\left( \bigotimes_{i=1}^{q} \A(U)(A_{i-1},A_{i}), \A(V)(\si\st A_0, \tau \st A_q) \right) .$$
\end{mydef}

An element $\te \in \End(\A)(p,q,r)$ can naturally be drawn as a box
$$\scalebox{0.8}{$\input{EndElements.tikz}$}$$
Hence, $\End(\A)$ will have a box operad structure, which we now make explicit.

In order to have a well-defined composition in this context, we require $\uuu$ to \emph{have finite factorisations}, that is, we assume a morphism $f$ to have a finite number of factorisations $f=gh$.

As an intermediary step, we define a horizontal composition of elements in $\End(\A)$ by tensoring over $\A$. For $\te \in \End(\A)(p,q,r)$ and $\te' \in \End(\A)(r,q',r')$, we define a set of elements $\te \otimes_{\A} \te'$: let $\si \in N_p(\uuu)$ and $\tau \in N_{r'}(\uuu)$ such that $\ev(\si) = \ev(\tau)$, and $A=(A_0,\ldots,A_{q+q'})$, then we define
$$(\te \otimes_\A \te')^{\si,\tau}_A :=  \sum_{\substack{ \ga \in N_r(\uuu) \\ \ev(\si)=\ev(\ga) = \ev(\tau) } } \te^{\si,\ga}_{(A_0,\ldots,A_q)} \otimes \te^{\ga,\tau}_{(A_{q+1},\ldots,A_{q+q'})}$$
pictured as
$$ \scalebox{0.8}{$\input{horizontalCompositionEnd.tikz}$} $$ 
Note that this set of elements is not part of $\End(\A)$. 
\begin{prop}
$\End(\A)$ comes equipped with a box operad structure:
\begin{itemize}
\item the composition map
$$\circ: \End(\A)(p,q,r) \otimes \bigotimes_{i=1}^n \End(\A)(p_{i-1},q_i,p_i)  \longrightarrow \End(\A)(p+p_0,q_1 + \ldots + q_n, r+p_n) $$
 composes morphisms of quivers after tensoring over $\A$, i.e. for respective elements $\te,\te_1,\ldots,\te_n \in \End(\A)$, a simplex $\si \in N_{p+p_0}(\A)$ and a simplex $\tau \in N_{r+p_n}(\A)$ such that $\ev(\si)= \ev(\tau)$, we obtain
$$(\te \circ (\te_1,\ldots,\te_n))^{\si,\tau} := \te^{ \si_{>p_0}, \tau_{> p_n}} \circ ( \te_1 \otimes_{\A} \ldots \otimes_{\A} \te_n)^{\si_{\leq p_0},\tau_{\leq p_n}}$$
pictured as
$$\scalebox{0.75}{$\input{boxoperadEnd.tikz}$}.$$
\item the identity map $\Id \in \End(\A)(0,1,0)$ is the unit.
\end{itemize} 
We call $\End(\A)$ the \emph{endomorphism dg box operad of $\A$}.
\end{prop}
\begin{proof}
The associativity of the composition is readily verified by the drawings.
\end{proof}
\begin{opm}
The finiteness condition on the category $\uuu$ can be omitted by working with coloured box operads. 
\end{opm}

\begin{mydef}\label{defcategory}
Let $\bbb$ be a dg box operad, a \emph{dg $\bbb$-category over $\uuu$} consists of a dg quiver $\A$ over $\uuu$  and a morphism of dg box operads $\bbb \longrightarrow \End(\A)$. 
\end{mydef}
\begin{opm}
For $\ppp$ a nonsymmetric dg operad, there exists the lesser well-known notion of dg $\ppp$-category $\A$ (see \cite[\S 5.1]{Leinster2004}): it consists of a (single) dg quiver $\A$ and for every $q\in \N$ and $A_0,A_q \in \Ob(\A)$ a chain map
$$\ppp(q) \otimes \prod_{A_1,\ldots,A_{q-1}}\A(A_{i_1},A_i) \longrightarrow \A(A_0,A_q)$$
which commutes with the composition of $\ppp$ and the composition of maps. Let $\Thin_!(\ppp)$ be the associated dg box operad concentrated in thin arities (see \S \ref{subsubsection:thinboxoperads}), i.e. $\Thin_!(\ppp)(0,q,0)=\ppp(q)$ and $0$ otherwise, then a dg $\Thin_!(\ppp)$-category over $\uuu$ consists of a collection of dg $\ppp$-categories $(\A(U))_{U \in \Ob(\uuu)}$.
\end{opm}

\subsubsection{Algebras over a box operad}

An associative algebra can be seen as a $k$-linear category with a single object. Here, we define an algebra over a box operad $\bbb$ similarly as a $\bbb$-category where every quiver has only a single object. In that case, the endomorphism box operad is easier to describe.

\begin{mydef}\label{defmoduleoverU}
A \emph{dg module $A$ over $\uuu$} is a collection of dg module $(A(U))_{U \in \Ob(\uuu)}$ indexed by the objects of $\uuu$.
\end{mydef}
Considering each dg module $A(U)$ as a dg quiver with a single object, we can describe the endomorphism box operad $\End(A)$ more easily as 
$$\End(A)(p,q,r) = \prod_{\substack{ \si:V \rightarrow U \; p\text{-simplex} \\ \tau:V \rightarrow U \; r\text{-simplex} \\ \ev(\si) = \ev(\tau) }}\Hom_k\left( \A(U)^{\otimes q}, \A(V) \right)$$
for $p,q,r \in \N$. Its composition is then simply provided by the composition of $k$-linear maps whenever the indexing simplices correspond.

\begin{mydef}\label{defalgebra}
Let $\bbb$ be a dg box operad, a \emph{dg $\bbb$-algebra over $\uuu$} consists of a dg module $A$ over $\uuu$ and a morphism of dg box operads $\bbb \longrightarrow \End(A)$. 
\end{mydef}
\begin{opm}
For $\ppp$ a nonsymmetric operad and $\Thin_!(\ppp)$ its associated box operad concentrated in thin arities (see \S \ref{subsubsection:thinboxoperads}), i.e. $\Thin_!(\ppp)(0,q,0)=\ppp(q)$ and $0$ otherwise, a dg $\Thin_!(\ppp)$-algebra over $\uuu$ consists of a collection of $\ppp$-algebras $(\A(U))_{U \in \Ob(\uuu)}$.
\end{opm}

\subsection{The monad $\Tboxop$} \label{parparmonad}

The operad $\boxop$ induces a monad $\mathcal{T}_{\boxop}$ on $\dgMod(k)^{\N^3}$ whose category of algebras coincides with $\dgBoxOperad$. As a result, box operads $\bbb$ come equipped with a multiplication and unit map
\begin{equation}\label{monadalgebra}
\Tboxop(\bbb) \overset{\mu}{\longrightarrow} \bbb \text{ and } k \overset{\eta}{\longrightarrow} \bbb
\end{equation} 
such that $\mu$ is associative and $\eta$ unital.

From here on, objects of $\dgMod(k)^{\N^3}$ are called $\N^3$-dg modules, that is, a $\N^3$-dg module $N$ is a collection of dg modules $(N(p,q,r))_{p,q,r\in \N}$.

We describe the monad $\Tboxop$ in more detail. For a $\N^3$-dg module $M$, we have 
\begin{equation}\label{freeboxop}
\Tboxop(M)(p,q,r) =  \bigoplus_{(p_i,q_i,r_i)} \left[ \boxop((p_1,q_1,r_1),\ldots,(p_n,q_n,r_n);(p,q,r)) \otimes \bigotimes_{i=1}^n M(p_i,q_i,r_i) \right]_{\Ss_n}
\end{equation}
where $[-]_{\Ss_n}$ denotes taking coinvariants. We write a general element of $\Tboxop$ as $(S;\onth{x})$ for $S \in \boxop$ and $x_1,\ldots,x_n \in M$: this is the representative of its equivalence class, where the $\Ss$-action simply permutes the labels of $S$ and the tuple $(x_1,\ldots,x_n)$. Remark that with permuting a Koszul sign appears. 

As stackings naturally define a total order on its set of vertices, we can canonically choose a representative $(S;\onth{x})$ in every class such that $S$ is in standard order (see \S \ref{parstackings}). Indeed, let $\si \in \Ss_n$ be the unique permutation such that $S^{\si}$ is in standard order, then we have
$$(S;\onth{x})= (S^{\si}; \onth{x} \cdot \si)$$
where $\onth{x} \cdot \si := (-1)^{\si(x)} x_{\si^{-1}(1)} \otimes \ldots \otimes x_{\si^{-1}(n)}$ with $(-1)^{\si(x)}$ the corresponding Koszul sign. 

The composition of the monad $\Tboxop$ is induced by the composition of $\boxop$ (which can be expressed as the substitution of stackings into boxes). As a result, $\Tboxop(M)$ is naturally a dg box operad called the \emph{free dg box operad} and it satisfies the corresponding universal property.

\subsection{Conilpotent box cooperads} \label{parparcomonad}

For the purpose of Koszul duality, we are interested in the subcategory of \emph{conilpotent} dg box cooperads which we now define.

A dg box cooperad $\ccc$ is \emph{coaugmented} if there exists a section of the counit $\eps$, that is, a morphism of dg box operads $\eta: k \longrightarrow \ccc$ such that $\eps \eta = \Id_{\ccc}$. Let $\overccc:= \Coker(\eta)$ be its \emph{coaugmentation coideal}, then $\ccc \cong \overccc \oplus k$.

We will define a conilpotent dg box cooperad as a coalgebra-structure on $\overccc$ over a comonad $\Tboxcoopeta$.

Let $\boxop\st$ be the $\N^3$-coloured symmetric dg cooperad obtained by taking the linear dual of $\boxop$ in every arity (see \cite[\S 5.7.2]{lodayvallette}). Let $\Tboxcoopst$ be the comonad on $\dgMod(k)^{\N^3}$ associated to $\boxop\st$. 

We now describe the following subcomonad $\Tboxcoop \sub\Tboxcoopst$ on $\dgMod(k)^{\N^3}$ in detail: for $M$ a $\N^3$-dg module, we define
\begin{equation*}
\Tboxcoop(M)(p,q,r) = \bigoplus_{(p_i,q_i,r_i)} \left[ \boxop\st((p_1,q_1,r_1),\ldots,(p_n,q_n,r_n);(p,q,r)) \otimes \bigotimes_{i=1}^n M(p_i,q_i,r_i) \right]^{\Ss_n}
\end{equation*}
where $[-]^{\Ss_n}$ denotes taking invariants. We write a general element of $\Tboxcoop$ as $(S;\onth{x})$ for $S \in \boxop\st$ and $x_1,\ldots,x_n \in M$: this is the representative of its equivalence class, where the $\Ss$-action simply permutes the labels of $S$ and the tuple $(x_1,\ldots,x_n)$. Note that taking the canonical representative of a class such that the stacking $S$ is in standard order, provides an isomorphism of functors $\Tboxop \cong \Tboxcoop$.

The comultiplication of $\Tboxcoop$ is induced by the decomposition of $\boxop\st$ (which can be expressed as decomposing a stacking into substackings).

\begin{opm}
Remark that the functor $\Tboxcoopst$ differs from $\Tboxcoop$ by replacing the direct sum with the product. Hence, the category of $\Tboxcoop$-coalgebras is a full subcategory of $\Tboxcoopst$-coalgebras which is a full subcategory of $\dgBoxCooperad$.
\end{opm}

Consider further the submonad $\Tboxcoopeta \sub \Tboxcoop$ where the component $(\eta;)$ corresponding to the unit $\eta \in \boxop(;(0,1,0))$ is omitted. Hence, for $M$ a $\N^3$-dg module, we have $\Tboxcoop(M) \cong \Tboxcoopeta(M) \oplus k$. 

\begin{mydef}
A \emph{conilpotent dg box cooperad} $\overccc$ is a $\Tboxcoopeta$-coalgebra. 
\end{mydef}
The following is immediate.
\begin{lemma}
For $\overccc$ a conilpotent dg box cooperad, $\ccc := \overccc \oplus k $ is a coalgebra over $\Tboxcoop$ with comultiplication
$$ \De_{\ccc} := \Id_k \oplus \De_{\overccc} : \ccc \longrightarrow \Tboxcoop(\overccc)$$
In particular, $\ccc$ is a coaugmented dg box cooperad with coaugmentation $\eta: k \hookrightarrow \overccc \oplus k$ the inclusion.
\end{lemma}
We henceforth call a dg box cooperad $\ccc$ \emph{conilpotent} if it arises as $\ccc = \overccc \oplus k$ from a conilpotent dg box cooperad $\overccc$. Denote by $\dgBoxCooperad^{\conil}$ the category conilpotent dg box cooperads and their morphisms.

\begin{prop}
Let $\ccc = \overccc \oplus k$ be a conilpotent dg box cooperad and $\phi: \ccc \longrightarrow M$ a morphism of $\N^3$-dg modules such that $\phi(\eta)= 0$, then there exists a unique morphism of dg box cooperads $\ccc \longrightarrow \Tboxcoop(M)$ given by 
$$
\begin{tikzcd}
\ccc \arrow[r, "\Delta_{\ccc}"] & \Tboxcoop(\overccc) \arrow[r, "\Tboxcoop(\phi)"] & \Tboxcoop(M)
\end{tikzcd}$$
extending $\phi$. 
\end{prop}
\begin{proof}
It suffices to observe that $\ccc\cong \overccc \oplus k, \Tboxcoop(\overccc) \cong \Tboxcoopeta(\overccc) \oplus k$ and $\Tboxcoop(M) \cong \Tboxcoopeta(M) \oplus k$. The result then follows from the comonadic description of $\overccc$.
\end{proof}

\begin{opm}
More commonly, conilpotency is defined using an exhaustive filtration, making sure that the iteration of the comultiplication stabilizes at a certain point. Analogous to operads, for $\boxop\st$-coalgebras this filtration can be defined by the number of boxes in a stacking.
\end{opm}

\subsection{Model structure} \label{parparmodel}

The category $\dgMod(k)^{\N^3}$ is equipped with the projective model structure inherited from $\dgMod(k)$: weak equivalences and fibrations are given arity-wise, that is, a weak equivalence (resp. fibration) $f:M \longrightarrow M'$ consists of a quasi-isomorphism (resp. epimorphism) $f_{p,q,r}:M(p,q,r) \longrightarrow M'(p,q,r)$ for every $p,q,r \in \N$ \cite[\S 4.5]{balchin2021}.

As we work over a field of characteristic zero, this model structure lifts to $\dgBoxOperad$.
\begin{prop}\label{propmodel}
The category $\dgBoxOperad$ carries a model structure where weak equivalences and fibrations are inherited from $\dgMod(k)^{\N^3}$ as arity-wise quasi-isomorphisms and epimorphisms respectively.
\end{prop}
\begin{proof}
As we work over a field of characteristic $0$, the projective model structure on $\dgMod(k)^{\N^3}$ transfers to $\dgBoxOperad$ \cite[Cor. 8.1.2]{whiteyau2018}.
\end{proof}
%
Both for $\dgMod(k)^{\N^3}$ and $\dgBoxOperad$, we call the above weak equivalences quasi-isomorphisms.

We introduce a number of relevant homotopical notions. Let $\bbb$ be a dg box operad.
\begin{itemize}
\item $\bbb$ is said to be \emph{quasi-free} if its underlying graded box operad is free, that is, $\bbb \cong (\Tboxop(E),d)$ for some $\N^3$-graded module of generators $E$ endowed with a differential $d$.
\item A \emph{model} for $\bbb$ is a dg box operad $\mmm$ and a quasi-isomorphism $f:\mmm \longrightarrow \bbb$ of dg box operads. It is a called a \emph{quasi-free model} if $\mmm$ is quasi-free.
\item A quasi-free model $\mmm \cong(\Tboxop(E),d)$ is a \emph{minimal model} if the restriction of $d$ to $E$ is $0$, that is, the map $\Tboxop(E) \overset{d}{\longrightarrow} \Tboxop(E) \twoheadrightarrow E$ is $0$.
\end{itemize}
   
\subsection{Operads as thin box operads} \label{parparinterplay}

\subsubsection{Thin boxes}
A box $(p,q,r)$ is \emph{thin} if $(p,q,r) = (0,q,0)$ and we draw such a box as 
$$ \scalebox{0.85}{$\input{thinbox.tikz}$} $$
A box $(p,q,r)$ is \emph{semi-thin} if $(p,q,r)= (0,q,r) $ and $r>0$ or $(p,q,r)= (p,q,0)$ and $p>0$. A box is \emph{non-thin} if it is neither thin nor semi-thin.

Thin boxes will play a crucial role throughout. In particular, thin boxes may be stacked at the $i$th input
$$  \scalebox{0.85}{$\input{partial_comp.tikz}$} \quad := \quad  \scalebox{0.85}{$\input{partial_comp_units.tikz}$} $$
which we call a \emph{partial stacking}. This is not possible for two arbitrary boxes.
%
%
%

\subsubsection{Thin box operads}\label{subsubsection:thinboxoperads}
Let 
$$\Thin: \Op \longrightarrow \boxop$$
be the map sending the colour $q$ to $(0,q,0)$ which considers a tree $T$ as a stacking of thin boxes $\Thin(T)$. It is immediate it is a morphism of symmetric coloured operads. 

The induced functor 
$$ \Thin^{*}: \dgBoxOperad(k) \longrightarrow \dgOperad(k)$$
defined as $\Thin^{*}(\bbb)(q) = \bbb(0,q,0)$, has a left Quillen functor \cite[Theorem 4.6.4]{hinich}
$$\Thin_! : \dgOperad(k) \longrightarrow \dgBoxOperad(k)  $$
defined as $\Thin_!(\ppp)(p,q,r)= \ppp(q)  \text{ if } (p,q,r) = (0,q,0)$ and $0$ otherwise.
\begin{prop}\label{propthinoperads}
\begin{enumerate}
\item $\Thin^* \circ \Thin_! =\Id_{\dgOperad(k)}$\label{thin1},
\item $\Thin_!$ embeds $\dgOperad(k)$ as coreflective subcategoy of $\dgBoxOperad(k)$.\label{thin2}
\end{enumerate}
\end{prop}
\begin{proof}
It suffices to prove the equality which is immediate from their definitions.
\end{proof}

\section{Bar and Cobar functors} \label{parbarcobar}

The aim of this section is to introduce the bar and cobar functors between \emph{augmented} dg box operads and \emph{conilpotent} dg box cooperads. Let $\bbb$ be such a dg box operad, the bar construction $\B \bbb$ is given by the cofree conilpotent dg box cooperad 
$$\left(\Tboxcoop(s_{\thin} \overline{\bbb}), d_{s_{\thin}\overline{\bbb}} + d_{\square}\right)$$
with a non-trivial differential $d_{\square}$. Let us unpack the relevant notions. 

In $\S \ref{parparinfinitesimal}$, we recall the notions of thin (de)suspension $s_{\thin}^{(-1)}$ and thin-quadratic stacking from \cite[\S 3.2]{dinhvanhermanslowen2023}. We require the latter to define the thin infinitesimal composite $M \infsquare N$ of $\N^3$-dg modules $M$ and $N$, reminiscent of the infinitesimal composite for operads \cite[\S 6.1]{lodayvallette}. Importantly, the composition of a dg box operad $\bbb$ induces a composition on the thin suspended infinitesimal composite
\begin{equation}\label{infcomp}
s_{\thin}\bbb \infsquare s_{\thin}\bbb \overset{\mu_{(1)}^{s_{\thin}}}{\longrightarrow} s_{\thin}\bbb
\end{equation}
of degree $-1$ even though $s_{\thin}\bbb \infsquare s_{\thin}\bbb$ also contains nonquadratic stackings.  

After introducing (co)derivations for box (co)operads in $\S \ref{parparderivation}$, we show in $\S \ref{parparbar}$ our main result (Lemma \ref{bardiff}): the coderivation $d_{\square}$ induced by \eqref{infcomp} is a differential, i.e. it squares to zero. Remark that the proof corresponds largely to the proof of the $\Linf$-structure for box operads from \cite[Theorem 3.8]{dinhvanhermanslowen2023}. The cobar functor is defined dually in $\S \ref{parparcobar}$.

We would also like to draw attention to a peculiarity about signs: it turns out that we have to ad hoc alter a sign in the definition of $\mu_{(1)}^{s_{\thin}}$ in order for Lemma \ref{bardiff} to hold. This feature however persist throughout the theory: it also appears implicitly in Proposition \ref{propconvolution} and explicitly in $\S 7.2$ (see Example \ref{exLaxc}).

\subsection{(Co)derivations} \label{parparderivation}

Due to their (co)monadic descriptions, (conilpotent) dg box (co)operads come equipped with a suitable theory of (co)derivations.

\subsubsection{Pseudo-linear composite}

The functors $\Tboxop \cong \Tboxcoop$ are not linear on maps as most stackings consist of more than two boxes. However, given two morphisms $f,f': M\longrightarrow M'$ in $\dgMod(k)^{\N^3}$, we can define their \emph{pseudo-linear composite} 
$$\Tboxop(f;f'): \Tboxop(M) \longrightarrow \Tboxop(M')$$
as
\begin{equation*}
\Tboxop(f;f')(S; \onth{x}) := \sum_{i=1}^n (-1)^{\eps} (S; fx_1 \otimes \ldots \otimes f'x_i \otimes \ldots \otimes fx_n) 
\end{equation*}
where $(-1)^{\eps}$ is the Koszul sign from permuting $f,f',x_1,\ldots,x_n$.

\subsubsection{Derivations}

Let $f:\bbb \longrightarrow\bbb'$ be a morphism of dg box operads, then a morphism of $\N^3$-dg modules $d: \bbb \longrightarrow \bbb'$ is a \emph{derivation with respect to $f$} if the diagram 
$$  \begin{tikzcd}
\Tboxop(\bbb) \arrow[d, "\Tboxop(f;d)"'] \arrow[r, "\mu^\bbb"] & \bbb \arrow[d, "d"] \\
\Tboxop(\bbb') \arrow[r, "\mu^{\bbb'}"]                      & \bbb'               
\end{tikzcd}$$
commutes. This holds when for every stacking $S \in \boxop$ and $\nth{x} \in \bbb$ we have that
\begin{equation*}
d( \mu^\bbb_S( \onth{x} )) = \sum_{i=1}^n (-1)^{ \sum_{j=1}^{i-1} |d||x_j|} \mu_S^{\bbb'}( fx_1 \otimes \ldots \otimes dx_i \otimes \ldots \otimes fx_n)
\end{equation*}

\begin{prop}\label{propderiv}
Let $\psi: M \longrightarrow \Tboxop(M)$ be a morphism in $\dgMod(k)^{\N^3}$, then it uniquely extends to a derivation with respect to $\Id_{\Tboxop(M)}$ as follows
$$ d_\psi: 
\begin{tikzcd}
\Tboxop(M) \arrow[rr, "\Tboxop(\eta^M;\psi)"] && \Tboxop(\Tboxop(M)) \arrow[r, "\mu^{\Tboxop(M)}"] & \Tboxop(M)
\end{tikzcd}$$
where $\eta^M : M \hookrightarrow \Tboxop(M)$ is the unit of the monad.
\end{prop}
\begin{proof}
Using the derivation property and $d_\psi \eta^M = \psi$, we compute
\begin{align*}
d_\psi = d_\psi \mu^{\Tboxop(M)} \Tboxop(\eta^M) &= \mu^{\Tboxop(M)} \Tboxop(\Id;d_\psi) \Tboxop(\eta^M)\\
&= \mu^{\Tboxop(M)} \Tboxop(\eta^M; d_\psi \eta^M)  = \mu^{\Tboxop(M)} \Tboxop(\eta^M;\psi)
\end{align*}
proving uniqueness. It is a standard computation that $d_\psi$ is indeed a derivation.
\end{proof}

\subsubsection{Coderivations}

Dually, let $f:\ccc \longrightarrow\ccc'$ be a morphism of conilpotent dg box cooperads, then a morphism of $\N^3$-dg modules $d: \ccc \longrightarrow \ccc'$ is a \emph{coderivation with respect to $f$} if the diagram 
$$  \begin{tikzcd}
\Tboxcoop(\ccc) \arrow[d, "\Tboxcoop(f;d)"']& \ccc  \arrow[l, "\De^\ccc"']  \arrow[d, "d"] \\
\Tboxcoop(\ccc')                       & \ccc' \arrow[l, "\De^{\ccc'}"']               
\end{tikzcd}$$
commutes. 
\begin{prop}\label{propcoderiv}
Let $\psi: \Tboxcoop(M) \longrightarrow M$ be a morphism in $\dgMod(k)^{\N^3}$, then it uniquely extends to a coderivation with respect to $\Id_{\Tboxcoop(M)}$ as follows
$$d^\psi: 
\begin{tikzcd}
\Tboxcoop(M) \arrow[r, "\De^{\Tboxcoop(M)}"] & \Tboxcoop(\Tboxcoopeta(M)) \arrow[rr, "\Tboxcoop(\eps^M;\psi)"] && \Tboxcoop(M)
\end{tikzcd}$$
where $\eps^M : \Tboxcoop(M) \twoheadrightarrow M$ is the counit of the comonad $\Tboxcoop$.
\end{prop}

\subsection{Thin infinitesimal composite} \label{parparinfinitesimal}

We review the notions of thin (de)suspension \cite[\S 3.2]{dinhvanhermanslowen2023} and thin-quadratic stackings \cite[\S 3.3.2]{dinhvanhermanslowen2023}. Using the latter, we define the infinitesimal composite. 

\subsubsection{Thin (de)suspension}

Let $M$ be a $\N^3$-dg module. The \emph{thin suspension} of $M$ is defined as 
 \begin{equation*}
 s_{\thin} M(p,q,r):= \begin{cases}  sM(0,q,0) & \text{ if } (p,r)=(0,0) \\ M(p,q,r) &\text{ otherwise } \end{cases}
 \end{equation*}
 Analogously, the \emph{thin desuspension} $s_{\thin}\inv M$ is defined by replacing $s$ by $s\inv$ in the above equation.

\subsubsection{Thin-quadratic stackings}
First, we define an auxiliary degree for a stacking. \newline
 Let $S \in \boxop((p_1,q_1,r_1),\ldots,(p_n,q_n,r_n);(p,q,r))$ be a stacking without semi-thin boxes, we define the \emph{thin desuspended degree}
$$\td(S) := \begin{cases} (1+ \#\text{ number of thin boxes}) - n & \text{ if } (p,r) \neq (0,0) \\ 
0 & \text{ if } (p,r) = (0,0) \end{cases}$$

\begin{mydef} \label{defthinquadratic}
A stacking $S \in \boxop$ consisting of $n$ boxes is \emph{thin-quadratic} if
\begin{enumerate}
\item $S$ does not contain semi-thin boxes,
\item $S$ has degree $2-n$,
\item $S$ is in standard order,
\item its horizontal composite graph $\Hh_S$ has exactly two connected components.
\end{enumerate}
\end{mydef}
In concrete terms, a thin-quadratic stacking $S$ consisting of $n$ boxes pertains to one of the three following classes
\begin{itemize}
\item $S$ is \emph{of type $\rom{1}$} if the stacking consists of exactly two thin boxes, i.e. the usual partial stackings
\begin{equation}\label{type1} \scalebox{0.8}{$\input{P_2_1.tikz}$} \end{equation}
\item  $S$ is \emph{of type $\rom{2}$} if it is a partial stacking of the form
\begin{equation}\label{type2} \scalebox{0.7}{$\input{partial_comp.tikz}$}
\end{equation}
where the bottom box is thin, and the top box is non-thin.
\item $S$ is \emph{of type $\rom{3}$} if it is of the form 
\begin{equation}\label{type3}
 \scalebox{0.8}{$\input{P_n.tikz}$}
  \end{equation}
  where 
  \begin{itemize}
 \item the bottom box is thin, and the $(n-1)$ other boxes are non-thin,
 \item the non-thin boxes are horizontally connected, i.e. the vertices in $\Hh_S$ corresponding to the non-thin boxes form a connected graph.
\end{itemize}   
\end{itemize}
\begin{opm}\label{nopartialstacking}
Note that thin-quadratic stackings of type $III$ remedy the lack of a partial stacking of a non-thin box $\theta'$ at the $i$th input of a box $\theta$, that is, the ``stacking"
$$ \scalebox{0.7}{$\input{partial_comp_problem.tikz}$} $$
 is not part of $\boxop$.
\end{opm}

\subsubsection{Thin infinitesimal composite}

Let $M$ and $M'$ be $\N^3$-dg modules, their \emph{thin infinitesimal composite} is defined as 
\begin{equation}
(M \infsquare M')(p,q,r) := \bigoplus_{\substack{ S \in \boxop((p_1,q_1,r_1),\ldots,(p_n,q_n,r_n);(p,q,r))\\ S \text{ thin-quadratic} }} M(p_1,q_1,r_1) \otimes \bigotimes_{i=2}^n M'(p_i,q_i,r_i)
\end{equation}
Note that $M \infsquare M \sub \Tboxop(M) \cong \Tboxcoop(M)$. Let $M \infnsquare M' $ be the summand consisting of the thin quadratic stackings consisting of $n$ boxes, then we have the decomposition
$$M \infsquare M' = \bigoplus_{n\geq 2} M \infnsquare M'.$$

\subsubsection{Thin infinitesimal (de)composition}
\label{parparparinfinitesimalcomp}
For a (conilpotent) dg box (co)operad their (de)composition induces a (de)composition to the infinitesimal composite, denoted respectively $\mu_{(1)}$ and $\De_{(1)}$. 

Let $\bbb$ be a dg box operad. Define its \emph{thin suspended infinitesimal composition}
\begin{equation*}
s_{\thin}\bbb \infsquare s_{\thin}\bbb \overset{\mu_{(1)}^{s_{\thin}}}{\longrightarrow}  s_{\thin}\bbb 
\end{equation*}
as 
\begin{equation*}
 \mu_{(1)}^{s_{\thin}}(S;\sthin x_1 \otimes \ldots \otimes \sthin x_n) := \begin{cases} (-1)^{|x_1|+1} s \mu_{S}^\bbb(x_1 \otimes x_2)  & \text{ if } S \text{ is of type } I \\
(-1)^{|x_1|} \mu_{S}^{\bbb}( x_1 \otimes x_2 ) & \text{ if } S \text{ is of type } II \\
\textcolor{orange}{-} \mu_{S}^\bbb (\onth{x}) & \text{ if } S \text{ is of type } III \end{cases}
\end{equation*}
Note that it is of degree $-1$ as it reduces the number of thin boxes by exactly $1$.

Analogously, for $\ccc$ a conilpotent dg box cooperad, we define the \emph{thin desuspended infinitesimal decomposition} 
$$\sthin\inv \ccc \overset{\De_{(1)}^{\sthin\inv}}{\longrightarrow} \sthin \ccc \infsquare \sthin \ccc$$
as 
\begin{align*}
 \De_{(1)}^{s_{\thin}\inv}(\sthin\inv c):= &\sum_{S \text{ type } I} (-1)^{|c_{(1)}|} (S; s\inv c_{(1)} \otimes s\inv c_{(2)} ) + \sum_{S \text{ type } II}  (-1)^{|c_{(1)}|+1} (S; c_{(1)} \otimes s\inv c_{(2)} )\\
 &\textcolor{orange}{-} \sum_{S \text{ type } III} (S;s\inv c_{(1)} \otimes c_{(2)} \otimes \ldots \otimes c_{(n)} )
 \end{align*}
\begin{opm}\label{rmkadaptedsign}
The introduction of thin (de)suspensions introduces a Koszul sign. However, in order for Lemmas \ref{bardiff} and \ref{cobardiff} to hold, we have adjusted the sign of type $III$ stackings by multiplying with $-1$.
\end{opm}

\subsection{Bar functor} \label{parparbar}

A dg box operad $\bbb$ is \emph{augmented} if there exists a section of the unit $\eta$, that is, a morphism of dg box operads $\eps: \bbb \longrightarrow k$ such that $\eps \eta = \Id_{\bbb}$. Let $\overline{\bbb}:= \Ker(\eps)$ be its \emph{augmentation ideal}, then $\bbb \cong \overline{\bbb} \oplus k$.

\begin{constr}\label{barconstr}
Let $(\Tboxcoop(\sthin\overbbb),d_{\sthin\overbbb})$ be the cofree conilpotent dg box cooperad. Define $d_{\square}$ as the unique coderivation induced by the degree $-1$ map
$$
\begin{tikzcd}
\Tboxcoop(\sthin\overbbb) \arrow[r, two heads] & \sthin\overbbb \infsquare \sthin \overbbb \arrow[r, "\mu_{(1)}^{\sthin}"] & \sthin \overbbb
\end{tikzcd}$$
In concrete terms, for a standard stacking $S$ and $\nth{x} \in \overbbb$, holds
\begin{multline*}
d_{\square}(S;\sthin x_1 \otimes \ldots \otimes \sthin x_n) \\ 
= \sum_{\substack{S = (S' \circ_i S'')^{\si\inv}  \\ S' \text{ standard } \\ S'' \text{ thin infinitesimal}}} (-1)^{\sum_{j=1}^{i-1} |\sthin x_j| + \si (|\sthin x|)} (S'; \sthin x_{\si (1)} \otimes \ldots \otimes \mu_{S''}^{\sthin}(\sthin x_{\si(i)} \otimes \ldots \otimes \sthin x_{\si(k+i-1)} ) \otimes \ldots \otimes \sthin x_{\si(n)}) 
\end{multline*}
Note that $\si$ is uniquely determined by $S,S'$ and $S''$ (as they are in standard order).
\end{constr}
\begin{lemma}\label{bardiff}
$d_{\square}$ is a differential, i.e. $d_{\square}^{2}=0$.

 Consequently $d_{\sthin \overbbb} +d_{\square}$ defines a differential on $\Tboxcoop(\sthin \overbbb)$.
\end{lemma}
\begin{proof}
First note that $d_{\square}$ and $d_{\sthin \overbbb}$ anti-commute. Hence, it suffices to show $d_{\square}^2 =0$. Due to symmetry, this reduces to computing the sum
\begin{equation*}\label{bareq}
 \sum_{\substack{ S,S' \\ \text{ thin infinitesimal } } } \sum_{i=1}^{k} (-1)^{\sum_{j=1}^{i-1} |\sthin x_j|}(-1)^{\si(|\sthin x|)}  \mu^{\sthin}_{S}( \sthin x_{\si(1)} \otimes \ldots \otimes \mu^{\sthin}_{S'} ( \sthin x_{\si(i)} \otimes \ldots \otimes \sthin x_{\si(i+l-1)}) \otimes  \ldots \otimes \sthin x_{\si(k+l-1)}) = 0
\end{equation*}
where $S$ and $S'$ consists of $k$ and $l$ boxes respectively, and $\si \in \Ss_{k+l-1}$ is the unique permutation such that $(S \circ_i S')^\si$ is in standard form.

As shown in \cite[Theorem 3.8]{dinhvanhermanslowen2023}, each term appears exactly twice and we have the following six cases
\begin{align}
&\scalebox{0.8}{$\input{case1_1.tikz}$}   \circ_1    \scalebox{0.8}{$\input{case1_2.tikz}$}  &=   \scalebox{0.8}{$\input{case1_3.tikz}$}  &= \quad \left(  \quad  \scalebox{0.8}{$\input{case1_4.tikz}$}   \circ_1    \scalebox{0.8}{$\input{case1_5.tikz}$} \quad  \right)^{(23)} \label{eqBar1} \\
&\scalebox{0.8}{$\input{case1_1.tikz}$}   \circ_1    \scalebox{0.8}{$\input{case1_2.tikz}$}  &=   \scalebox{0.8}{$\input{case2_1.tikz}$}  &=   \scalebox{0.8}{$\input{case1_4.tikz}$}   \circ_2    \scalebox{0.8}{$\input{case2_2.tikz}$}\label{eqBar2} \\
 &\scalebox{0.8}{$\input{case3_1.tikz}$}   \circ_1    \scalebox{0.8}{$\input{case3_2.tikz}$}  &=   \scalebox{0.8}{$\input{case3_3.tikz}$}  &= \quad \left(  \quad  \scalebox{0.8}{$\input{case3_4.tikz}$}   \circ_1    \scalebox{0.8}{$\input{case3_5.tikz}$}  \quad \right)^{(23)} \label{eqBar3}\\
   &\scalebox{0.8}{$\input{case3_1.tikz}$}   \circ_1    \scalebox{0.8}{$\input{case3_2.tikz}$}  &=   \scalebox{0.8}{$\input{case4_1.tikz}$} &=  \scalebox{0.8}{$\input{case3_4.tikz}$}   \circ_2    \scalebox{0.8}{$\input{case2_2.tikz}$}\label{eqBar4} \\
   &\scalebox{0.8}{$\input{case5_1.tikz}$}   \circ_i    \scalebox{0.8}{$\input{case5_2.tikz}$} &=   \scalebox{0.8}{$\input{case5_3.tikz}$} &= \quad \left( \quad   \scalebox{0.8}{$\input{case5_4.tikz}$}   \circ_1    \scalebox{0.8}{$\input{case5_5.tikz}$} \quad  \right)^{(i+1 \ldots k+1)}\label{eqBar5} \\
   &\scalebox{0.8}{$\input{case6_1.tikz}$}   \circ_i    \scalebox{0.8}{$\input{case5_2.tikz}$} &=   \scalebox{0.8}{$\input{case6_2_fc.tikz}$} &= \quad \left(  \quad \scalebox{0.8}{$\input{case6_3.tikz}$}   \circ_j    \scalebox{0.8}{$\input{case6_4.tikz}$} \quad \right)^{\si} \label{eqBar6}
\end{align}
We compute their signs. In general, we have $S\circ_i S' = (\tilde{S} \circ_j \tilde{S}')^\tau$ where the thin infinitesimal stackings $S,S',\tilde{S}$ and $\tilde{S}'$ consist respectively of $k,l,\tilde{k}$ and $\tilde{l}$ boxes. Then we compute 
\begin{align*}
 &  \mu^{\sthin}_{S}( \sthin x_{1} \otimes \ldots \otimes \mu^{\sthin}_{S'} ( \sthin x_{i} \otimes \ldots \otimes \sthin x_{i+l-1}) \otimes  \ldots \otimes \sthin x_{k+l-1})
 \\
 &= (-1)^{\mu^{\sthin}_S + \mu^{\sthin}_{S'}}  \sthin \mu_{S\circ_i S'}(\ofromto{x}{k+l-1})\\
 &= (-1)^{\mu^{\sthin}_S + \mu^{\sthin}_{S'}}(-1)^{ \si(|x|)} \sthin \mu_{\tilde{S} \circ_j \tilde{S}'}( x_{\si(1)} \otimes \ldots \otimes x_{\si(k+l-1)})\\
 &= (-1)^{\mu^{\sthin}_S + \mu^{\sthin}_{S'}} (-1)^{\mu^{\sthin}_{\tilde{S}} + \mu^{\sthin}_{\tilde{S}'}} (-1)^{ \si(|x|)} \\ 
 &\mu^{\sthin}_{\tilde{S}}( \sthin x_{\si(1)} \otimes \ldots \otimes \mu^{\sthin}_{\tilde{S}'} ( \sthin x_{\si(j)} \otimes \ldots \otimes \sthin x_{\si(j+l'-1)}) \otimes  \ldots \otimes \sthin x_{\si(k'+l'-1)})  
\end{align*}
We thus have to show that
\begin{equation}\label{signbardiff} (-1)^{\mu^{\sthin}_S + \mu^s_{S'}} (-1)^{\mu^{\sthin}_{\tilde{S}} + \mu^{\sthin}_{\tilde{S}'}}  (-1)^{\si(|x|)}  (-1)^{\si(|\sthin x|)} (-1)^{\sum_{t=1}^{i-1} |\sthin x_t|} (-1)^{\sum_{t=1}^{j-1} |\sthin x_{\si(t)}|} = -1
\end{equation}
Note that $|\sthin x| = |x|$ for non-thin rectangles. First, we compute $(-1)^{\mu^{\sthin}_S+ \mu^{\sthin}_{S'}}$ and $(-1)^{\mu^{\sthin}_{\tilde{S}} + \mu^{\sthin}_{\tilde{S}'}}$:
\begin{center}
\begin{tabular}{ c | c | c |  c |  c || c}
\phantom{} & $\mu^{\sthin}_S$ & $\mu^{\sthin}_{S'}$ & $\mu^{\sthin}_{\tilde{S}}$ & $\mu^{\sthin}_{\tilde{S}'}$ & $\mu^{\sthin}_S + \mu^{\sthin}_{S'} + \mu_{\tilde{S}}^{\sthin} + \mu_{\tilde{S}'}^{\sthin}$\\
\hline 
$\eqref{eqBar1}$ & $|\mu_{S'}(x_1,x_2)|+1$ & $|x_1|+1$ & $|\mu_{\tilde{S}'}(x_1,x_3)|+1$ & $|x_1|+1$ & $|x_2|+|x_3|$ \\
$\eqref{eqBar2}$ & $|\mu_{S'}(x_1,x_2)|+1$ & $|x_1|+1$ & $|x_1|+1$ & $|x_2|+1$ & $|x_1|$ \\
$\eqref{eqBar3}$ & $|\mu_{S'}(x_1,x_2)|$ & $|x_1|$ & $|\mu_{\tilde{S}'}(x_1,x_3)|$ & $ |x_1|$ & $|x_2|+|x_3|$ \\
$\eqref{eqBar4}$ &  $|\mu_{S'}(x_1,x_2)|$ & $|x_1|$ & $|x_1|$ & $|x_2|+1$ & $|x_1|+1$ \\
$\eqref{eqBar5}$ & $1$ & $|x_i|$ & $|\mu_{\tilde{S}}(x_1,\ldots, \widehat{x_{i+1}},\ldots,x_{k+1})|$ & $1$ & $\sum_{t=1}^{i-1} |x_t| + \sum_{t=i+2}^{k+1} |x_t|$\\
$\eqref{eqBar6}$ & $1$ & $|x_i|$ & $1$ & $1$ & $|x_i|+1$
\end{tabular}
\end{center}
Next, we compute for the first five cases of $(-1)^{\si(|x|)}  (-1)^{\si(|\sthin x|)} (-1)^{\sum_{t=1}^{i-1} |\sthin x_t|} (-1)^{\sum_{t=1}^{j-1} |\sthin x_{\si(t)}|}$:
\begin{center}
\begin{tabular}{c|c|c|c|c||c}
\phantom{} & $\si(|\sthin x|) $ & $\si(|x|)$ & $\sum_{t=1}^{i-1}|\sthin x_t|$ & $\sum_{t=1}^{j-1}|\sthin x_{\si(t)}|$ & \\
\hline
$(1)$ & $|\sthin x_2||\sthin x_3|$ & $|x_2||x_3|$ & $0$ & $0$ & $|x_2|+|x_3|+1$ \\
$(2)$ & $0$ & $0$ & $0$ & $|\sthin x_1|$ & $|\sthin x_1|$\\
$(3)$ & $|\sthin x_2||\sthin x_3|$ & $|x_2||x_3|$ & $0$ & $0$ & $|x_2|+|x_3|+1$ \\
$(4)$ & $0$ & $0$ & $0$ & $|\sthin x_1|$ & $|\sthin x_1|$ \\
$(5)$ & $|\sthin x_{i+1}|\sum_{t=i+2}^{k+1} |\sthin x_t|$ & $|x_{i+2}|\sum_{t=i+2}^{k+1} |x_t|$ & $\sum_{t=1}^{i-1} |\sthin x_t|$ & $0$ & $\sum_{t=1}^{i-1} |\sthin x_t| + \sum_{t=i+2}^{k+1} |x_t|$   
\end{tabular}
\end{center}
This proves equation \eqref{signbardiff} for the first five cases. In order to verify the sixth case, we look at $\si$ in more detail. It decomposes into two permutations: one which permutes the left-hand side composition into standard order
$$\scalebox{0.8}{$\input{shuffl_1.tikz}$} \quad \overset{(j \ldots i+1)}{\leadsto} \quad  \scalebox{0.8}{$\input{shuffl_2.tikz}$}$$
followed by a permutation that puts the right-hand side composition into standard order
$$\scalebox{0.8}{$\input{shuffl_3.tikz}$} \quad \overset{\phi}{\leadsto} \quad  \scalebox{0.8}{$\input{shuffl_4.tikz}$} $$
As only $x_1$ and $x_{i+1}$ are thin boxes, the difference of $\si(|x|)$ and $\si(|\sthin x|)$ revolves only around the shuffle $(j \ldots i+1)$. Hence, we have that $(-1)^{\si(|x|)} = (-1)^{ \si(|\sthin x|) + \sum_{t=i+1}^{j-1} |x_t|}$. 
As a result, we compute
\begin{equation}
(-1)^{ |x_i| +1 + \sum_{t=i+1}^{j-1} |x_t| + \sum_{t=1}^{i-1} |sx_t| + \sum_{\substack{ t=1 \\ t \neq i+1} }^{j-1} |sx_t| } = -1
\tag{\ref{eqBar6}}
\end{equation}
and thus $d_{\square}$ squares to zero.
\end{proof}

\begin{mydef}
Let $\bbb$ be a dg box operad, then its \emph{Bar construction} $\B \bbb$ is defined as the conilpotent dg box cooperad
$$\B \bbb := (\Tboxcoop(\sthin \overbbb), d_{\square} + d_{\sthin \overbbb})$$
\end{mydef}
The bar construction extends to a functor 
\begin{equation}
\B: \dgBoxOperad_{\aug} \longrightarrow \dgBoxCooperad
\end{equation}
\begin{prop}
The bar construction $\B$ preserves quasi-isomorphisms.
\end{prop}
\begin{proof}
We refer to the proof of \cite[Prop. 6.5.4]{lodayvallette}.
\end{proof}

\subsection{Cobar functor} \label{parparcobar}
We develop the dual picture to the bar construction. Let $\ccc \cong \overccc \oplus k$ be a conilpotent dg box cooperad.

\begin{constr}\label{cobarconstr}
Let $(\Tboxop(\sthin\inv\overccc),d_{\sthin\inv\overccc})$ be the free dg box operad. Define $d_{\square}$ as the unique derivation induced by the degree $-1$ map
$$
\begin{tikzcd}
\sthin\inv\overccc \arrow[r, "\De^{\sthin\inv}_{(1)}"] & \sthin\inv\overccc \infsquare \sthin\inv\overccc \arrow[r, hook] & \Tboxop(\sthin\inv\overccc)
\end{tikzcd}$$
In concrete terms, for a standard stacking $S$ and $\nth{x} \in \overbbb$, holds
$$d_\square( S; \sthin\inv x_{1} \otimes \ldots \otimes \sthin\inv x_{n} ) = \sum_{i=1}^{n} (-1)^{\sum_{j=1}^{i-1} |\sthin \inv x_j|} (S \circ_i S^{'(x_i)};  \sthin \inv x_{1} \otimes \ldots \otimes \sthin\inv x^i_{(1)} \otimes \ldots \otimes \sthin \inv x_{(l)}^i \otimes \ldots \otimes \sthin \inv x_{n})$$
where $\De(x_{i})=  S^{'(x_{i})} \otimes x^i_{(1)} \otimes \ldots \otimes x_{(l)}^i$.
\end{constr}
\begin{lemma}\label{cobardiff}
$d_{\square}$ is a differential, i.e. $d_{\square}^{2}=0$.

 Consequently $d_{\sthin\inv \overccc} +d_{\square}$ defines a differential on $\Tboxop(\sthin\inv \overccc)$.
\end{lemma}

\begin{mydef}
Let $\ccc$ be a conilpotent dg box cooperad, then its \emph{Cobar consctruction} $\Om \bbb$ is defined as the augmented dg box operad
$$\Om \bbb := (\Tboxop(\sthin\inv \overccc), d_{\sthin\inv \overccc} +d_{\square})$$
\end{mydef}
The cobar construction extends to a functor 
\begin{equation}
\Om: \dgBoxCooperad^{\conil} \longrightarrow \dgBoxOperad
\end{equation}
A conilpotent dg box cooperad $\ccc$ is \emph{$2$-connected} if it is non-negatively graded with no thin elements of degree $1$ and a single element of degree $0$ (i.e. the unit). 
\begin{prop}
The cobar construction preserves quasi-isomorphism between $2$-connected dg box cooperads, that is, it is non-negatively graded with no thin elements of degree $1$ and a single element of degree $0$ (i.e. the unit).
\end{prop}
\begin{proof}
We refer to the proof of \cite[Prop. 6.5.8]{lodayvallette}. Note, however, that the thin desuspension only affects the thin elements.
\end{proof}

\section{Twisting morphisms} \label{partwisting}

Given a dg box operad $\bbb$ and a conilpotent dg box cooperad $\ccc$, we introduce in $\S \ref{parparconvolution}$ the \emph{convolution dg box operad} $\Hom(\ccc,\bbb)$ and enhance its shifted totalisation $\prod \Hom(s_{\thin}^{-1}\ccc,\bbb)$ with a suspended $\Linf$-structure (Proposition \ref{propconvolution}). 

In $\S \ref{parpartwisting}$, we introduce the important notion of \emph{twisting morphism} as a Maurer-Cartan element of $\prod \Hom(s_{\thin}^{-1}\ccc,\bbb)$. Although they are maps $s_{\thin}^{-1}\ccc \longrightarrow \bbb$ of degree $0$, we write the set of twisting morphisms as $\Tw(\ccc ,\bbb)$ and often write $\ccc \longrightarrow \bbb$ for a twisting morphism. We also extend $\Tw(-,-)$ to a bifunctor.

 Our main result (Theorem \ref{thmrosetta}), often called the \emph{Rosetta Stone} \cite{lodayvallette}, proves that the functors $\Tw(\ccc,-)$ and $\Tw(-,\bbb)$ are represented by respectively the cobar construction $\Om \ccc$ and the bar construction $\B \bbb$. As a consequence, we obtain that the cobar functor is left adjoint to the bar functor. 

\subsection{Convolution box operad} \label{parparconvolution}

Let $\bbb$ be a dg box operad and $\ccc$ a conilpotent dg box cooperad, then the $\N^3$-dg module 
$$\Hom(\ccc,\bbb)(p,q,r) := \Hom(\ccc(p,q,r), \ppp(p,q,r))$$
with differential $d(f) := d_\bbb f - (-1)^{|f|} f d_\ccc$, comes equipped with a dg box operad structure: for $\nth{f} \in \Hom(\ccc,\bbb)$ and a stacking $S$ we define $\mu_S^{\Hom}(\onth{f})$ as 
\begin{equation}\label{convolutioncomposition}
\begin{tikzcd}
{\ccc(p,q,r)} \arrow[r, "\De_S"] & {\bigotimes_{i=1}^n \ccc(p_i,q_i,r_i) } \arrow[r, "\bigotimes_{i=1}^n f_i"] & {\bigotimes_{i=1}^n \bbb(p_i,q_i,r_i) } \arrow[r, "\mu_S"] & {\bbb(p,q,r)}
\end{tikzcd}
\end{equation}

\begin{prop}
$\Hom(\ccc,\ppp)$ is a dg box operad.
\end{prop}
\begin{proof}
Proving that $d$ is a differential, is classical. The map \eqref{convolutioncomposition} respects the composition of stackings as $\mu^{\bbb}$ and $\De^{\ccc}$ do.
\end{proof}

Consider the dg module $\prod \Hom(\sthin\inv \ccc, \bbb) := \prod_{p,q,r} \Hom(\sthin\inv \ccc,\bbb)(p,q,r)$ and define for $f_1,\ldots,f_n \in \prod \Hom(\sthin \ccc, \bbb)$ the operations
\begin{equation}
(-1)^{\sum_{i=1}^n |f_i| } \sss P_n(f_1 \otimes \ldots \otimes f_n) :  
\begin{tikzcd}
\sthin \inv \ccc \arrow[r, "\De^{\sthin \inv}"] & \sthin \inv \ccc \infnsquare \sthin\inv \ccc \arrow[r, "\otimes_{i=1}^{n} f_i"] & \bbb \infnsquare \bbb \arrow[r, "\mu"] & \bbb
\end{tikzcd}
\end{equation}
and
\begin{equation}
\sss L_n(f_1,\ldots,f_n) := \sum_{\si \in \Ss_n} (-1)^{ \si(f)} \sss P_n(f_{\si(1)},\ldots,f_{\si(n)})
\end{equation}
where $(-1)^{\si(f)}$ is the respective Koszul sign.
\begin{prop}\label{propconvolution}
The operations $(\sss L_n)_{n\geq 2}$ define a suspended $\Linf$-structure on $(\prod\Hom(\sthin\inv \ccc, \bbb),d)$. 
\end{prop}
\begin{proof}
This $\Linf$-structure is directly related to the differential of the cobar construction as follows: we have that $d_{\Om \ccc} = d_{\sthin\inv \ccc} + d_{\square}$ and $d_{\square}$ decomposes as the sum $\sum_{n\geq 2} d_n$ where $d_n$ is defined by the factorization through the summand $
\begin{tikzcd}
\sthin \inv \ccc \arrow[r, "d_n"] & s\inv \ccc \infnsquare \sthin\inv \ccc
\end{tikzcd}$. 

By setting $\sss L_1(f) = \sss P_1(f) := f d_{s\inv \ccc}$ and $d_1 := d_{s\inv\ccc}$,  we have that 
\begin{equation}
\label{eqPndsquare}
\sss P_n(f_1 \otimes \ldots \otimes f_n) = (-1)^{\sum_{j=1}^n |f_{j}|} \mu \circ (f_1\otimes \ldots \otimes f_n) \circ d_n
\end{equation}

We now relate the suspended $\Linf$-relations to the fact that $d_{\Om \ccc}$ squares to zero.

Given $f_1,\ldots,f_n \in \Hom(\sthin\inv \ccc, \bbb)$, we wish to show that
$$ \sum_{\substack{ k+l =n+1 \\ k,l\geq 1 }}\sum_{\chi \in \Sh_{k-1,l}} (-1)^{ \chi(f)} \sss L_k( \sss L_l(f_{\chi(1)},\ldots,f_{\chi(l)}),\ldots, f_{\chi(n)}) = 0  $$
and thus equivalently
$$\sum_{\substack{ k+l =n+1 \\ k,l\geq 1}}\sum_{i=1}^k \sum_{\si \in \Ss_n} (-1)^{ \si(f)+ |\sss P_l| \sum_{j=1}^{i-1}|f_{\si(j)}|} \sss P_k(f_{\si(1)},\ldots, \sss P_l(f_{\si(i)},\ldots,f_{\si(i+l-1)}),\ldots, f_{\si(n)}) = 0  $$
First, we compute 
\begin{align*}
&\sss P_k(f_{\si(1)},\ldots,\sss P_l(f_{\si(i)},\ldots,f_{\si(i+l-1)}),\ldots, f_{\si(n)}) \\
&= \mu (f_{\si(1)},\ldots, \mu (f_{\si(i)},\ldots,f_{\si(i+l-1)})d_l,\ldots,f_{\si(n)})d_k (-1)^{ \sum_{j=1}^n |f_{\si(j)}| + |\sss P_l|+ \sum_{j=i}^{i+l-1}|f_{\si(j)}|} \\
&= \mu (f_{\si(1)},\ldots,f_{\si(n)}) (\Id^{i-1}, d_l , \Id^{k-i})d_k (-1)^{|d_l|\sum_{j=i+l}^n |f_{\si(j)}|+  \sum_{j=1}^n |f_{\si(j)}| + |\sss P_l| + \sum_{j=i}^{i+l-1}|f_{\si(j)}| }
\end{align*}
The sign reduces then to $(-1)^{\sum_{j=1}^{i-1} |f_{\si(j)}| + 1}$. 
We can thus compute for every $n\geq 1$ and $\si \in \Ss_n$
\begin{align*}
&\sum_{\substack{ k+l =n+1 \\ k,l\geq 1}}\sum_{i=1}^k  (-1)^{|\sss P_l| \sum_{j=1}^{i-1}|f_{\si(j)}|} \sss P_k(f_{\si(1)},\ldots,\sss P_l(f_{\si(i)},\ldots,f_{\si(i+l-1)}),\ldots, f_{\si(n)}) \\
&= \sum_{\substack{ k+l =n+1 \\ k,l\geq 1}}\sum_{i=1}^k - \mu (f_{\si(1)},\ldots,f_{\si(n)}) (\Id^{i-1}, d_l , \Id^{k-i})d_k \\
&= - \mu (f_{\si(1)},\ldots,f_{\si(n)})\left(  \sum_{\substack{ k+l =n+1 \\ k,l\geq 1}}\sum_{i=1}^k (\Id^{i-1}, d_l , \Id^{k-i})d_k \right)
\end{align*}
Note that $d_{\Om \ccc}^2 =0$ decomposes as $\sum_{\substack{ k+l =n+1 \\ k,l\geq 1}}\sum_{i=1}^k (\Id^{i-1}, d_l , \Id^{k-i})d_k =0$ for every $n\geq 1$. Hence, the above equation is zero and thus by summing over $\si \in \Ss_n$ the equation is shown for $(\sss L)_{n\geq 1}$.

We finish the proof by showing that 
$$\partial_{d_\bbb}(P_n(f_1,\ldots,f_n))= d_\bbb P_n(f_1,\ldots,f_n) + \sum_{i=1}^n (-1)^{\sum_{j=1}^{i-1}|f_i|} P_n(f_1,\ldots,d_\bbb f_i,\ldots,f_n) = 0.$$
 Indeed, this is a simple consequence of $\mu$ being a morphism of dg box operad $\Tboxop(\bbb)\longrightarrow \bbb$ and $d_\bbb$ being a derivation:
\begin{align*}
d_\bbb P_n(f_1,\ldots,f_n) &= d_\bbb \mu (\nth{f}) d_n =  (-1)^{\sum_{j=1}^{i-1}|f_j|} \mu (f_1,\ldots,d_\bbb f_i,\ldots,f_n)d_n
\end{align*}

\end{proof}

\subsection{Twisting morphism} \label{parpartwisting}

\begin{mydef}
A map $\al: \sthin \inv \ccc \longrightarrow \bbb$ is \emph{a twisting morphism from $\ccc$ to $\bbb$} if it is a Maurer-Cartan element of $\prod \Hom(\sthin \inv \ccc, \bbb)$, that is, a map of degree $0$ satisfying the equation
$$ d(\al) + \sum_{n\geq 2} \sss P_n(\al,\ldots,\al) = 0$$
We often also write $\ccc \longrightarrow \bbb$ for a twisting morphism. In case that either $\ccc$ or $\bbb$ is (co)augmented, we require $\eps \al = 0$ and $\al \eta = 0$.
\end{mydef}

Let $\Tw(\ccc,\bbb)$ denote the set of twisting morphisms from $\ccc$ to $\bbb$.

\begin{prop}
We have a bifunctor
$$\Tw: (\dgBoxCooperad^{\conil})^{\op} \times \dgBoxOperad \longrightarrow \Set$$
sending a conilpotent dg box cooperad $\ccc$ and a dg box operad $\bbb$ to their set of twisting morphisms $\Tw(\ccc,\bbb)$. It acts on morphisms by pre- and post-composition respectively.
\end{prop}
\begin{proof}
For $f:\bbb \longrightarrow \bbb'$ a morphism of dg box operads, it suffices to show that post-composition induces a (strict) morphism of suspended $\Linf$-algebras
\begin{equation} \label{postcomp}
\prod \Hom(s_{\thin}^{-1}\ccc,\bbb) \overset{f \circ - }{\longrightarrow} \prod \Hom(s_{\thin}^{-1}\ccc,\bbb')
\end{equation}
Indeed, as a consequence, it sends a Maurer-Cartan element $\al: s_{\thin}^{-1} \longrightarrow \bbb$ to a Maurer-Cartan element $f \circ \al : s_{\thin}^{-1} \ccc \longrightarrow \bbb'$. The map \eqref{postcomp} preserves the structure maps $\sss L_n$ due to the following commutative diagram 
$$
\begin{tikzcd}
{\ccc(p,q,r)} \arrow[r, "\De_S"] & {\bigotimes_{i=1}^n \ccc(p_i,q_i,r_i) } \arrow[r, "\bigotimes_{i=1}^n f_i"] \arrow[rd, "\otimes_{i=1}^n f\circ f_i"'] & {\bigotimes_{i=1}^n \bbb(p_i,q_i,r_i) } \arrow[r, "\mu_S"] \arrow[d, "\otimes_{i=1}^n f"] & {\bbb(p,q,r)} \arrow[d, "f"] \\
                                 &                                                                                                                       & {\bigotimes_{i=1}^n \bbb'(p_i,q_i,r_i) } \arrow[r, "\mu_S"]                               & {\bbb'(p,q,r)}              
\end{tikzcd} $$
for $f_1,\ldots,f_n \in \Hom(\ccc, \bbb)$ and a stacking $S$.  

Analogously, pre-composition with a morphism of conilpotent dg box cooperads $\ccc \longrightarrow \ccc'$ preserves twisting morphisms $\ccc' \longrightarrow \bbb$. 
\end{proof}

\subsection{Rosetta stone} \label{parparrosetta}

\begin{theorem}\label{thmrosetta}
For an augemented dg box operad $\bbb$ and a conilpotent dg box cooperad $\ccc$, we have natural isomorphisms 
$$\Hom_{\dgBoxOperad}(\Om \ccc, \bbb) \cong \Tw(\ccc,\bbb) \cong \Hom_{\dgBoxCooperad^{\conil}}(\ccc ,\B \bbb)$$
exhibiting $\Om$ as the left adjoint of $\B$.
\end{theorem} 
\begin{proof}
By duality, it suffices to verify one of the natural isomorphisms. Let $f:\Om \ccc \longrightarrow \bbb$ be a morphism of dg box operads, then it is uniquely characterized by its image on the generators $\overline{f}: \sthin\inv \overline{C} \longrightarrow \bbb$ such that $f$ commutes with the differential $d_{\sthin \inv \ccc} + d_{\square}$. This translates to the following commuting diagram
$$\begin{tikzcd}
\sthin\inv\bccc \arrow[d, "d_{\sthin\inv \bccc}+d_{\square}"'] \arrow[rr, "\overline{f}"]                                               &  & \bbb \arrow[d, "d_{\bbb}"] \\
\sthin\inv\bccc \oplus \sthin\inv\bccc \infsquare \sthin\inv\bccc \arrow[rr, "\overline{f} + \mu(\overline{f} \infsquare \overline{f})"'] &  & \bbb                    
\end{tikzcd}$$
As a result, $\overline{f}$ commutes with $d_{\Om \ccc}$ if and only if $d_{\bbb}\overline{f} = \overline{f}d_{s^{-1}_{\thin}\bccc} + \mu (\overline{f} \square_{(1)} \overline{f})d_{\square}$. As we have that $\mu (\overline{f} \square_{(1)} \overline{f})d_{\square}= -\sum_{n\geq 2} \sss P_n(\overline{f},\ldots,\overline{f})$ (see \eqref{eqPndsquare}), it shows that that this is equivalent to $\overline{f}$ satisfying the Maurer-Cartan equation.
\end{proof}

%% file: KD_modules.tex

\section{The twisted complex} \label{parcomplex}

Given a twisting morphism $\al: \ccc \longrightarrow \bbb$ between a conilpotent dg box cooperad $\ccc$ and a dg box operad $\bbb$, we establish in this section a complex $\ccc \smsquaresub{\al} \bbb$ together with a subcomplex $\ccc \smsquareconnsub{\al} \bbb$ called the \emph{twisted complex} and \emph{connected twisted complex} respectively. 

As a first step, in $\S \ref{parparcomposite}$, we show that the box composite $M \smsquare M'$ of two $\N^3$-dg modules $M$ and $M'$, spanned by the $2$-leveled stackings  
$$ \scalebox{0.8}{$\input{2Level_stack_intro.tikz}$} \quad ,$$
defines a left normal skew monoidal category $(\dgMod(k)^{\N^{3}}, \smsquare, k)$ (Proposition \ref{propskewmonoidal}), that is, a monoidal category for which the associator and right unitor 
$$ (M \smsquare M') \smsquare M'' \overset{a}{\hookrightarrow} M \smsquare( M' \smsquare M'') \text{ and } M \overset{r_u}{\hookrightarrow} M \smsquare k $$
are not invertible. Importantly, a box operad is equivalent to a monoid herein (Proposition \ref{operadasmonoid}). Interestingly, skew monoidal categories appeared in recent years in the context of  operadic categories \cite{lack2018}.

Next, in $\S \ref{parpardifferential}$, we show that a map $\al: s_{\thin}\inv\ccc \longrightarrow \bbb$ induces a derivation $d_{\al}$ on the free right $\bbb$-module $s_{\thin}^{-1}\ccc \smsquare \bbb$. Right modules and derivations are treated in $\S \ref{parparmodule}$.
Analogous to the twisted complex for operads \cite[\S 6.4.11]{lodayvallette}, the map $s_{\thin}^{-1}\ccc \longrightarrow s_{\thin}^{-1}\ccc \smsquare \bbb$ inducing $d_{\al}$ factors through the thin infinitesimal composite $s_{\thin}^{-1}\ccc \infsquare s_{\thin}^{-1}\ccc$ and applies $\al$ to the upper level. Proposition \ref{difftwisting} then states that $\al$ is a twisting morphism if and only if $d_\al$ constitutes a differential, i.e. it squares to zero. In contrast with the situation for operads, the left twisting complex seems to not be available, which we attribute to the lack of free left $\bbb$-modules due to the skewness.

On the other hand, in $\S \ref{parparconnected}$, we observe that the differential $d_{\al}$ restricts well to the subspace $s_{\thin}^{-1}\ccc \smsquareconn \bbb$ of the twisted complex consisting of those $2$-level stackings for which at least one level is made up of solely thin boxes.

We finish this section by showing that the complexes 
$$ \B \bbb \smsquareconnsub{\pi} \bbb \text{ and }  \ccc \smsquaresub{\iota}\Om \ccc$$
are acyclic (Lemma \ref{acyclic}), for $ \pi: \B \bbb \longrightarrow \bbb$ and $\iota: \ccc \longrightarrow \Om \ccc$ the twisting morphisms associated to the unit and counit of the bar-cobar adjunction (Theorem \ref{thmrosetta}). The proof of Lemma \ref{acyclic} is a generalisation of \cite[Lem. 6.5.14]{lodayvallette} by extending the corresponding homotopies. In contrast, we are not able to extend the proof to show the acyclicity of their counterparts  $\B \bbb \smsquaresub{\pi} \bbb$ and $\ccc \smsquareconnsub{\iota}\Om \ccc$ (Remark \ref{rmkacyclic}). 

\subsection{The box composite} \label{parparcomposite}

\begin{mydef}
For $\N^3$-dg modules $M$ and $M'$, we define their \emph{box composite} $M \smsquare M'$ as 
$$(M \smsquare M')(p,q,r) := \bigoplus_{ \substack{q = q_1 + \ldots + q_n \\ p_1,\ldots,p_{n-1} \geq 0 \\ p = p_0 + p' \\ r = p_n + r' }}  M(p',n,r') \otimes \bigotimes_{i=1}^n M'(p_{i-1}, q_i ,p_i).$$
\end{mydef}

The box composite models in a natural way the stacking of labeled boxes in two levels
\begin{equation}\label{boxcomposite}
\scalebox{0.8}{$\input{box_composite.tikz}$}
\end{equation}

This naturally extends to morphisms, making the box composite a bifunctor. 

A \emph{(left) skew monoidal category} $(\vvv,\boxtimes,k,a,r_u,l_u)$ consists of the same data as a monoidal category without demanding invertibility of the comparison natural transformations
$$ (v \boxtimes v') \boxtimes v'' \overset{a}{\longrightarrow} v \boxtimes ( v' \boxtimes v''), \quad v \overset{l_u}{\longrightarrow} k \boxtimes v, \quad v \boxtimes k \overset{r_u}{\longrightarrow} v  $$
which are called associator, left and right unitor respectively. They still satisfy the same pentagon and triangle equations of a monoidal category \cite[Def. 2.1]{szlachanyi2012}. It is called \emph{left normal} if the left unitor is invertible \cite{lack2018}.

\begin{prop} \label{propskewmonoidal}
$(\dgMod(k)^{\N^3}, \smsquare,k)$ is a left normal skew monoidal category with unit $k$, defined as $k(0,1,0) := k$ and $0$ otherwise. In particular, for $M,M'$ and $M''$ in $\dgMod(k)^{ \N^3}$ we have
\begin{enumerate}
\item an associator $a:(M \smsquare M') \smsquare M'' \hookrightarrow M \smsquare( M' \smsquare M'')$,
\item a right unitor $r_u:M \hookrightarrow M \smsquare k$ and a left unitor $l_u: k \smsquare M\cong M$. 
\end{enumerate}
\end{prop}
\begin{proof}
We provide drawings for the associators and unitors. 

The associator is the inclusion of components, that is, an arbitrary component of $((M \smsquare M') \smsquare M'')(p,q,r)$ is given by 
\begin{equation}\label{boxcomposite_associator}
\scalebox{0.8}{$\input{boxcomposite_associator.tikz}$}
\end{equation}
We observe that it is also a component of $M\smsquare (M' \smsquare M'')$.
%
The left unitor is simply the isomorphism $k \smsquare M \cong M$ as $k$ is concentrated in arity $(0,1,0)$. Similarly, for $q>0$, the right unitor is an isomorphism $(M\smsquare k) (p,q,r) \cong M(p,q,r)$. On the other hand, for $q=0$, the box composite $M \smsquare k$ is given by
\begin{equation}
(M \smsquare k)(p,0,r) = \bigotimes_{ 0 \leq t \leq \min \{p,r\} } \scalebox{0.8}{$\input{rightunitor.tikz}$}
\end{equation}
Hence, the right unitor similarly is an inclusion of the component $t=0$.

The five coherence axioms \cite[Def. 2.1]{szlachanyi2012} can be immediately verified using similar drawings.
\end{proof}
\begin{opm}
In general, $(M \smsquare M') \smsquare M'' \not\cong M \smsquare ( M' \smsquare M'')$. For example the component 
$$\scalebox{0.85}{$\input{skewmonoidal_counterexample.tikz}$} $$
is part of the latter, but not of the first. 
\end{opm}

\begin{prop}\label{operadasmonoid}
The category of dg box operads is isomorphic to the category of monoids in $(\dgMod(k)^{\N^3}, \smsquare,k)$.
\end{prop}
\begin{proof}
Immediate by unwinding the definitions.
\end{proof}

\subsection{Right $\bbb$-modules} \label{parparmodule}
Let $\bbb$ be a dg box operad. By Proposition \ref{operadasmonoid}, $\bbb$ is a monoid and thus it comes with a notion of left and right $\bbb$-modules. In concrete terms, a \emph{right $\bbb$-module} is a $\N^3$-dg module $M$ equipped with a right action map
\begin{equation}
 \la^M: M \smsquare \bbb \longrightarrow M 
 \end{equation}
satisfying associativity and unit diagrams. In particular, the free right $\bbb$-module functor $\Tbbb$ on a $\N^3$-dg module $M$ is given by $\Tbbb(M):= M \smsquare \bbb$ with right action
\begin{equation}
\begin{tikzcd}
(M \smsquare \bbb) \smsquare \bbb \overset{a}{\hookrightarrow} M \smsquare (\bbb \smsquare \bbb) \arrow[r, "M \smsquare \mu"] & M \smsquare \bbb
\end{tikzcd}
\end{equation}

\subsubsection{Derivations}

Let $M$ be a right $\bbb$-module. A \emph{derivation} $d:M \longrightarrow M$ on a right $\bbb$-module $M$ is a chain map such that the diagram
\begin{equation}
\begin{tikzcd}
M \smsquare \bbb \arrow[r, "\la^M"] \arrow[d, "d \smsquare \bbb + M \smsquarelinear d_{\bbb}"'] & M \arrow[d, "d"] \\
M \smsquare \bbb \arrow[r, "\la^M"]                                                        & M               
\end{tikzcd}
\end{equation}
commutes, where 
\begin{equation}
(M \smsquarelinear d_{\bbb})(x;y_1 \otimes \ldots \otimes y_n) := \sum_{i=1}^n (-1)^{|x| +\sum_{j<i} |y_j|} (x; y_1 \otimes \ldots \otimes d_{\bbb} y_i \otimes \ldots \otimes y_n)
\end{equation}
for $x\in M, y_1,\ldots,y_n \in \bbb$ forming a $2$-level stacking.

The following proposition is standard.
\begin{prop}\label{modulederivation}
Let $\psi:M \longrightarrow M \smsquare \bbb$ be a morphism in $\N^3$-dg modules, then it uniquely extends to a derivation as follows
$$d_\psi:= (M \smsquare \mu)(\psi \smsquare \bbb) + M \smsquarelinear d_{\bbb}$$
where the first term is the following composite
$$M \smsquare \bbb \overset{\psi \smsquare \bbb}{\longrightarrow} (M \smsquare \bbb) \smsquare \bbb \overset{a}{\hookrightarrow} M \smsquare (\bbb \smsquare \bbb) \overset{M \smsquare \mu}{\longrightarrow} M \smsquare \bbb$$
\end{prop}
\begin{opm}
Note that the notion of derivation also makes sense for left $\bbb$-modules. However, due to the direction of the associator, $\bbb \smsquare M$ is not any longer the free left $\bbb$-module on $M$. Hence, the unique extension theorem for $\bbb \smsquare M$-derivations is lacking.
\end{opm}


\subsection{The twisted differential} \label{parpardifferential}

Let $\ccc$ be a conilpotent dg box cooperad and $\bbb$ a dg box operad. Let $\al: s_{\thin}\inv\ccc \longrightarrow \bbb$ be a morphism of $\N^3$-dg modules. Consider the following two maps of degree $-1$
\begin{equation}
\begin{tikzcd}
\sthin\inv \ccc  \arrow[r, hook] & \sthin \inv \ccc \smsquare k \arrow[r, "d_{\sthin \inv \ccc} \smsquare \eta"] & \sthin \inv \ccc \smsquare \bbb
\end{tikzcd}
\end{equation}
and 
\begin{equation}
\begin{tikzcd}
\sthin \inv \ccc  \arrow[r, "\De^{\sthin \inv}_{(1)}"] & \sthin \inv \ccc \infsquare \sthin \inv \ccc \arrow[rr, "\sthin \inv \ccc \infsquare \al"] &  & \sthin \inv \ccc \infsquare \bbb \hookrightarrow \sthin \inv \ccc \smsquare \Tboxop \bbb  \arrow[rr, "\sthin \inv \ccc \smsquare \mu"] &  & \sthin \inv \ccc \smsquare \bbb
\end{tikzcd}
\end{equation}
Geometrically, the second map acts on a non-thin box $c$ as
\begin{equation}\label{diff_dral}
\scalebox{0.7}{$\input{Twist_diff_1.tikz}$} \quad \longmapsto \quad -\; \scalebox{0.7}{$\input{Twist_diff_2_1.tikz}$} \; + (-1)^{|c_{(1)}|} \; \scalebox{0.7}{$\input{Twist_diff_2_2.tikz}$} \quad \longmapsto \quad - \; \scalebox{0.7}{$\input{Twist_diff_3_1.tikz}$} \quad + (-1)^{|c_{(1)}|} \; \scalebox{0.7}{$\input{Twist_diff_3_2.tikz}$} 
\end{equation}
where we denote the green dotted lines as applying the composition $\mu$. 

Let $d_\alpha: \sthin\inv \ccc \smsquare \bbb \longrightarrow \sthin \inv \ccc \smsquare \bbb$ be the unique derivation extending their sum (Proposition \ref{modulederivation}), then it decomposes as 
\begin{equation}
d_\al = d_{\sthin \inv \ccc} \smsquare \bbb + \sthin \inv \ccc \smsquarelinear d_\bbb + d^r_\al
\end{equation}
Note that the first two terms form the differential $d_{\sthin\inv\ccc \smsquare \bbb}$.

\begin{prop}\label{difftwisting}
If $\al \in \Tw(\ccc,\bbb)$, then $d_\al$ is a differential.
\end{prop}
\begin{proof}
It suffices to show that
\begin{equation*}
( d_{\sthin \inv \ccc \smsquare \bbb} d^r_\al + d^r_\al d_{\sthin \inv \ccc \smsquare \bbb} + (d^r_\al)^2 )(c) =0
\end{equation*}
for any $c \in \ccc$. For thin boxes, this is the classical result \cite[Lemma 6.4.12]{lodayvallette}. Assume $c$ is non-thin and write $\De_{(1)}(c) = c_{(1)} \otimes \ldots \otimes c_{(n)}$. In this proof, green dotted lines denote applying the composition $\mu$ of $\bbb$. We will show the following two equations
\begin{equation}\label{difftwisted_eq1}
(d_{\sthin \inv \ccc \smsquare \bbb} d^r_\al + d^r_\al d_{\sthin \inv \ccc \smsquare \bbb})( c) = \sum_{i=2} (-1)^{\sum_{j<i} | c_{(j)}|} \quad   \scalebox{0.7}{$\input{dAldCPTot_2.tikz}$} \quad + \quad
 \scalebox{0.7}{$\input{dAldCPTot_1.tikz}$}
\end{equation}
and 
\begin{equation}\label{difftwisted_eq2}
( d_\al^r)^2  ( c) = \quad \sum_{i\geq 2} (-1)^{\sum_{j<i} |c_{(j)}|} \; \scalebox{0.7}{$\input{dAlFinal_2.tikz}$} \quad  + \scalebox{0.7}{$\input{dAlFinal_1.tikz}$}.
\end{equation}
As $\al$ is a twisting morphism, this completes the proof.

We now prove equation \eqref{difftwisted_eq1}. First, we compute
\begin{align*}
&d_{\sthin\inv \ccc \smsquare \bbb}\; d_\al^r \left( c \right) = d_{\sthin\inv \ccc \smsquare \bbb} \left( \quad  -  \; \scalebox{0.7}{$\input{Twist_diff_3_1.tikz}$}  \quad +  (-1)^{|c_{(1)}|} \; \scalebox{0.7}{$\input{Twist_diff_3_2.tikz}$} \quad \right)   \\
&=\quad - \; \scalebox{0.7}{$\input{dCPdAl_1.tikz}$}  \quad +  \sum_{i=2}^n 	(-1)^{ \sum_{j<i}|c_{(j)}|}\;  \scalebox{0.7}{$\input{dCPdAl_2.tikz}$}  \quad +  (-1)^{|c_{(1)}|} \; \scalebox{0.7}{$\input{dCPdAl_3.tikz}$} \quad + \quad  \scalebox{0.7}{$\input{dCPdAl_4.tikz}$}
\end{align*}
Next, we compute
\begin{align*}
&d_\al^r d_{ \sthin\inv \ccc \smsquare \bbb }(c) \\
&=  (\sthin \inv \ccc \smsquare \mu ) ( \sthin\inv \ccc \infsquare \al ) \sthin\inv \De_{(1)}^{\sthin \inv} (d_\ccc c) \\
&= - (\sthin\inv \ccc \smsquare \mu ) ( \sthin \inv \ccc \infsquare \al ) \sthin\inv d_{\sthin\inv\ccc \infsquare \sthin\inv\ccc}\De_{(1)}^{\sthin \inv}( c) \\
&=   \; \scalebox{0.7}{$\input{dAldCP_11.tikz}$} \quad -    \sum_{i=2}^n (-1)^{\sum_{j<i} |c_{(j)}|} \; \scalebox{0.7}{$\input{dAldCP_12.tikz}$} \quad -  (-1)^{| c_{(1)}|} \; \scalebox{0.7}{$\input{dAldCP_13.tikz}$} \quad  -  \quad  \scalebox{0.7}{$\input{dAldCP_14.tikz}$} 
\end{align*}
where the second equality holds due to $d_\ccc$ being a coderivation.

Finally, we prove equation \eqref{difftwisted_eq2}. We obtain this result in two steps. First, we show that
\begin{align*}
 (d^r_\al)^2(c) &= -(-1)^{\sum_{j <i} |c_{(j)}|} \; \scalebox{0.7}{$\input{dAldAl_13.tikz}$} \quad + \quad  (-1)^{|c_{(i1)}| + \sum_{j<i} |c_{(j)}|} \;   \scalebox{0.7}{$\input{dAldAl_12.tikz}$}  \quad \\
 &+ \quad (-1)^{|s\inv c_{(21)}|} \; \scalebox{0.7}{$\input{dAldAl_11.tikz}$}. 
\end{align*}
Indeed, applying $d_\al^r$ to $c$ provides us with two sets of terms as in \eqref{diff_dral}. Applying $d_\al^r$ to this first set gives
\begin{align*}
d_\al^r \left( \quad  - \; \scalebox{0.7}{$\input{Twist_diff_3_1.tikz}$} \quad \right) &= \quad-  (-1)^{|s\inv c_{(11)}| + |s\inv c_{(12)}|  \sum_{1< j<i}| \al c_{(j)}|} \; \scalebox{0.7}{$\input{dAldAl_5.tikz}$} \\
&=-(-1)^{\sum_{j <i} |c_{(j)}|} \; \scalebox{0.7}{$\input{dAldAl_13.tikz}$}
\end{align*}
where $c_{(i)}$ denotes the left-most box stacked on top of $c_{(12)}$, and the second equation holds do to (co)associativity.

Applying $d_\al^r$ to the second set of terms, we obtain
\begin{align*}
  &d_\al^r \left( \quad  (-1)^{|c_{(1)}|} \; \scalebox{0.7}{$\input{Twist_diff_3_2.tikz}$} \quad \right) \\
&=  \quad -(-1)^{ |c_{(1)}| + |s\inv c_{(2)}| \sum_{j>i} | c_{(1j)}| } \; \scalebox{0.7}{$\input{dAldAl_4.tikz}$} \\ 
&\quad + \quad  (-1)^{|c_{(1)}| + |c_{(11)}|} \; \scalebox{0.7}{$\input{dAldAl_1.tikz}$} \quad + \quad    (-1)^{|c_{(1)}| + |c_{(11)}|} \; \scalebox{0.7}{$\input{dAldAl_3.tikz}$}\\
& \quad + \quad  (-1)^{ | c_{(1)}| + | c_{(11)}| + |s\inv c_{(12)}||s\inv c_{(2)}|}\;  \scalebox{0.7}{$\input{dAldAl_2.tikz}$} \\
&= \quad (-1)^{|c_{(i1)}| + \sum_{j<i} |c_{(j)}|} \;   \scalebox{0.7}{$\input{dAldAl_12.tikz}$}  \quad + \quad (-1)^{|c_{(21)}|} \; \scalebox{0.7}{$\input{dAldAl_11.tikz}$}
\end{align*}
where the second and fourth term in the second equality cancel out due to (co)associativity.

To finish the proof, it thus suffices to show the equality
\begin{align}
\label{diff_summation_pn}\sum_{i\geq 2} (-1)^{\sum_{j<i} |c_{(j)}|} \; \scalebox{0.7}{$\input{dAlFinal_2.tikz}$} \quad &=  \sum_{i\geq 2} \quad  (-1)^{|c_{(i1)}|+ \sum_{j<i} |c_{(j)}| }  \; \scalebox{0.7}{$\input{dAldAl_12.tikz}$} \\
& - (-1)^{ \sum_{j<i} |c_{(j)}|}\quad \sum_{i\geq 2} \quad  \scalebox{0.7}{$\input{dAldAl_13.tikz}$} \notag
\end{align} 
Indeed, we compute
\begin{equation}\label{eqPal}
\scalebox{0.7}{$\input{dAlFinal_2.tikz}$} \quad = \quad - \;   \scalebox{0.7}{$\input{dAlSumP_2.tikz}$} \quad +  (-1)^{|c_{(i1)}|} \; \scalebox{0.7}{$\input{dAlSumP_1.tikz}$} 
 \end{equation}
Again, using the (co)associativity, we can alternate between the unique thin infinitesimal stacking below and above a thin box (see Lemma \ref{bardiff}), and obtain
\begin{align*}
\scalebox{0.7}{$\input{dAlExtra_1.tikz}$} \quad = \quad (-1)^{|c_{(i1)}| + \sum_{i \leq j < k} |c_{(j)}|} \; \scalebox{0.7}{$\input{dAlExtra_2.tikz}$}
\end{align*}
As a result, the summation on the left-hand side of \eqref{diff_summation_pn} reduces to the summation on the right. 
\end{proof}

\begin{mydef}
For a twisting morphism $\al \in \Tw(\ccc,\bbb)$, we define its \emph{twisted complex} $\ccc \smsquaresub{\al}\bbb$ as 
$$(s\inv\ccc \smsquaresub{\al} \bbb, d_\al)$$
\end{mydef}

This construction is functorial in both $\ccc$ and $\bbb$ for compatible twisting morphisms, that is, for a box operad morphism $f: \ccc \longrightarrow \ccc'$ and a box cooperad morphism $g: \bbb \longrightarrow \bbb'$ the square
$$
\begin{tikzcd}
\sthin\inv\ccc \arrow[d, "\al"'] \arrow[r, "\sthin \inv f"] & \sthin \inv \ccc' \arrow[d, "\al'"] \\
\bbb \arrow[r, "g"']                  & \bbb'                  
\end{tikzcd}$$
commutes. Moreover, it is easy to see that in this case $\sthin\inv f \smsquare g: \ccc \smsquaresub{\al} \bbb \longrightarrow \ccc' \smsquaresub{\al'} \bbb'$ is a chain map.

\subsection{The connected twisted subcomplex} \label{parparconnected}

\subsubsection{The connected box composite}

For $\N^3$-dg modules $M$ and $M'$, we define their \emph{connected box composite} $M \smsquareconn M'$ as the dg submodule of $M \smsquare M'$ consisting of the components 
$$M(p',n,r') \otimes \bigotimes_{i=1}^n M'(p_{i-1},q_i,p_i)$$
such that either $p'=r' =0$ (the bottom box is thin), $p_0= \ldots = p_n =0$ (the top row of boxes are thin) or $n=0$ (there is no upper level). Geometrically, the following types of stacking are part of the connected box composite $M \smsquareconn M'$
$$\scalebox{0.8}{$\input{Connected.tikz}$}$$
Note that we consider the case $n=0$ as part of the second type of stacking.
\begin{prop}
Let $\al: \ccc \longrightarrow \bbb$ be a twisting morphism, then 
$$(\sthin \inv \ccc \smsquareconn \bbb,d_\al) \sub \ccc \smsquareconnsub{\al} \bbb $$
is a dg submodule, which we call \emph{the connected twisted complex} and which we denote by $\ccc \smsquareconnsub{\al} \bbb$.
\end{prop}
\begin{proof}
Consider an element $(x;y_1 \otimes \ldots \otimes y_n)$ of the connected composite $\sthin \inv \ccc \smsquareconn \bbb$. If the bottom box $x$ is thin, then it suffices to observe that the same holds for the terms in $d_\al$. If the top row of boxes $y_1,\ldots,y_n$ are thin, then we have two possibilities for the terms in $d_\al(x;y_1\otimes \ldots \otimes y_n)$. Either, $x$ is decomposed into a type $\rom{2}$ stacking, and thus the top row of boxes remains thin, or $x$ is decomposed into a type $\rom{3}$ stacking, and thus its bottom box is thin.
\end{proof}
\begin{opm}
In general, the connected twisted complex is not quasi-isomorphic to the twisted complex as we show in our application of Koszul duality to the box operad $\Lax$ (Theorem \ref{laxnotkoszul} and Remark \ref{notquasiiso}).	
\end{opm}

\subsection{The twisted bar and cobar complex} \label{parpartwistedbar}

The unit and counit of the Bar-Cobar adjunction induce morphisms 
$$ \pi: \Tboxcoop(\sthin\overbbb) \twoheadrightarrow \sthin \overbbb \hookrightarrow \sthin \bbb$$
and 
$$\iota: \sthin\inv \ccc \twoheadrightarrow \sthin \inv \overccc \hookrightarrow \Tboxop(\sthin\inv\overccc)$$
which are twisting morphisms due to Proposition \ref{thmrosetta}. We call them the \emph{universal twisting morphsisms} and the associated twisted complexes $\B \bbb \smsquaresub{\pi} \bbb$ and $\ccc \smsquaresub{\iota} \Om \ccc$ the \emph{twisted bar}.
 and \emph{cobar complex} respectively. 
 
\begin{lemma}\label{acyclic}
The complexes
$$ \B \bbb \smsquareconnsub{\pi} \bbb \text{ and }  \ccc \smsquaresub{\iota}\Om \ccc$$
are acyclic.
\end{lemma}
\begin{proof}
Both cases require similar, but different treatments.

\proofcase{1}{ $\B \bbb \smsquareconnsub{\pi} \bbb$ is acyclic}

It suffices to adapt the proof from [Lem. 6.5.14, Loday-Vallette] by extending the homotopy. We include the full proof for completeness.

Consider the subspaces $\B \bbb \; \smsquareconnupdown{n}{\pi}  \bbb$ made up of stackings with exactly $n$ non-trivial vertices 
, i.e. those elements in $\overbbb$. We define the following increasing filtration $F_{i}(\barbpi) := \bigoplus_{n \leq i} \B \bbb \smsquareconnupdown{n}{\pi} \bbb$. Note that $d_{\pi}F_{i} \subseteq F_{i}$ as $d_{\pi}^r$ of $\B \bbb$ does not increase the number of boxes, and it strictly decreases if the elements in $\bbb$ are not units. Hence, we have a spectral sequence $E_{pq}^{\bullet}$ with zeroth page
 $$E^{0}_{st} = F_{w}((\B \bbb \smsquareconnsub{\pi} \bbb)_{s+t}) / F_{s-1}((\B \bbb \smsquareconnsub{\pi} \bbb)_{s+t}) = (\smsquareconnupdown{s}{\pi}  )_{s+t}$$
 with differential $d^{0} = d_{\overbbb} \smsquare \Id_{\bbb} + \Id_{\B\bbb} \smsquarelinear d_{\bbb} + d$ where 
 $d$ is the part of $d_{\pi}^{r}$ which does not change the number of elements of $\overbbb$ appearing. Hence, $d$ takes an element in $\B \bbb$ and composes it with units in $\bbb$.
 
For $s=0$, we have $(E^{0}_{0t},d^{0})=(k,0)$ and thus $E^{1}_{00}=k$ and $E^{1}_{0t} = 0$ for $t\neq 0$. For $s>0$, we introduce a contracting homotopy $h$ on $E^{0}_{\bullet q}$. On type $I$ and $II$ stackings, we define $h$ as in \cite[Lemma 6.5.14]{lodayvallette}.
On type $III$ stackings, we define
$$h\left(\quad \scalebox{0.7}{$\input{homotopy_11.tikz}$} \quad \right) = (-1)^\mu \quad \scalebox{0.7}{$\input{homotopy_12.tikz}$}$$
where the green dotted lines denotes applying the composition and $(-1)^\mu$ denotes the sign obtained from rearranging the elements $\sthin x_1,\ldots,\sthin x_m$ and $\sthin y_1,\ldots,\sthin y_n$ in order to compose. This map has degree $1$ as solely the degree of the bottom box is altered as it turns from thin to non-thin by composing with the elements $\sthin y_1, \ldots, \sthin y_n$. We will show that it is a homotopy between $\Id$ and $0$, i.e. $hd^0 + d^0 h= \Id$. 

For type $I$ and $II$ stackings, this is the classical result. For type $III$ stackings, we compute
\begin{align}
d^0 h\left(\quad  \scalebox{0.7}{$\input{homotopy_11.tikz}$} \quad \right) &= \left(-1\right)^\mu d^0 \left( \quad  \scalebox{0.7}{$\input{homotopy_12.tikz}$} \quad \right) \nonumber \\
&= \quad \left(-1\right)^\mu \left(d_{\sthin \overbbb} \smsquare \bbb\right)\left(\quad  \scalebox{0.7}{$\input{homotopy_12.tikz}$} \quad \right)\;  + \quad  \scalebox{0.7}{$\input{homotopy_11.tikz}$}  \label{problemhomotopybar}
\end{align}
and 
\begin{align*}
h d^0 \left(\quad  \scalebox{0.7}{$\input{homotopy_11.tikz}$} \quad \right) &= h \left(d_{\sthin \overbbb} \smsquare \bbb  + \B \bbb \smsquarelinear d_{\bbb}\right)\left(\quad  \scalebox{0.7}{$\input{homotopy_11.tikz}$} \quad \right) \\
&= -  \left(d_{\sthin \overbbb} \smsquare  \bbb\right) h \left(\quad  \scalebox{0.7}{$\input{homotopy_11.tikz}$} \quad \right) \\
&= - \left(-1\right)^\mu \left(d_{\sthin \overbbb} \smsquare  \bbb\right)\left(\quad  \scalebox{0.7}{$\input{homotopy_12.tikz}$} \quad \right)
\end{align*}
Therefore, the spectral sequence collapses, that is, $E_{pq}^1 = 0$ for $p>0$ and any $q$, whence showing that the homology of $\B \bbb \smsquareconnsub{\pi} \bbb$ vanishes.

\proofcase{2}{$ \ccc \smsquaresub{\iota} \Om \ccc$ is acyclic}

We extend the homotopy $h$ proving that the classical twisted cobar complex $\ccc \circ_{\iota} \Om \ccc$ is acyclic \cite[Lem. 6.5.14]{lodayvallette}, to $\ccc \smsquaresub{\iota} \Om \ccc$. Indeed, consider an element $\; \scalebox{0.75}{$ \input{cobarhomotopy_1.tikz}$} \; \in \ccc \smsquaresub{\iota} \Om \ccc$, then define the map $h$ of degree $1$ as zero, except for 
\begin{equation}
\label{homotopy1}
h\left( \quad \scalebox{0.75}{$ \input{cobarhomotopy_def.tikz}$} \quad \right) =  -  \; \scalebox{0.75}{$\input{decompX1.tikz}$} \; 
\end{equation}
where $X_1$ decomposes as 
$$X_1 = \; \scalebox{0.75}{$\input{decompX1.tikz}$} \;$$
In order to show $h$ is a homotopy from $\Id$ to $0$, it suffices to compute that $d_{\iota}h + hd_{\iota} = \Id$, which does not differ from the usual calculation for operads.
%
\end{proof}
\begin{opm}\label{rmkacyclic}
The above proofs seem to fail respectively for $\B \bbb \smsquaresub{\pi} \bbb$ and $\ccc \smsquareconnsub{\iota}\Om \ccc$. 

For the complex $\B \bbb \smsquaresub{\pi} \bbb$, the map $h$ easily extends, but it then no longer constitutes a homotopy: the problem lies in equation \eqref{problemhomotopybar}. Namely if the bottom box is thin, the differential $d^0$ will not decompose it into a stacking with a non-thin box at the bottom. 

For the complex $\ccc \smsquareconnsub{\iota}\Om \ccc$, it is unclear how to define the homotopy from \eqref{homotopy1} as it does not restrict well. For example, when the upper level stacking $X_1$ decomposes in two levels of non-thin boxes, their image under $h$ is not part of the connected box composite.
\end{opm}

\section{Fundamental theorem of (inclined) box operads} \label{parfundamental}

Assembling the previous sections, our main goal of this section is to relate the homology of the twisted complex $\ccc \smsquaresub{\al} \bbb$ to the question whether the bar construction $\ccc \longrightarrow \B\bbb$ and cobar construction $\Om \ccc \longrightarrow \bbb$ are fibrant and cofibrant replacements respectively.

Our first result, in $\S \ref{parparpartial}$, is the \emph{partial fundamental theorem for box operads} (Theorem \ref{partialfundamental}): the connected twisted complex $\ccc \smsquareconnsub{\al} \bbb$ is acyclic if and only if the induced morphism $\ccc \longrightarrow \B \bbb$ is a quasi-isomorphism. The technical heart of the proof consists of two results: the connected twisted bar construction $\B \bbb \smsquareconnsub{\pi} \bbb$ is acyclic (Proposition \ref{acyclic}), and the \emph{box operadic comparison lemma} (Proposition \ref{comparisonlemma}). The latter generalises the operadic comparison lemma \cite[Lem. 6.7.1]{lodayvallette} to the connected twisted complex for box operads. 

Unfortunately, we are not able to show the fundamental theorem for the cobar construction: either we lack a proof showing that the connected twisted cobar complex $\ccc \smsquareconnsub{\iota} \Om \ccc$ is acyclic or an extension of the comparison lemma to the full twisted complex. For the latter, an important obstruction is that the natural weight grading defined in $\S \ref{parparparweight}$ is not efficient enough (see also Remark \ref{remweight}).

Our solution in order to obtain the fundamental theorem for the cobar construction is to look at an important subcategory of box (co)operads, that is, the subcategory of \emph{inclined} box operads. This property solely pertains to the underlying $\N^3$-dg module: a $\N^3$-dg module $M$ is \emph{inclined} if it does not contain elements of arity $(p,q,r)$ for which $p<r$. For box (co)operads with underlying inclined $\N^3$-dg module, the grading $p+q$ on arities proves to be sufficiently efficient in order to obtain the inclined box operadic comparison lemma (Proposition \ref{operadiclemmainclined}). As a result, we obtain the fundamental theorem for inclined box operads (Theorem \ref{thminclinedfundamental}): the twisted complex $\ccc \smsquaresub{\al} \bbb$ is acyclic if and only if the induced morphism $\Om \ccc \longrightarrow \bbb$ is a quasi-isomorphism.

Finally, having prepared all the necessary components, we develop in $\S \ref{parparkoszul}$ the theory of Koszul box operads: it applies to \emph{thin-quadratic} box operads $\bbb$ by defining a Koszul dual box cooperad $\bbb^{\antishriek}$, a twisting morphism $\ka: \bbb^{\antishriek} \longrightarrow \bbb$ and its \emph{Koszul complex} $\bbb^{\antishriek} \smsquaresub{\ka} \bbb$. If $\bbb$ is inclined, we obtain our main theorem (Theorem \ref{thmkoszul}): $\bbb$ is Koszul if and only if its Koszul complex is acyclic if and only if $\Om \bbb^{\antishriek} \longrightarrow \bbb$ exhibits $\Om \bbb^{\antishriek}$ as the minimal model of $\bbb$.  

\subsection{The partial fundamental theorem} \label{parparpartial}

\subsubsection{Weight grading}\label{parparparweight}

A \emph{$\N^3$-weight graded dg module}, or $\N^3$-wdg module for short, $M$ is a $\N^3$-dg module having an extra grading $\om \in \Z$ which is preserved by the differential. In other words, $M$ is the direct sum of $\N^3$-dg submodules $M^{(\om)}$ indexed by this weight. Note that we write $M^{(\om)}_l(p,q,r)$ for the $k$-module of weight $\om$, degree $l$ and arity $(p,q,r)$.

\begin{vb}
For a given $\N^3$-wdg module $M$, the (conilpotent) (co)free box (co)operads $\Tboxop(M)$ and $\Tboxcoop(M)$ have a natural induced weight: for $X$ a stacking of $n$ boxes labeled $x_1,\ldots,x_n \in M$ we set
\begin{equation}\label{weightformula1}
\om(X;\onth{x})) := \sum_{i=1}^k \om(x_i) + \begin{cases}
0  & \text{ if } X \text{ consists of thin boxes,} \\
 1 + \#\text{thin boxes} - n & \text{ otherwise.}
\end{cases}
\end{equation} 
\end{vb}
\begin{vb}
Given two $\N^3$-wdg modules $M$ and $N$, the composite $M \smsquare N$ and connected composite $M \smsquareconn N$ have a natural induced weight given by the formula \eqref{weightformula1}.
\end{vb}
A \emph{weight graded dg box (co)operad}, or wdg box (co)operad for short, is a dg box (co)operad with underlying $\N^3$-wdg module such that the (de)composition and (co)unit maps preserve the weight (with respect to the above defined weight grading on the box composite). Moreover, in the weight graded case, we assume the twisting morphism to respect the induced weight grading.

\begin{lemma}
For a twisting morphism $\al: \ccc \longrightarrow \bbb$ the differential $d_{\al}$ of the twisted composite $\ccc \smsquaresub{\al} \bbb$ preserves the weight.
\end{lemma}
\begin{proof}
The differential decomposes as $d_{\al} = d_{\sthin\inv \ccc \circ \bbb} + d_{\al}^{r}$. The first term preserves the weight as the differentials of $\ccc$ and $\bbb$ do. The second term preserves the weight as it involves the (de)composition and the twisting morphism, which all preserve the weight.
\end{proof}

A wdg box operad $\bbb$ is \emph{connected} if it decomposes weight-wise as 
$$\bbb = \K \eta \oplus \bbb^{(1)} \oplus \ldots$$
that is, $\bbb$ is nonnegatively weight graded and its zero weight is one dimensional concentrated in degree $0$. Analogously, we define \emph{connected} conilpotent wdg box cooperads.

The following are canonical examples of such connected structures.
\begin{vb}\label{weightfree}
Given a $\N^3$-dg module $E$, the (conilpotent) (co)free box (co)operads $\Tboxop(E)$ and $\Tboxcoop(E)$ are naturally weight-graded by setting $\om(x) = 1$ for $x\in E$ and $\om(\eta) = 0$. As a result, the weight of a stacking equals the number of thin boxes it contains (including the resulting box when composing) plus $1$. Thus, we have 
\begin{align*}
\Tboxop(E)^{(0)} &\cong k \Id \\
\Tboxop(E)^{(1)} &\supseteq E \\
\Tboxop(E)^{(2)} &\supseteq E \infsquare E
\end{align*}
\end{vb}
\begin{opm}\label{remweight}
Note the important difference for box operads with the natural weight grading on the (conilpotent) (co)free (co)operad: there exist stackings of more than one generator having weight $1$. For example, consider any stacking of non-thin boxes.
\end{opm}

\underline{From here on, following \cite{lodayvallette}, we assume the homological degree and weight grading to be nonnegative.}

\subsubsection{The box operadic comparison lemma}

\begin{lemma}\label{lemweight}
Let $M$ and $N$ be $\N^3$-wdg modules such that their weight zero component is concentrated in thin arities. 
For an element of weight $n$
$$(x; y_1 \otimes \ldots \otimes y_k) \in (M \smsquareconn N)^{(n)}$$
we have 
\begin{enumerate}
\item if $\om(x) >0$, then $\om(y_1),\ldots,\om(y_k) < n$, and
\item if $\om(x) = n$, then $\om(y_1)=\ldots=\om(y_k) = 0$.
\end{enumerate}
\end{lemma}
\begin{opm}
This lemma does not hold for the box composite.
\end{opm}
\begin{proof}
According to formula \eqref{weightformula1}, we have 
$$n = \om(x) + \sum_{i=1}^k \om(y_i) + 1 + \#\text{ thin boxes } - (1+k)$$
If $y_1,\ldots, y_k$ are part of thin arities, then $n = \om(x) + \sum_{i=1}^k \om(y_i)$,
and thus the result follows as in the classical case. 

Assume now that they are not. Hence, $x$ labels a thin box and thus
$$n = \om(x) + \sum_{i=1}^k \om(y_i) + 1 - k = \om(x) + \sum_{i=1}^k (\om(y_i) -1) + 1.$$
As $y_i$ does not label a thin box, $\om(y_i) \geq 1$ by assumption. The result then follows immediately.
\end{proof}
\begin{lemma}\label{filtration}
Let $\bbb$ be a connected wdg box operad and $\ccc$ a connected conilpotent wdg box cooperad, then for a fixed weight $n$ we have a bounded below and exhaustive filtration
$$F_s( ( \sthin\inv \ccc \smsquare \bbb )^{(n)} ) := \bigoplus_{d+m \leq s} ( \sthin \inv \ccc_d^{(m)}  \; \smsquare \;  \bbb )^{(n)}$$
with second page computed as
\begin{equation}\label{E2}
E^2_{st}(( \sthin\inv \ccc \smsquare \bbb )^{(n)} ) \cong ( \bigoplus_{m=0}^n H_{s-m}(s_{\thin}\inv \ccc^{(m)}) \smsquare \underbrace{H}_{t+m}(\bbb) )^{(n)}.
\end{equation}
Moreover, this result also holds for the connected twisted subcomplex.
\end{lemma}
\begin{proof}
Due to both the homological and weight being nonnegative, it is bounded below. Moreover, it is clearly exhaustive.

We posit that $(d_{s_{\thin}^{-1}\ccc} \smsquare \bbb)F_s \sub F_s, (s_{\thin}^{-1}\ccc \smsquare d_{\bbb})F_s \sub F_{s-1}$ and $d^r_\al F_s \sub F_{s-2}$. By definition, the first two inclusions are immediate. Let us show the third. Unfolding the definition of $d^r_\al$, it suffices to show that 
$$ (\sthin\inv \overccc \infsquare \sthin\inv \overccc)^{(m)} \sub \bigoplus_{m' \leq m -1} \sthin\inv \ccc^{(m')} \infsquare \sthin \inv \overccc$$
as a twisting morphism sends the coaugmentation $\eta$ to zero. This follows by the same reasoning as in Lemma \ref{lemweight} and $\ccc$ being connected. As $d^r_{\al}$ applies $\De_{(1)}^{\sthin\inv}$ of degree $-1$ to the bottom box, we obtain that $d^r_{\al}$ lowers the filtration by $2$. 
\end{proof}

\begin{lemma}\label{lemspectral}
Let $\mmm$ and $\mmm'$ be filtered complexes with exhaustive and bounded below filtrations, and $f:\mmm \longrightarrow \mmm'$ a quasi-isomorphism of filtered complexes. If $E^2_{st} f: E^2_{st} \mmm \longrightarrow  E^2_{st} \mmm'$ is an isomorphism for every $s,t$ such that $s\neq a$, then $E^2_{st} \Phi^{(n)}$ is an isomorphism for every $s,t$. 
\end{lemma}
\begin{proof}
This is extracted from the proof of \cite[Lem. 6.7.1]{lodayvallette}.
\end{proof}

\begin{prop}[Box operadic comparison lemma]\label{comparisonlemma}
Let $g: \bbb \longrightarrow \bbb'$ be a morphism of connected wdg box operads and $f: \ccc \longrightarrow \ccc'$ a morphism of connected wdg box cooperads. For $\al: \ccc \longrightarrow \bbb$ and $\al' : \ccc' \longrightarrow \bbb'$ twisting morphisms compatible with $(f,g)$, then
\begin{enumerate}
\item $f,g$ and $f \smsquareconn g: \ccc \smsquareconnsub{\al} \bbb \longrightarrow \ccc'  \smsquareconnsub{\al'} \bbb'$ satisfy the two-out-of-three property for quasi-isomorphisms.
\item if $f$ and $g$ are quasi-isomorphisms, then $f\smsquare g: \ccc \smsquaresub{\al} \bbb \longrightarrow \ccc'  \smsquaresub{\al'} \bbb'$ is a quasi-isomorphism.
\end{enumerate}
\end{prop}
\begin{proof}
The proof can be taken mutatis mutandis from \cite[Lemma 6.7.1]{lodayvallette}. We include the full proof for completeness, pinpointing where Lemma \ref{lemweight} is used.

For convenience, we write $\mmm := \ccc \smsquareconnsub{\al} \bbb, \mmm' := \ccc' \smsquareconnsub{\al'} \bbb'$ and $\Phi := f \smsquareconn g$, and add the superscript $(-)^{(n)}$ to denote its weight $n$ component.

Lemma \ref{filtration} shows that we have a convergent filtration on both $\mmm^{(n)}$ and $\mmm^{'(n)}$ such that
$$E^2\Phi^{(n)} = (H(f) \smsquareconn H(g))^{(n)}.$$

%
Assume $f$ and $g$ are quasi-isomorphisms, then $E^2\Phi^{(n)} = H(f) \smsquareconn H(f)$ is an isomorphism and thus, due to convergence, also $\Phi$. This is independent of working with the connected subcomplex or not.

%
Assume $\Phi$ is a quasi-isomorphism. If $g$ is a quasi-isomorphism, then we show that $f$ is a quasi-isomorphism as well. We proceed by induction on $n$ to show that $f^{(n)}$ is a quasi-isomorphism. For $n=0$, we simply have $f^{(0)}: k \longrightarrow k$ due to connectedness. Assume $n>0$. We show that $E^2_{st}\Phi^{(n)}$ is an isomorphism for every $t\neq -n$ and that $E^2_{s(-n)} \Phi^{(n)}= H_{s-n}(f^{(n)})$. By Lemma \ref{lemspectral}, we conclude that $H_{s-n}(f^{(n)})$ is an isomorphism.



Let us analyse $E^2_{st}(\Phi^{(n)})$ based on $t$. Remember from Lemma \ref{lemspectral} that 
$$E^2_{st}(( \sthin\inv \ccc \smsquare \bbb )^{(n)} ) \cong ( \bigoplus_{m=0}^n H_{s-m}(s_{\thin}\inv \ccc^{(m)}) \smsquare \underbrace{H}_{t+m}(\bbb) )^{(n)}.$$
\begin{itemize}
\item[$t<-n$] As the grading of $\bbb$ is nonnegative, we have $0 \leq t + m < -n + m$ and $0 \leq m \leq n$, thus $E^2_{st}(\mmm^{(n)}) = 0$. Hence, $E^2_{st}(\Phi^{(n)})=0$ is an isomorphism. 
\item[$t= -n$] Analogously, the above reasoning implies that we are reduced to the component $m = n$ and thus 
$$E^2_{s(-n)}(\mmm^{(n)}) = (H_{s-n}(\sthin \inv \ccc^{(n)}) \smsquareconn \underbrace{H}_{0}(\bbb))^{(n)} $$
Due to Lemma \ref{lemweight}, we can conclude $E^2_{s(-n)}(\mmm^{(n)}) \cong H_t(\sthin^{-1} \ccc^{(n)})$. As a result, $E^2_{-1t}(\Phi^{(n)}) = H_t(f^{(n)})$.
\item[$t>-n$] In this case, we have that $0 \leq m <n$ and thus by induction $E^2_{st} \Phi^{(n)}$ is an isomorphism.
\end{itemize}

Reversely, assume $f$ is a quasi-isomorphism instead of $g$, then we show that $g$ is a quasi-isomorphism. The proof is completely analogous save the analysis of the morphism $E^2 \Phi^{(n)}$. Now, we show that $E^{2}_{st}\Phi^{(n)}$ is an isomorphism for $s \neq -1$ and $E^2_{-1t} \Phi^{(n)} = H_{t}(g^{(n)})$. Let us analyse $E^{2}_{st}\Phi^{(n)}$ based on $s$.
\begin{itemize}
\item[$s<-1$] As the grading of $\sthin \inv \ccc$ is greater than or equal to $-1$, we have that $H_{s-m}(s^{-1}_{\thin}\ccc^{(m)}) = 0$ for $m\geq 0$. Indeed,  we have that $-1 \leq s-m < -1 -m$ and thus $m < 0$. Hence, $E^2_{st}(\mmm^{(n)}) = 0$ and thus $E^2_{st}(\Phi^{(n)})=0$ is an isomorphism. 
\item[$s= -1$] Analogously, the above reasoning implies that we are reduced to the component $m = 0$ and thus 
$$E^2_{-1t}(\mmm^{(n)}) = (H_{-1}(\sthin \inv \ccc^{(0)}) \smsquare \underbrace{H}_{t}(\bbb))^{(n)} \cong H_t(\bbb)^{(n)}$$
Hence, $E^2_{-1t}(\Phi^{(n)}) = H_t(g^{(n)})$.
\item[$s>-1$] In this case $0<m\leq n$, and thus, by Lemma \ref{lemweight}, we have that
$$H(f) \smsquareconn H(g): (H_{s-m}(\sthin^{-1} \ccc^{(m)}) \smsquareconn \underbrace{H}_{t+m}(\bbb))^{(n)} \longrightarrow (H_{s-m}(\sthin^{-1} \ccc^{'(m)}) \smsquareconn \underbrace{H}_{t+m}(\bbb'))^{(n)} $$
is composed of $H(g)$ and $H(f^{(n_i)})$ of weight $n_i < n$. Hence, $E^{2}_{st} \Phi^{(n)}$ is an isomorphism by induction.
\end{itemize}
\end{proof}

\begin{theorem}[Partial Fundamental Theorem]\label{partialfundamental}
Let $\bbb$ be a connected wdg box operad and $\ccc$ a conilpotent connected wdg box cooperad. Let $\al: \ccc \longrightarrow \bbb$ be a twisting morphism, then the following are equivalent
\begin{enumerate}
\item the connected twisted complex $\ccc \smsquareconnsub{\al} \bbb$ is acyclic,
\item the morphism $f_{\al}: \ccc \longrightarrow \B \bbb$ is a quasi-isomorphism.
\end{enumerate}
Moreover, if $g_{\al}: \Om \ccc \longrightarrow \bbb$ is a quasi-isomorphism, then $\ccc \smsquaresub{\al} \bbb$ is acyclic. 
\end{theorem}
\begin{proof}
This is an immediate consequence of combining Lemma \ref{acyclic} and Theorem \ref{comparisonlemma}.
%
%
\end{proof}
\begin{opm}
Note that a third equivalent statement is missing: as long as we do not know that $\ccc \circ_\iota \Om \ccc$ is acyclic in general, we cannot conclude that $\ccc \circ_\al \bbb$ is acyclic if and only if $g_{\al}:\Om \ccc \longrightarrow \bbb$ is a quasi-isomorphism.
\end{opm}

\subsection{Inclined box operads} \label{parparinclined}

\begin{mydef}
A $\N^3$-dg module $N$ is \emph{(left) inclined} if for $p<r$ and $q\in \N$ holds $N(p,q,r) =0$.

A dg box (co)operad $\bbb$ is \emph{inclined} if its underlying $\N^3$-dg module is inclined.
\end{mydef}
\begin{opm}
Note that we can equivalently work with \emph{right inclined} by reversing the roles $p$ and $r$.
\end{opm}

\begin{vb}
For an inclined $\N^3$-dg module $N$, the (conilpotent) (co)free box (co)operad is inclined.
\end{vb}

\begin{lemma}\label{aritylemma}
Let $M$ and $N$ be inclined $\N^3$-dg modules and consider an element
$$x = \; \scalebox{0.85}{$\input{operadiclemma3.tikz}$} \; \in M \smsquare N$$
then we have
\begin{enumerate}[label=(\Roman*)]
\item $p_{i-1}+q_i \leq p+q$ for $i=1,\ldots,k$, \label{operadicfcoperads1}
\item if $p_{i-1} + q_i = p+q$ for some $i\in \{1,\ldots,k\}$, then $x_0$ is thin, i.e. $p'=r'=0$. 
\end{enumerate}
\end{lemma}
\begin{proof}
We will use the following facts about the arities of elements in the composite product $M\smsquare N$
\begin{enumerate}
\item $0 \leq q_i$ and $0 \leq p_i$,\label{arities1}
\item $p = p' + p_0$, $r = r'+ p_k$ and $q = q_1 + \ldots + q_k$, \label{arities2}
\end{enumerate}
For inclined box (co)operads, we can moreover say for $i=1,\ldots,k$ that
\begin{enumerate}[resume]
\item $0 \leq p_i \leq p_{i-1}$ and $p'\geq r'$. \label{arities4}
\end{enumerate}
We start by showing property \ref{operadicfcoperads1}. Due to \eqref{arities1} and \eqref{arities2}, we have $q = \sum_{i=1}^n q_i \geq q_i$. Due to \eqref{arities4}, we have 
$$ p= p' + p_0 \geq p_1 \geq \ldots \geq p_k$$
Hence, $p+q \geq p_{i-1} +q_i$. 

Now, assume $p+q= p_{i-1}+q_i$ for some $i \in \{1,\ldots,k\}$. As we have $p+q \geq p'+ p_0 +q_i \geq p' + p_{i-1} +q_i$, we deduce $p' = 0$ and thus $r'=p'=0$ due to \eqref{arities4}. 
\end{proof}

This is sufficient to provide an efficient inductive argument to prove the following.

\begin{prop}[Inclined box operadic comparison lemma] \label{operadiclemmainclined}
Let $f:\ccc \longrightarrow \ccc'$ be a morphism of inclined connected conilpotent wdg box cooperads, $g:\bbb \longrightarrow \bbb'$ a morphism of inclined connected wdg box operads and $\al: \ccc \longrightarrow \bbb$ and $\al' \ccc' \longrightarrow \bbb'$ twisting morphisms compatible with $(f,g)$. If $f\smsquare g: \ccc \smsquaresub{\al} \bbb \longrightarrow \ccc' \smsquaresub{\al'} \bbb'$ and $f$ are quasi-isomorphisms, then so is $g$.
\end{prop}
\begin{proof}
We adapt the proof from Proposition \ref{comparisonlemma}: we solely have to provide a new argument where Lemma \ref{lemweight} is used.

First, we observe that on the thin parts, this reduces to the classical case for operads. Hence, we can assume to work in the non-thin part.

The proof goes by induction on $n$ to show that $f^{(n)}$ is a quasi-isomorphism. For our purposes, Lemma \ref{lemweight} is only used to show that 
\begin{equation}\label{Hfg}
E^2_{st}\Phi^{(n)} = H(f) \smsquareconn  H(g) : (H_{s-m}(\sthin^{-1} \ccc^{(m)}) \smsquareconn \underbrace{H}_{t+m}(\bbb))^{(n)} \longrightarrow (H_{s-m}(\sthin^{-1} \ccc^{'(m)}) \smsquareconn \underbrace{H}_{t+m}(\bbb'))^{(n)}
\end{equation}
is an isomorphism for $s>-1$ and $0 < m \leq n$. We replace the connected box composite by the box composite and proceed by additional induction on $p+q \geq 0$ to show that $g^{(n)}_{p,q,r}$ is a quasi-isomorphism. Note that the above constructions all restrict well according to arity. If $p+q= 0$, we are in the thin case which is already covered. Assume $p+q>0$.
Consider a general element 
$$ \scalebox{0.85}{$\input{operadiclemma2.tikz}$} \in (H_{s-m}(\sthin^{-1} \ccc^{(m)}) \smsquareconn \underbrace{H}_{t+m}(\bbb))^{(n)}$$
where each element $x_i$ has weight $n_i$. For $i>0$,  $x_i$ has arity $(p_{i-1},q_i,p_i)$ and $x_0$ has arity $(p',k,r')$. We then have
$$E^2_{st}(\Phi^{(n)})([x_0];[x_1],\ldots,[x_k]) = (H(f^{(n_0)})[x_0];H(g^{(n_1)})[x_1], \ldots , H(g^{(n_k)})[x_k])$$
As $H(f)$ is a quasi-isomorphism, we solely need to show that the $H(g_{p_{i-1},q_i,p_i}^{(n_i)})$ are as well. By Lemma \ref{aritylemma}, we know that either $p_{i-1}+q_i<p+q$ for each $i=1,\ldots,k$, or the bottom box is thin. In the latter case, this means that the weights $n_1,\ldots,n_k$ are strictly smaller than $n$ (see Lemma \ref{lemweight}). Hence, in both cases, by induction $H(g^{(n_i)}_{p_i,q_i,r_i})$ is an isomorphism. Thus, the map \eqref{Hfg} is an isomorphism. 
\end{proof}

\begin{theorem}[Inclined Fundamental Theorem]\label{thminclinedfundamental}
Let $\bbb$ be an inclined connected wdg box operad and $\ccc$ an inclined conilpotent connected wdg box cooperad. Let $\al: \ccc \longrightarrow \bbb$ be a twisting morphism, then the following are equivalent
\begin{enumerate}
\item the twisted complex $\ccc \smsquaresub{\al} \bbb$ is acyclic,
\item the morphism $g_{\al}:\Om \ccc \longrightarrow \bbb$ is a quasi-isomorphism.
\end{enumerate} 
\end{theorem}
\begin{proof}
Due to Theorem \ref{partialfundamental}, it suffices to only prove the direction if $\ccc \smsquaresub{\al} \bbb$ is acyclic. In that case, due to Proposition \ref{acyclic} the map $\ccc \smsquare g_{\al}$ is a quasi-isomorphism and thus, due to Proposition \ref{operadiclemmainclined}, the morphism $g_{\al}$ as well.
\end{proof}

\subsection{Koszul box operads} \label{parparkoszul}

\subsubsection{Thin quadratic box operads} 
Let $\bbb$ be a box operad. 
\begin{mydef}
A \emph{box operadic ideal} $\mathcal{I}$ of $\bbb$ is a $\N^3$-submodule of $\bbb$ such that for any composable family of operations $(x;x_1,\ldots,x_n)$ of $\bbb$, if one of them is in $\mathcal{I}$, then so is their composition $x \circ (x_1,\ldots,x_n)$. 
\end{mydef}
\begin{opm}
The conditions of a box operadic ideal $\mathcal{I} \subseteq \bbb$ are exactly such that the box operad structure transfers to the $\N^3$-module quotient $\faktor{\bbb}{\mathcal{I}}$.
\end{opm}
Observe that for any $\N^3$-submodule $R \subseteq \bbb$, we have a smallest box operadic ideal $(R)\subseteq \bbb$ generated by $R$. 

A \emph{presentation} of $\bbb$ consists of a $\N^3$-module $E$ and a $\N^3$-submodule $R \subseteq \Tboxop(E)$ such that $\bbb \cong \faktor{\Tboxop(E)}{(R)}$ as box operads, and we write $\bbb= \bbb(E,R)$. 

\begin{mydef}
A box operad $\bbb$ is \emph{thin-quadratic} if it admits a \emph{thin-quadratic} presentation $\bbb= \bbb(E,R)$, that is, we have $R \sub E \infsquare E$. 
\end{mydef}
Let $\bbb=\bbb(E,R)$ be thin-quadratic, then we associate a canonical \emph{Koszul dual graded box cooperad} denoted
$\bbb^{\antishriek}$.
\begin{mydef}\label{defkoszuldualboxcooperad}
Let $\bbb^{\antishriek}$ be the box subcooperad $\ccc(\sthin E,\sthin^2 R)$ of the conilpotent cofree box cooperad $\Tboxcoop(\sthin E)$ which is universal among the box subcooperads $\ccc$ of $\Tboxcoop(\sthin E)$ such that the composite
$$\ccc \hookrightarrow \Tboxcoop(\sthin E) \twoheadrightarrow \faktor{\Tboxcoop(\sthin E)^{(2)}}{\sthin^2 R}$$
is $0$. That is, there then exists a unique morphism of box cooperads $\ccc \longrightarrow \bbb^{\antishriek}$ making the diagram
$$ 
\begin{tikzcd}
\ccc \arrow[d] \arrow[r]      & \Tboxcoop(\sthin E) \\
\bbb^{\antishriek} \arrow[ru] &             
\end{tikzcd}$$
commute. 
\end{mydef}


 Consider the associated morphism of degree $0$
$$\ka: \bbb^{\antishriek} \twoheadrightarrow \sthin E \hookrightarrow \sthin \bbb,$$
then the following holds. 
\begin{lemma}
Both $\bbb$ and $\bbb^{\antishriek}$ are connected weight graded, and $\ka$ is a twisting morphism preserving the weight. Moreover, if $\bbb$ is inclined, then so is $\bbb^{\antishriek}$.
\end{lemma}
\begin{proof}
As $\bbb$ is thin-quadratic, the ideal $(R)$ is homogeneous with respect to the weight grading and thus $\bbb$ is naturally weight graded by $\Tboxop(E)$. Similarly for $\bbb^{\antishriek}$.

Since $\ka$ is $0$, except in weight $1$, the map $ \sum_{n \geq 2} P_n(\ka,\ldots,\ka)$ is only possibly non-zero in weight $2$. However, it then factors as 
$$(\bbb^{\antishriek})^{(2)} =  \ccc (\sthin E ,\sthin^2 R)^{(2)} \longrightarrow \sthin R^{(2)} \longrightarrow \Tboxop(\sthin E)^{(2)} \twoheadrightarrow \bbb$$
which is clearly $0$.	  
\end{proof}
\begin{mydef}
For a thin-quadratic box operad $\bbb$, we define its \emph{Koszul complex} as the twisted complex $\bbb^{\antishriek}\smsquaresub{\ka} \bbb$.
\end{mydef}

\subsubsection{Koszul box operads}

\begin{mydef}\label{defkoszul}
A thin-quadratic box operad is \emph{Koszul} if its Koszul complex is acyclic.
\end{mydef}

\begin{theorem}\label{thmkoszul}
Let $\bbb$ be a thin-quadratic inclined box operad, then the following are equivalent
 \begin{enumerate}
 \item $\bbb$ is \emph{Koszul},
 \item $\bbb^{\antishriek} \smsquaresub{\ka} \bbb$ is acyclic,
 \item $\Om \bbb^{\antishriek} \longrightarrow \bbb$ is a quasi-isomorphism.
\end{enumerate}
In particular, $\Om \bbb^{\antishriek}$ is the minimal model of $\bbb$.
\end{theorem}
\begin{proof}
Immediate from Definition \ref{defkoszul} and Theorem \ref{thminclinedfundamental}.
\end{proof}

%% file: KD_MorP.tex
\section{Application $1$: a minimal model for the box operad encoding morphisms of $\ppp$-algebras}
\label{section:applicationmorphismPalgebras}

As a first application, we apply Koszul duality to the box operad $\Mor(\ppp)$ for $\ppp$ a nonsymmetric Koszul operad. As it encodes a (collection of) morphism(s) of $\ppp$-algebras, we hereby answer an open problem by Markl \cite[Problem 9]{markl2002} in the case of nonsymmetric operads. Indeed, the main result of this section is that we establish a minimal model $\Mor(\ppp)_\infty$ of the box operad $\Mor(\ppp)$ as the cobar construction on a box cooperad $\Mor(\ppp)^{\antishriek}_{p \leq 1}$ (Theorem \ref{theoremminimalmodelmorP}) and show that it encodes $\infty$-morphisms of $\ppp_\infty$-algebras (Corollary \ref{corinfinitymorphism}). In contrast, in \cite{markl2002} Markl obtains the minimal model for the coloured operad encoding a morphism of $\ppp$-algebras by ad hoc techniques.

In \S \ref{subsection:MorP}, we define the box operad $\Mor(\ppp)$ and compute a $k$-module basis. In \S \ref{subsection:morkoszuldual} we compute its Koszul dual box cooperad $\Mor(\ppp)^{\antishriek}$ explicitly. This allows us to compute in \S \ref{parparnotkoszulmor} its Koszul complex $\Mor(\ppp)^{\antishriek} \smsquaresub{\ka} \Mor(\ppp)$ via a filtration showing that it is not acyclic (Theorem \ref{propMorisnotKoszul}). As a result, $\Mor(\ppp)$ is not Koszul in our sense (Definition \ref{defkoszul}).

%

In \S \ref{subsection:minimalmodelmorP}, we rectify the situation by realizing the minimal model $\Mor(\ppp)_\infty$ as the cobar construction on a box subcooperad $\Mor(\ppp)^{\antishriek}_{p\leq 1}$ of $\Mor(\ppp)^{\antishriek}$ (Theorem \ref{theoremminimalmodelmorP}). This is based on the following observation: the bottom box of the non-trivial classes of the homology of the Koszul complex of $\Mor(\ppp)$ are decorated by elements of arity $(p,q,p)$ for $p>1$. Luckily, the elements of arity $(0,q,0)$ and $(1,q,1)$ form our new box subcooperad $\Mor(\ppp)^{\antishriek}_{p\leq 1}$. Due to Koszul duality for box operads, it suffices to show that the twisted complex associated to the restricted Koszul morphism $\ka_{p\leq 1}: \Mor(\ppp)^{\antishriek}_{p\leq 1} \longrightarrow \Mor(\ppp)$ is acyclic (Proposition \ref{twistedcomplexmor}).
We further compute that $\Mor(\ppp)_\infty$ indeed encodes a (collection of) $\infty$-morphism(s) between $\ppp_\infty$-algebras (Corollary \ref{corinfinitymorphism}). 

In this subsection, let $\ppp$ be a nonsymmetric Koszul operad with quadratic presentation $\ppp =\faktor{\TOp(E)}{(R)}$ for some $\N$-module of generators $E$ and relations $R \subseteq \TOp(E)$. 

\subsection{The box operad $\Mor(\ppp)$}\label{subsection:MorP}
We identify $\ppp$ with its associated box operad $\Thin_!(\ppp)$ concentrated in thin arities, i.e. $\ppp(0,q,0)= \ppp(q)$ and $0$ elsewhere.
\begin{mydef}\label{Mor}
Let $\Mor(\ppp)$ be the box operad generated by $E$ and the element
$$  f \in \Mor(\ppp)(1,1,1) $$
 satisfying the relations $R$ and for every $e\in E(q)$ the relation
\begin{equation}\label{morPrel}
\quad \scalebox{0.85}{$\input{morP_funct1.tikz}$} \quad  =  \quad \scalebox{0.85}{$\input{morP_funct2.tikz}$}\quad.
\end{equation}
\end{mydef}
Observe that $\ppp$ has a thin-quadratic presentation with generators $E \oplus k f$ and relations $R$ and \eqref{morPrel}. In particular, the thin part of $\Mor(\ppp)$ corresponds to $\ppp$, i.e.  $\Mor(\ppp)(0,q,0) \cong \ppp(q)$.

The following shows that $\Mor(\ppp)$ is a box operadic version of the coloured operad $\PMor$ from \cite[Ex. 1]{markl2002}. Assume that $\uuu$ has finite factorisation (see $\S \ref{parparalgebras}$). 
\begin{prop}\label{proplaxalgebra}
A dg $\Mor(\ppp)$-algebra (resp. category) $\A$ over $\uuu$ corresponds to a collection of dg $\ppp$-algebras (resp. categories) $(\A(U))_{U\in \Ob(\uuu)}$ and a collection of morphisms of dg $\ppp$-algebras (resp. functors of dg $\ppp$-categories) $(\A^u= f^u: \A(U) \longrightarrow \A(V))_{u\in \uuu(U,V)}$.
\end{prop}
\begin{proof}
This is immediate from Definition \ref{defalgebra} and Definition \ref{Mor}.
\end{proof}
\begin{opm}
Again note that the finiteness condition on $\uuu$ can be omitted by working with coloured box operads.
\end{opm}
Let us compute the box operad $\Mor(\ppp)$ in every arity.
\begin{vb}\label{exMorPelement}
For $p\geq 0$ and $x \in \ppp(q)$, $\Mor(\ppp)(p,q,p)$ contains the following element
$$x_{p} := \quad \scalebox{0.85}{$\input{morphPelement.tikz}$} \quad $$
In particular, $\eta_p = f^p$ for $\eta \in \ppp(1)$ the unit.
\end{vb}

\begin{prop}\label{morbasis}
For $p,q,r \geq 0$, we have
$$ \Mor(\ppp)(p,q,r) \cong \begin{cases} \ppp(q) & \text{ if } p=r \\ 0 & \text{ otherwise } \end{cases} $$
In particular, $\Mor(\ppp)$ is inclined.
\end{prop}
\begin{proof}
It is clear from the generators and their relations, that $\Mor(\ppp)(p,q,r)=0$ for $p\neq r$. Let $x \in \Mor(p,q,p)$, then we show that there exists an element $x' \in \ppp(q)$ such that $x= x'_p$ (Example \ref{exMorPelement}). Due to relation \eqref{morPrel}, $x$ can be computed as
$$\scalebox{0.85}{$\input{morPoneDim.tikz}$}$$
for $x_1,\ldots,x_q \in \Mor(p,1,p)$ and $x'\in \ppp(q)$. Moreover, every $x_i$ is solely generated by instances of $f$ and thus $x_i = f^p$. 
\end{proof}

\begin{vb}
The box operad $\Mor(\Assoc)$ is generated by two elements $m$ and $f$ satisfying the associativity relation and \eqref{morPrel}. In particular, for every $p\geq 0$ and $q\geq 1$, the $k$-module $\Mor(\Assoc)(p,q,p)= \Assoc(q)$ is one dimensional.
\end{vb}

\subsection{The box cooperad $\Mor(\ppp)^{\antishriek}$} \label{subsection:morkoszuldual}

In this subsection, we realize the Koszul dual box cooperad $\Mor(\ppp)^{\antishriek}$, defined using an universal property, as an explicit box subcooperad $\Mor(\ppp^{\antishriek})$ of $\Tboxcoop(sE\oplus k f^c)$.
\begin{mydef}
Let $\Mor(\ppp)^{\antishriek}$ be the Koszul dual box cooperad (Definition \ref{defkoszuldualboxcooperad}) of $\Mor(\ppp)$, that is, it is cogenerated by elements 
$$ se^c \in \Mor(\ppp)^{\antishriek}(0,q,0) \text{ for } se \in sE(q) \text{ and } f^c \in \Mor(\ppp)^{\antishriek}(1,1,1)$$
respectively placed in degree $|e|+1$ and $0$, with corelations given by $s^2R$ and 
\begin{equation} \label{corelationMorP}
\quad \scalebox{1}{$\input{morP_funct1_co.tikz}$} \quad + \quad \scalebox{1}{$\input{morP_funct2_co.tikz}$}\;.
\end{equation}
\end{mydef}
\begin{opm}
Observe that corelation \eqref{corelationMorP} is the thin suspension of \eqref{morPrel} and it has the peculiar sign change remarked in Remark \ref{rmkadaptedsign}.
\end{opm}
Observe the following facts
\begin{itemize}
\item $\Mor(\ppp)^{\antishriek}$ is a box subcooperad of $\Tboxcoop(sE \oplus k f^c)$ and $\ppp^{\antishriek}$ is a subcooperad of $\TcOp(sE)$.
\item $\Tboxcoop(sE \oplus k f^c)$ is concentrated in arities $(p,q,p)$ for $p,q\geq 0$.
\end{itemize}
For $p,q \geq 0$, consider the following map of $k$-modules
$$\psi_{p,q}: \Tboxcoop(sE \oplus k f^c)(p,q,p) \longrightarrow \TcOp(sE)(q)$$
sending the labelled stacking $(S;y_1 \otimes \ldots \otimes y_n)$ to the labelled tree $(T; y_{i_1} \otimes \ldots \otimes y_{i_k})$ obtained by deleting all the non-thin boxes and their labels. Remark that in case $S$ consists of only non-thin boxes, we obtain the empty tree which represents the unit $\eta^c \in \TcOp(sE)$.
\begin{vb}
The following labelled stacking is send under $\psi_{2,3}$ to the following labelled tree (drawn using thin boxes)
$$\psi_{2,3}\left( \; \scalebox{0.85}{$\input{examplePhiMorP_1.tikz}$} \; \right) = \; \scalebox{0.85}{$\input{examplePhiMorP_2.tikz}$} \; .$$
\end{vb}
Let us now use these maps $\psi = (\psi_{p,q})_{p,q \geq 0}$ to construct an explicit box subcooperad of $\Tboxcoop(sE \oplus k f^c)$. The idea goes as follows: as every $y \in \ppp^{\antishriek}$ is a summation of trees labelled by generators in $sE$, we construct a new element $y_p$ as a summation of all the stackings obtained by adding instances of $f^c$ to the trees appearing in $y$. 

We obtain the following definition, using the desuspended sign $(-1)^{S^d}$ associated to a stacking $S$ from \cite[Construction A.1]{dinhvanhermanslowen2023}. It is characterized uniquely by the properties 
$$(-1)^{C^{p_0,\ldots,p_n;p,r}_{q_1,\ldots,q_n}} = (-1)^{(p_0+p_n)r + \sum_{i=1}^n (q_i-1)(n-i+p-r + p_0-p_{i-1})},(-1)^{(S\circ_i S')^{\dd}} = (-1)^{S^{\dd}+S^{'\dd}} \text{ and } (-1)^{(S^{\dd})^\si} = (-1)^{S^{\dd}+ \si}$$
 for a permutation $\si$. 
\begin{mydef}
For $y = \sum_{i} (-1)^\eps(T_i; x^i_1 \otimes \ldots \otimes x^i_n) \in \ppp^{\antishriek}(q)$ and $p\geq 0$, define the element
$$y_p := \sum_{(S; w_1 \otimes \ldots \otimes w_m) \in \psi_{p,q}^{-1}(T_i; x^i_1 \otimes \ldots \otimes x^i_n)} (-1)^{\eps+S^d} (S; w_1 \otimes \ldots \otimes w_m) \in \Tboxcoop(sE \oplus k f^c)$$
which has degree $|y_p|=|y|-1$ for $p>0$ and $|y_0|=|y|$. Define the graded $\N^3$-module
$$\Mor(P^{\antishriek})(p,q,p) := \{ y_p \; | \; y \in \ppp^{\antishriek}, p \geq 0 \}$$
\end{mydef}
\begin{opm}
For every $p,q\geq 0$, we have an isomorphism of graded $k$-modules $\Mor(\ppp^{\antishriek})(p,q,p) \cong \ppp^{\antishriek}(q)$.
\end{opm}

\begin{vb}\label{exMorPdual}
We have that $\eta^c_1 = f^c$ for the unit $\eta^c \in \ppp^{\antishriek}(1)$, and $se^c_1= \eqref{corelationMorP}$ the corresponding corelation from $\Mor(\ppp)^{\antishriek}$.
\end{vb}
\begin{vb}
For $\ppp= \Assoc$, consider the associativity corelation
$$ m^c_3 = \; \scalebox{0.75}{$\input{mc3_1.tikz}$} \; - \; \scalebox{0.75}{$\input{mc3_2.tikz}$} \; \in \Assoc^{\antishriek}(3)$$
where $m^c_2$ is the cogenerator of $\Assoc^{\antishriek}$, then $(m^c_3)_2$ is the summation of stackings
\begin{align*}
(m^c_3)_2 &= \; \scalebox{0.75}{$\input{mc32_1.tikz}$} \; - \; \scalebox{0.75}{$\input{mc32_2.tikz}$} \; + \; \scalebox{0.75}{$\input{mc32_3.tikz}$} \; - \; \scalebox{0.75}{$\input{mc32_4.tikz}$} \; + \; \scalebox{0.75}{$\input{mc32_5.tikz}$} \; - \; \scalebox{0.75}{$\input{mc32_6.tikz}$} \; + \; \scalebox{0.75}{$\input{mc32_7.tikz}$} \; - \; \scalebox{0.75}{$\input{mc32_8.tikz}$} \; \\
&+  \; \scalebox{0.75}{$\input{mc32_9.tikz}$} \; - \; \scalebox{0.75}{$\input{mc32_10.tikz}$} \;+  \; \scalebox{0.75}{$\input{mc32_11.tikz}$} \;-  \; \scalebox{0.75}{$\input{mc32_12.tikz}$} \;
\end{align*} 
\end{vb}

Remember that for (box) cooperadic structures we use the Sweedler notation 
$$\De(y) =  (y^{(0)} ; y^{(1)} \otimes \ldots \otimes y^{(n)})$$
which is shorthand for a summation of terms of the form
$$(y^0; y^1 \otimes \ldots \otimes y^n)$$
where $n$ depends on the arity of the bottom element $y^0$. Observe also that each term can have a different sign which we keep implicit.

\begin{lemma}
$\Mor(\ppp^{\antishriek})$ is a box subcooperad of $\Tboxcoop(sE \oplus f^c)$. In particular, for $y \in \ppp^{\antishriek}$ and $p\geq 0$, we have
\begin{equation} \label{decompMorPc}
\De(y_p) = \sum_{\substack{ p= p_0+p_1 }}  (-1)^{\sum_{i=1}^n (ar(y^{(i)})-1)(n-i)} \quad\scalebox{0.85}{$\input{decompmorP.tikz}$} 
\end{equation}
where $\Delta(y)= (y^{(0)} ; y^{(1)} \otimes \ldots \otimes y^{(n)})$ and $ar(y^{(i)})$ denotes its arity.
\end{lemma}
\begin{proof}
Without loss of generality, we can assume $y$ to be a single labelled tree $(T; se^c_1 \otimes \ldots \otimes se^c_k) \in \ppp^{\antishriek}\subseteq \TcOp(sE)$. Then each $y^{(i)}$ represents a subtree thereof such that grafting the trees $y^{(1)}, \ldots, y^{(n)}$ back on $y^{(0)}$ we obtain $y$ again.

The left hand side of the equation \eqref{decompMorPc} is indexed by the set $\psi^{-1}(p,q)(y)$ consisting of all labelled stackings $z=(S;w_1 \otimes \ldots \otimes w_m)$ obtained from the labelled tree $y$ by adding all possible instances of $f^c$ restricted by the arity $(p,q,p)$. Moreover, the decomposition of $z$ is simply the sum over all possible labelled substackings of $z$.

The right hand side of the equation \eqref{decompMorPc} is indexed by the partition $p=p_0+p_1$ and elements $(z^0;z^1 \otimes \ldots \otimes z^n)$ where every labelled stacking $z^i$ is obtained from the labelled tree represented by $y^{(i)}$ by adding instances of $f^c$, restricted by either $p_0$ or $p_1$.

It is thus clear that each term appears exactly once on each side of the equation. Indeed, there exists only an apparent ambiguity for boxes lablled by $f^c$ in $z$ as we have the choice whether it is part of a bottom stacking $z^0$ or a stacking on the top level. This choice is resolved by the choice of partition $p= p_0+p_1$.  

Disregarding the signs, we thus have shown
\begin{align*}
&\sum_{(S;w_1 \otimes \ldots \otimes w_m) \in \psi^{-1}_{p,q}(y)} (-1)^{S^d}\De(S;w_1 \otimes \ldots \otimes w_m) \\
&=  \sum_{ \substack{ p = p_0 + p_1 \\ (S^i; w^i_1 \otimes \ldots \otimes w^i_{m_i}) \in \psi^{-1}_{p_1,q_i}(y^{(i)}) \\
(S^0; w^0_1 \otimes \ldots \otimes w^0_{m_0}) \in \psi^{-1}_{p_0,n}(y^{(0)}) } } (-1)^{S^d_0 + \sum_{i=1}^n (ar(y^{(i)})-1)(n-i)+ S^d_i} \quad \scalebox{0.85}{$\input{decompmorP_z.tikz}$}
\end{align*}
from which follows the required equation \eqref{decompMorPc}. Now let us compute the signs: we have the following composition of stackings
$$S = C^{p_0,\ldots,p_0;p_1,p_1}_{q_1,\ldots,q_n} \circ (S_0,\ldots, S_n)$$
where $q_i = ar(S_i)$, and thus
$$(-1)^{S^d} = (-1)^{ \left( C^{p_0,\ldots,p_0;p_1,p_1}_{q_1,\ldots,q_n} \right)^d + \sum_{i=0}^n S_i^d  + \Delta} $$
where $(-1)^{\De}$ is the permutation sign appearing in the decomposition $\De$. The final sign then follows from 
$(-1)^{ \left( C^{p_0,\ldots,p_0;p_1,p_1}_{q_1,\ldots,q_n} \right)^d } = (-1)^{\sum_{i=1}^n (q_i-1)(n-i)}$.
\end{proof}

\begin{prop}\label{propmorPkoszuldual}
We have an isomorphism of graded box cooperads
$$\Mor(\ppp)^{\antishriek} \cong \Mor(\ppp^{\antishriek}).$$
\end{prop}
\begin{proof}
Example \ref{exMorPdual} and the fact that the thin part of $\Mor(\ppp^{\antishriek})$ is isomorphic to $\ppp^{\antishriek}$ shows that $\Mor(\ppp^{\antishriek})$ lies in the kernel of the projection  
$$\Tboxcoop(sE \oplus k f^c) \overset{\pi}{\twoheadrightarrow} \faktor{\Tboxcoop(sE \oplus k f^c)^{(2)}}{\left( R, \eqref{corelationMorP} \right)}.$$
Consider now a box subcooperad $\ccc \subseteq \Tboxcoop(sE \oplus k f^c)$ also lying in the kernel of $\pi$. First we observe that the thin part of $\ccc$, that is, $(\ccc(0,q,0))_{q\geq 0}$, thus lies in the kernel of the projection
$$ \TcOp(sE) \overset{\pi'}{\twoheadrightarrow} \faktor{\TcOp(sE)^{(2)}}{(R)}$$
and thus it is a subcooperad of $\ppp^{\antishriek}$. 

We now show that $\ccc \subseteq \Mor(\ppp^{\antishriek})$ for the non-thin arities using the corelators \eqref{corelationMorP}. For $z\in \ccc$ non-thin, we have that $z$ is a summation of labelled stackings $z'= (S;w_1 \otimes \ldots \otimes w_m)$. As $\ccc \subseteq ker(\pi)$, every thin-quadratic substacking of $z'$ is part of a corelation in $R$ or \eqref{corelationMorP}. Due to the latter, the labelled stacking obtained from $z'$ by commuting a cogenerator $se^c$ with instances of $f^c$ is also part of the summation $z$ with an appropriate sign. For example, if $z'=\;\scalebox{0.85}{$\input{morP_chain_1.tikz}$} \;$, then so are $\scalebox{0.85}{$\input{morP_chain_2.tikz}$} \; \text{ and } \; \scalebox{0.85}{$\input{morP_chain_3.tikz}$} $. Hence, every element of $\psi^{-1}_{p,q}\psi_{p,q}(z')$ is a term in $z$. As a result, $z'$ is part of an element $y= \sum_i (-1)^{\eps}(T^i; se^c_i \otimes \ldots se) \in \ppp^{\antishriek}$ such that every term of $y_p$ is a term in $z$. As a result, $z= y_p$ or a finite sum of such elements, whence $\ccc \subseteq \Mor(\ppp^{\antishriek})$.

%
%
%
\end{proof}

\begin{vb}
Similar to $\Assoc$, we thus obtain that the Koszul dual box cooperad $\Mor(\Assoc)^{\antishriek}$ is the linear dual of $\Mor(\Assoc)$. 
\end{vb}

\subsection{$\Mor(\ppp)$ is not Koszul} \label{parparnotkoszulmor}

We compute the Koszul complex $\Mor(\ppp)^{\antishriek}\smsquaresub{\ka}\Mor(\ppp)$ associated to the thin quadratic box operad $\Mor(\ppp)$. The twisting morphism $\ka: \sthin\inv \Mor(\ppp)^{\antishriek} \longrightarrow \Mor(\ppp)$ sends generators $s^{-1}sE\cong E$ to $E$ and $f^c$ to $f$ and other elements to zero.
 The twisted complex $\Mor(\ppp)^{\antishriek} \smsquaresub{\ka} \Mor(\ppp)$ is generated as $k$-module by the elements
\begin{equation} \label{basiselement}
 \scalebox{0.85}{$\input{kosmorP_element.tikz}$} 
 \end{equation}
for $p_0,p_1 \geq 0$ and $y\in \ppp^{\antishriek}$ and $x^1,\ldots,x^n \in \ppp$. We call them \emph{$\Mor(\ppp)$-basis elements}. 

We describe the differential $d_\ka = d_\ka^r$ on a basis element as above. For $p_0\neq 1$, we have
\begin{equation} \label{diffKos_1_mor}
d_\ka\left( \quad \scalebox{0.85}{$\input{kosmorP_element.tikz}$}  \quad \right) = \sum_{i=1}^{n-1} (-1)^{|y^{(0)}|+q-i+\delta_{p_00}} \quad \scalebox{0.85}{$\input{d_ka_decompmorP.tikz}$} .
\end{equation}
where $\delta_{p_00}$ is the Kronecker delta. For $p_0=1$, we have 
\begin{align}
d_\ka\left( \; \scalebox{0.85}{$\input{kosmorP_element_p0is1.tikz}$}  \; \right) &= \sum_{i=1}^{n-1} (-1)^{|y^{(0)}|+q-i} \quad \scalebox{0.85}{$\input{d_ka_decompmorP_p0is1.tikz}$}  \label{diffKos_2_mor}\\
&\quad -  \quad \scalebox{0.85}{$\input{d_ka_decompmorP_2.tikz}$}  \notag
\end{align}
%
%
%
%

\begin{prop}\label{propMorisnotKoszul}
The box operad $\Mor(\ppp)$ is not Koszul. In particular, the homology of its Koszul complex is given by
$$H(\Mor(\ppp)^{\antishriek}\smsquaresub{\ka}\Mor(\ppp)) \cong k \; \scalebox{0.85}{$\input{unitKosMorP.tikz}$} \; \oplus \bigoplus_{\substack{p_0 \geq  2 \\ p_1 \geq 0}} k \; \scalebox{0.85}{$\input{koscomplex_morP_homology.tikz}$} $$
\end{prop}
\begin{proof}
We compute the homology of the Koszul complex via the following converging descending filtration
$$
F_t(\Mor(\ppp)^{\antishriek}\smsquaresub{\ka}\Mor(\ppp)) := \bigoplus_{\substack{p_1 \geq t, p_0 \geq 0 \\ x^1,\ldots,x^n \in \ppp, y \in \ppp^{\antishriek} }} k  \; \scalebox{0.85}{$\input{kosmorP_element.tikz}$} \; .
$$ 
Its zeroth page
$$E_{st}^0(\Mor(\ppp)^{\antishriek}\smsquaresub{\ka}\Mor(\ppp)) = F_s(\Mor(\ppp)^{\antishriek}\smsquaresub{\ka}\Mor(\ppp))_{s+t} / F_{s+1}(\Mor(\ppp)^{\antishriek}\smsquaresub{\ka}\Mor(\ppp))_{s+t}$$
consists of those elements of degree $s+t$ with top level horizontal inputs equal to $s$. Hence, we have
$$E_{st}^0(\Mor(\ppp)^{\antishriek}\smsquaresub{\ka}\Mor(\ppp)) \cong \bigoplus_{\substack{p_0 \geq 0, |y|=s+t-\delta_{p_00} \\ x^1,\ldots,x^n \in \ppp, y \in \ppp^{\antishriek} } } k  \; \scalebox{0.85}{$\input{kosmorP_element_E0.tikz}$}$$
with differential $d^0$ given by the original differential $d_{\ka}$ with the last term in \eqref{diffKos_2_mor} omitted. Upon inspection of \eqref{diffKos_1_mor} and \eqref{diffKos_2_mor}, we observe that
$$E_{st}^0(\Mor(\ppp)^{\antishriek}\smsquaresub{\ka}\Mor(\ppp)) \cong \bigoplus_{p_0 \geq 0} s^{s+1-\delta_{p_0 0}}(\ppp^{\antishriek} \circ_{\ka'} \ppp)$$
as chain complexes where $\ppp^{\antishriek} \circ_{\ka'} \ppp$ is the Koszul complex associated to the operad $\ppp$. As $\ppp$ is Koszul, we have that its Koszul complex is concentrated in arity $1$ and degree $-1$. As a result, we can compute the first page: $E^1_{st}$ vanishes except for $t=-s-\delta_{p_00}$, in which case we obtain
$$E^1_{st}(\Mor(\ppp)^{\antishriek}\smsquaresub{\ka}\Mor(\ppp))= H_t(E^0_{s\bullet}(\Mor(\ppp)^{\antishriek}\smsquaresub{\ka}\Mor(\ppp))= \bigoplus_{p_0\geq 0} k \; \scalebox{0.85}{$\input{koscomplex_morP_E1.tikz}$}$$
The differential $d^1$ vanishes, except for
$$d^1\left(\; \scalebox{0.85}{$\input{koscomplex_morP_E1_1.tikz}$} \;\right) =- \; \scalebox{0.85}{$\input{koscomplex_morP_E1_2.tikz}$}$$
Hence, the spectral sequence converges at the second page, computing the homology of $ \Mor(\ppp)^{\antishriek}\smsquaresub{\ka}\Mor(\ppp) $ providing the required result.
\end{proof}
\begin{opm}
Observe that, as the homology of the Koszul complex of $\ppp$ is trivial, the homology of the Koszul complex of $\Mor(\ppp)$ solely depends on the added element $f$.
\end{opm}

\begin{cor}
The morphism of dg box operads
$$f_{\ka}:\Om \Mor(\ppp)^{\antishriek} \twoheadrightarrow \Mor(\ppp) : \sthin^{-1} y_{p} \longmapsto \begin{cases} x & \text{ if } p=0 \\
f & \text{ if } y=\eta^c \text{ and } p=1 \\
0 & \text{ otherwise } \end{cases}$$
is not a quasi-isomorphism.
\end{cor}
\begin{opm}
The above result is already clear upon inspection of the zeroth homology group: for instance for $\ppp= \Assoc$, we have
$$H_0( \Om \Mor(\Assoc)^{\antishriek}(2,1,2))  = k \quad \scalebox{0.85}{$\input{x11_mor.tikz}$} \quad \oplus k\quad \scalebox{0.85}{$\input{x1_x1_mor.tikz}$} \quad  \text{, while } \Mor(\ppp)(2,1,1) = k \quad \input{m11_mor.tikz} \quad.$$
\end{opm}

\subsection{The minimal model $\MorPinf$}
\label{subsection:minimalmodelmorP}
In this subsection, we provide a box operadic answer to Markl's open question \cite[Problem 9]{markl2002} for nonsymmetric operads.

Observe that the decomposition $\De$ of $\Mor(\ppp^{\antishriek})$ (see \eqref{decompMorPc}) solely involves lower arities. We thus have the following sequence of box subcooperads. 
\begin{mydef}
 
\begin{itemize}
\item For $t\geq 0$, let $\Mor(\ppp)_{p\leq t}^{\antishriek}$ be the wdg box subcooperad of $\Mor(\ppp)^{\antishriek}$ spanned by the elements $(y_p)_{p \in \ppp^{\antishriek},p\leq t}$, i.e.
$$\Mor(\ppp)_{p\leq t}^{\antishriek}(p,q,r) = \begin{cases} \Mor(\ppp^{\antishriek})(p,q,p)& \text{if } p=r \leq t \\
0 & \text{otherwise} \end{cases}$$
\item Let $\MorPinf := \Om \Mor(\ppp)_{p\leq1}^{\antishriek}$ be the cobar construction on $\Mor(\ppp)_{p\leq1}^{\antishriek}$.
\end{itemize}
\end{mydef}

The inclusion $\Mor(\ppp)_{p\leq1}^{\antishriek} \sub \Mor(\ppp)^{\antishriek}$ of box cooperads corresponds to an inclusion of dg box operads 
$i: \MorPinf  \hookrightarrow \Om \Mor(\ppp)^{\antishriek}$.
 Via the Rosetta Stone (Theorem \ref{thmrosetta}) the morphism of dg operads $f_{\ka_{p \leq 1}}:=f_{\ka}i: \MorPinf \longrightarrow \Mor(\ppp)$ corresponds to the twisting morphism $\ka_{p\leq1}: \sthin^{-1}\Mor(\ppp)^{\antishriek}_{p \leq 1} \longrightarrow \Mor(\ppp)$ which is the restriction of $\ka$ to $\Mor(\ppp)^{\antishriek}_{p \leq 1}$. 
 
\begin{prop}\label{twistedcomplexmor}
The twisted complex $\Mor(\ppp)_{p\leq1}^{\antishriek} \smsquaresub{\ka_{p\leq 1}}\Mor(\ppp)$ is an acyclic subcomplex of the Koszul complex $\Mor(\ppp)^{\antishriek}\smsquaresub{\ka}\Mor(\ppp)$. 
\end{prop}
\begin{proof}
This follows from the fact that the proof of Proposition \ref{propMorisnotKoszul} still holds mutatis mutandis for $\Mor(\ppp)^{\antishriek}_{p\leq 1}$ instead of $\Mor(\ppp)^{\antishriek}$.
\end{proof}

\begin{theorem}\label{theoremminimalmodelmorP}
The dg box operad $\MorPinf$ is the minimal model of $\Mor(\ppp)$ via $f_{\ka_{p\leq 1}}$. In particular, $\MorPinf$ has as graded $\N^3$-module of generators
\begin{align*}
s^{-1}\ppp^{\antishriek}(q) \subseteq \MorPinf(0,q,0) \\
\ppp^{\antishriek}(q) \subseteq  \MorPinf(1,q,1) 
\end{align*}
for $q\geq 0$. We write these generators respectively as $s^{-1}y_0$ and $y_1$ for any $y\in \ppp^{\antishriek}$. Then, the differential for $y \in \ppp^{\antishriek}$ can be written as
\begin{align*}
d(s^{-1}y_0) &= d_{\ppp_{\infty}}(s^{-1}y)= (-1)^{|y^{(0)}|+1} \;  \scalebox{0.85}{$\input{Pinf_diff.tikz}$} \\
d(y_1) &=  (-1)^{|y^{(0)}|} \;  \scalebox{0.85}{$\input{Pinf_diff_1.tikz}$}\quad  - \quad \scalebox{0.85}{$\input{Pinf_diff_2.tikz}$} 
\end{align*} 
where $\De_{(1)}(y)= y^{(0)} \otimes y^{(1)}$ is the infinitesimal decomposition and $\De(y) = (y^{(0)}; y^{(1)} \otimes \ldots \otimes y^{(n)})$ the full decomposition in $\ppp^{\antishriek}$. 
\end{theorem}
\begin{opm}
 Observe that the projection $\Mor(\ppp)^{\antishriek} \twoheadrightarrow \Mor(\ppp)^{\antishriek}_{p\leq 1}$ is in general not a morphism of box cooperads (look for $\ppp= \Assoc$ at $(m^c_{q})_2$ for example), yet the induced morphism $\Om \Mor(\ppp)^{\antishriek} \twoheadrightarrow \MorPinf$ is a morphism of dg box operads.
 \end{opm}
 
 \begin{cor}\label{corinfinitymorphism}
 A dg $\Mor(\ppp)_\infty$-algebra (resp. category) $\A$ over $\uuu$ corresponds to a collection of $\ppp_\infty$-algebras (resp. categories) $(\A(U))_{U\in \Ob(\uuu)}$ and a collection of $\infty$-morphisms (resp. functors) of $\ppp_\infty$-algebras (resp. categories) $(\A^u= f^u: \A(U) \longrightarrow \A(V))_{u\in \uuu(U,V)}$.
 \end{cor}
 \begin{proof}
 Immediate from Theorem \ref{theoremminimalmodelmorP} and the definition of $\infty$-morphism \cite[Thm. 10.2.6]{lodayvallette}.
 \end{proof}
 \begin{opm}
 The dg box operad $\Mor(\ppp)_\infty$ is the box operadic version of the minimal model of the coloured operad $\PMor$  \cite[Thm. 7]{markl2002}.
 \end{opm}
 
 \begin{vb}
 The minimal model $\Mor(\Assoc)_{\infty}$ of $\Mor(\Assoc)$ is generated by the elements 
\begin{align*} 
 \scalebox{0.85}{$\input{m0q.tikz}$} \;\in \Mor(\Assoc)_\infty(0,q,0) \text{ of degree } q-2 
 \end{align*}
 for $q\geq 2$, and 
 $$
 \scalebox{0.85}{$\input{m1q.tikz}$} \; \in \Mor(\Assoc)_\infty(1,q,1) \text{ of degree } q-1
$$
for $q \geq 1$, such that 
$$d(\; \scalebox{0.85}{$\input{m0q.tikz}$} \;) = \sum_{\substack{q+1= k+l \\
k,l\geq 2 \\
i=1,\ldots, k}} (-1)^{i+ l(k-i)} \; \scalebox{0.85}{$\input{dm01q.tikz}$} $$
and
$$d(\; \scalebox{0.85}{$\input{m1q.tikz}$} \;) =   \sum_{\substack{q+1= k+l \\
k,l\geq 2 \\
i=1,\ldots, k}} (-1)^{i-1+l(k-i)} \;  \scalebox{0.85}{$\input{dm1q.tikz}$}\quad  - (-1)^{\sum_{j=1}^k (l_j+1)(k-j)}  \sum_{\substack{q= l_1 + \ldots + l_k \\
l_1,\ldots,l_k \geq 1 \\
k\geq 2 }} \; \scalebox{0.85}{$\input{dm1q_2.tikz}$}.$$
Hence, $\Mor(\Assoc)_\infty$ encodes $\Ainf$-morphisms (resp. functors) between $\Ainf$-algebras (resp. categories).
 \end{vb}

%% file: KD_Lax.tex

\section{Application 2: a minimal model for the box operad encoding lax prestacks} \label{parlax}

We apply Koszul duality to the box operad $\Lax$ encoding lax prestacks. Our main goal in this section is to establish an explicit minimal model $\Laxinf$ for the box operad $\Lax$. Along the way, we compute the Koszul complex $\Lax^{\antishriek} \smsquaresub{\ka} \Lax$ of $\Lax$. As $\Mor(\Assoc)$ is a box suboperad of $\Lax$, this application encompasses and extends Application $1$ for the case $\ppp= \Assoc$. 

First, in $\S \ref{parparlax}$, we define the box operad $\Lax$ which encodes lax prestacks (Proposition \ref{proplaxalgebra}) via three generators $(m,f,c)$ and thin-quadratic relations (Definition \ref{lax}). We calculate its basis as $k$-module (Proposition \ref{laxbasis}) and in particular show that $\Lax$ is inclined. Further, we show in \S \ref{parparcooperad} that its Koszul dual box cooperad $\Lax^{\antishriek}$ is the linear dual of $\Lax$. 

In \S \ref{parparnotkoszul}, we perform a detailed calculation of the Koszul complex $\Lax^{\antishriek} \smsquaresub{\ka} \Lax$, vastly extending a similar computation for its subcomplex $\Mor(\Assoc)^{\antishriek}\smsquaresub{\ka}\Mor(\Assoc)$. Our main result proves it is not acyclic showing $\Lax$ is not Koszul (Theorem \ref{laxnotkoszul}). 
In order to compute the homology, we defined a new filtration on the complex (Lemma \ref{lemfiltration}). We show its first page $E^1$ decomposes as a sum of tensor products of subcomplexes isomorphic to certain arities of the classical Koszul complex of $\Assoc$ (Lemmas \ref{lemmaE1_1} and \ref{lemmaE1_2}). Note that each page $E^k$ has a non-trivial differential, as opposed to the spectral sequence for $\Mor(\Assoc)$ which already collapses at the second page. Finally, we show that there are non-trivial classes in $E^{\infty} = H(\Lax^{\antishriek}\smsquaresub{\ka}\Lax)$. In fact, we show that the bottom box of these non-trivial classes, which is decorated by an element of $\Lax^{\antishriek}$, always has arity $(p,q,r)$ with $r>1$. Hence, we obtain an important acyclic subcomplex $\Lax_{\leq 1}^{\antishriek} \smsquaresub{\ka} \Lax$ of the Koszul complex.

In \S \ref{parparminimal} we exploit this acyclic subcomplex to obtain an explicit minimal model $\Laxinf$ for $\Lax$. Indeed, by endowing the free graded box operad 
$$\Laxinf := \Tboxop(s_{\thin}^{-1}\Lax^{\antishriek}_{\leq r}) \sub \Om \Lax^{\antishriek}$$
with a differential through the projection $p: \Om \Lax^{\antishriek} \twoheadrightarrow \Laxinf$, we obtain a quasi-free dg box operad $\Laxinf$. Importantly, we show that $\Lax^{\antishriek}_{r\leq 1}\smsquaresub{\ka} \Laxinf$ is acyclic (Proposition \ref{laxacyclic}). Although $\Lax^{\antishriek}_{1\leq r}$ lacks a box cooperad structure, we are still able to mimick the inclined fundamental theorem and conclude that the induced morphism $f: \Laxinf \longrightarrow \Lax$ is a quasi-isomorphism, effectively exhibiting $\Laxinf$ as a minimal model for $\Lax$ (Theorem \ref{thmlaxinf}).
Following \cite{markl2004}, $\Laxinf$-algebras are called \emph{lax prestacks up to homotopy}.

\subsection{The box operad $\Lax$} \label{parparlax}

\begin{mydef}\label{lax}
Let $\Lax$ be the box operad generated by elements
$$ m\in \Lax(0,2,0) \qquad f \in \Lax(1,1,1) \qquad c \in \Lax(2,0,1)$$
 satisfying the relations
\begin{enumerate}
\itemsep1em
\item \label{laxrel1} 
$\quad  \scalebox{0.75}{$\input{lax_assoc1.tikz}$} \quad  =  \quad \scalebox{0.75}{$\input{lax_assoc2.tikz}$}\quad,$
\item \label{laxrel2}
$\quad \scalebox{0.75}{$\input{lax_funct1.tikz}$} \quad  =  \quad \scalebox{0.75}{$\input{lax_funct2.tikz}$}\quad,$
\item \label{laxrel3}
$\quad \scalebox{0.75}{$\input{lax_nat1.tikz}$} \quad  =  \quad \scalebox{0.75}{$\input{lax_nat2.tikz}$}\quad,$
\item \label{laxrel4} 
 $\quad \scalebox{0.75}{$\input{lax_cocycle1.tikz}$} \quad  =  \quad \scalebox{0.75}{$\input{lax_cocycle2.tikz}$}\quad,$
\end{enumerate}
\end{mydef}
\begin{opm}
Observe that $\Mor(\Assoc)$ is a box suboperad of $\Lax$ and in particular that $\Lax(0,q,0) \cong \Mor(\Assoc)(0,q,0) \cong \Assoc(q) \cong k$ a one dimensional $k$-module.
\end{opm}

 \emph{A dg lax prestack $\A$ over $\uuu$} is a lax functor over $\uuu$ taking values in the category of dg categories, that is, $\A$ is a lax functor $\uuu \longrightarrow \dgCat(k)$ \cite[Def. 4.8]{dinhvanhermanslowen2023}. Note that, by convention, we omit taking the opposite of $\uuu$ in this paper.
Assume $\uuu$ has finite factorisation (see $\S \ref{parparalgebras}$).
\begin{prop}\label{proplaxalgebra}
A dg $\Lax$-category over $\uuu$ corresponds to a dg lax prestack over $\uuu$.
\end{prop}
\begin{proof}
This is immediate from Definition \ref{defalgebra}, Definition \ref{lax} and \cite[Def. 4.8]{dinhvanhermanslowen2023}.
\end{proof}
\begin{opm}
Again note that the finiteness condition on $\uuu$ can be omitted by working with coloured box operads.
\end{opm}

\subsubsection{The unbiased definition of $\Lax$}

\begin{mydef}\label{defpart}
Let $\Part$ be the monoidal category of partitions, that is,
\begin{itemize}

\item $\Part$ has a monoid of objects $(\N,+,0)$,
\item $\Part(p,r)$ is the set of partitions of $\{1,\ldots,p\} $ in $r$ pairwise-disjoint subsets, i.e. $r$-tuples $(\fromto{t}{r})$ such that $\sum_{i=1}^r t_i = p$ and $t_i >0$. By default, we set $\Part(0,0) := \{ (0) \}$.
\item The composition of $P = (\fromto{t}{r}) \in \Part(p,r)$ and $P' = (\fromto{t'}{s})\in \Part(r,s)$ is given by composing the corresponding partitions
$$P \circ P' = \left( \sum_{i=1}^{t_1'} t_i, \ldots , \sum_{i= \sum_{j<s}t'_j + 1}^{ \sum_{j\leq s} t'_j} t_i \right)$$
\item The monoidal product of $P = (\fromto{t}{r}) \in \Part(p,r)$ and $P' = (\fromto{t'}{s})\in \Part(p',r')$ is given by concatenation.
\end{itemize}  
\end{mydef}
\begin{opm}
Note that $\Part(p,1)$ and $\Part(p,p)$ consist of a single partition, and $\Part(p,r)$ is empty if $p<r$. Note also that $\Part$ is the monoidal category generated by the (nonsymmetric) operad $\Assoc$.
\end{opm}

\begin{mydef}\label{laxstacking}
A \emph{$\Lax$-stacking} is a stacking $S$ consisting of boxes of arity $(0,2,0), (1,1,1)$ and $(2,0,1)$. 
\end{mydef}

\begin{constr}
Let $S$ be a $\Lax$-stacking, then we associate a partition $P_S$ as follows:
\begin{enumerate}
\item if $S$ consists of solely thin stackings, we set $P_S = (0)$,
\item otherwise, consider the decomposition of its horizontal composite graph $\Hh_S= (G_1,\ldots,G_r)$ into its connected components. Observe that each connected component $G_i$ has a single output. Let $t_i$ be the number of inputs of $G_i$, then define a partition $P_S :=(t_1,\ldots,t_r) \in \Part(p,r)$ by counting from bottom to top (see Example \ref{vbmpq}). 
\end{enumerate}
As the relations in $\Lax$ preserve the partition associated to a $\Lax$-stacking, every element $x\in \Lax$ has an associated partition $P_x$. 
\end{constr}

\begin{vb}\label{vbmpq}
Let $P=(t_1,\ldots,t_r) \in \Part(p,r)$ be a partition for $r>0$, then we define the elements
\begin{equation}\label{eqmpq}
m_{P,q} := \quad \scalebox{0.85}{$ \input{maximalElement.tikz}$} \quad, \quad \eps^{p,p} := \quad \input{eps_p.tikz} \quad \text{ and } \quad \eps^{i,p} :=  \quad \input{eps_i.tikz} 
\end{equation}
where $1 \leq i < p$. Note that we set $m_q:= m_{(0),q}$. These elements respectively live in $\Lax(p,q,r), \Lax(p,1,p)$ and $\Lax(p,0,p-1)$ with associated stackings $P, (1,\ldots,1)$ and $(\underbrace{1,\ldots,1}_{(i-1)-\text{times}},2,1\ldots,1)$. 
\end{vb}

\begin{prop}\label{laxbasis}
For $p,q,r \in \N$, we have
$$ \Lax(p,q,r) = \bigoplus_{P \in \Part(p,r)} k m_{P,q} $$
if $p+q-r \geq 1$, and $0$ otherwise.
\end{prop}
\begin{proof}
It suffices to show this for the non-thin part. Let $x \in \Lax(p,q,r)$, then we prove that $x = m_{P_x,q}$.
Due to relation \eqref{laxrel2}, $x$ can be computed in $\Lax$ as 
$$\input{laxoneDim.tikz}$$
for $n= p+q-r$ and $x_1,\ldots,x_n \in \Lax$ consisting of compositions of $f$ and $c$, whence $x_i = \eps^{a_i,b_i}$ for some $1 \leq a_i \leq b_i$. We show that if $a \leq c$ and $a<b$, then  
$$ m \circ ( \eps^{a,b} \otimes \eps^{c,d}) = m \circ (\eps^{c+1, d+1} \otimes \eps^{a,b-1} ).$$
Indeed, we have the following three cases
\begin{alignat}{3}
&\qquad\quad  \underline{a=c} &&\qquad\quad  \underline{a<c <d} && \qquad\quad \underline{c =d} \notag\\
& && && \notag \\
&\input{coherence6.tikz} \quad = \quad \input{coherence5.tikz}\;,\qquad &&\quad \input{coherence4.tikz} \quad = \quad \input{coherence3.tikz}\;, \qquad &&\quad\input{coherence2.tikz} \quad = \quad \input{coherence1.tikz}\;. \label{towers}
\end{alignat}
As a result, we can rewrite 
$$ x= \input{coherence7.tikz}$$
where $p > c_1 > \ldots > c_{p-r} > 0$. Hence, we observe that $t_i-1$ is equal to the number of $c_j=p-j+i+1$. As a result, $x$ is exactly of the form as defined in Example \ref{vbmpq}, thus $x = m_{P_x,q}$.
\end{proof}

\begin{opm}
Observe that $\Lax(p,q,p)= \Mor(\Assoc)(p,q,p)$ where the element $m_{(1,\ldots,1),q}$ with unique partition $(1,\ldots,1)\in \Part(p,p)$ equals $(m_q)_p$ from $\Mor(\Assoc)$.
\end{opm}

\subsection{The box cooperad $\Lax^{\antishriek}$} \label{parparcooperad}

\begin{mydef}
Let $\Lax^{\antishriek}$ be the Koszul dual box cooperad of $\Lax$, that is, it is cogenerated by elements 
$$ m^c \in \Lax^{\antishriek}(0,2,0), f^c \in \Lax^{\antishriek}(1,1,1) \text{ and } c^c\in \Lax^{\antishriek}(2,0,1)$$
respectively of degree $1,0$ and $0$, with corelations given by the thin suspension of the relations of $\Lax$.
\end{mydef}

\subsubsection{The unbiased definition of $\Lax^{\antishriek}$}
For $P\in \Part(p,r)$ and $q\in \N$ such that $p+q-r \geq 1$, let $\Laxboxop(P,q)$ be the set of $\Lax$-stackings $S$ such that $P_S =P$ and the resulting box is of arity $(p,q,r)$. 

Let $(-1)^{S^{\dd}}$ be the desuspended sign associated to a stacking $S$ from \cite[Construction A.1]{dinhvanhermanslowen2023}. It is characterized uniquely by the properties $(-1)^{C^{p_0,\ldots,p_n;p,r}_{q_1,\ldots,q_n}} = (-1)^{(p_0+p_n)r + \sum_{i=1}^n (q_i-1)(n-i+p-r + p_0-p_{i-1})},(-1)^{(S\circ_i S')^{\dd}} = (-1)^{S^{\dd}+S^{'\dd}}$ and $(-1)^{(S^{\dd})^\si} = (-1)^{S^{\dd}+ \si}$ for a permutation $\si$. 

Define the elements
$$m^c_{P,q} := \sum_{S\in \Laxboxop(P,q)} (-1)^{S^{\dd}} \; S \quad \in \Tboxop(m^c,f^c,c^c)(p,q,r)$$
which have degree $q-1 + p-r$.

\begin{vb}\label{exLaxc}
We easily compute that $m_{(0),3}^c, m_{(1),2}^c, m_{(2),1}^c$ and $m_{(3),0}^c$ correspond to the thin suspension of the relations $\eqref{laxrel1}, \eqref{laxrel2}, \eqref{laxrel3}$ and $\eqref{laxrel4}$ respectively. Note that, due to the additional sign from the thin infinitesimal decomposition, this means 
$$m_{(1),2}^c = \quad \scalebox{0.75}{$\input{lax_funct1.tikz}$} \quad + \quad \scalebox{0.75}{$\input{lax_funct2.tikz}$}\;,$$
see also Remark \ref{rmkadaptedsign}.
\end{vb}

\begin{prop}
For $p,q,r\in \N$ such that $p+q-r\geq 1$, holds
$$\Lax^{\antishriek}(p,q,r) = \bigoplus_{P \in \Part(p,r)} m^c_{P,q}$$
and $0$ otherwise. In particular, we have
\begin{equation} \label{decompLaxc}
\De(m_{P,q}^c) = \sum_{ \substack{ \sum_{i=1}^n q_i = q }} \;\sum_{\substack{ P_0 \in \Part(p',r') \\ P_i\in \Part(p_{i-1},p_{i}) \\  P = P_0 \otimes ( P_1 \circ \ldots \circ P_n ) }}  (-1)^{\left( C_{\nth{q}}^{p_0,\ldots,p_n;p',r'} \right)^{\dd}} \quad\scalebox{0.85}{$\input{decomplax.tikz}$} 
\end{equation}
\end{prop}
\begin{proof}
We first show the equality, which proves that $(m_{P,q}^c)_{P,q}$ defines a box subcooperad $\ccc$ of $\Tboxcoop(m^c,f^,c^c)$. 

We compute
\begin{align*}
 \De( m^c_{P,q}) &= \sum_{S \in \Laxboxop(P,q)} \; \sum_{ S= S_0 \circ (S_1,\ldots,S_n)} (-1)^{S^{\dd} + \De } \quad \scalebox{0.85}{$\input{Lax_koszul_dual_1.tikz}$} \\
 &=  \sum_{ \substack{ \sum_{i=1}^n q_i = q }} \;\sum_{\substack{ P_0 \in \Part(p',r') \\ P_i\in \Part(p_{i-1},p_{i}) \\  P = ( P_1 \circ \ldots \circ P_n )\otimes P_0 }} (-1)^{S^{\dd} + \sum_{i=0}^n (S_i)^{\dd}+ \De} \quad \scalebox{0.85}{$\input{decomplax.tikz}$}
\end{align*}
as we have that
$$\Big \{ (P,P_{S_0}, \ldots, P_{S_n}) : S \in \Laxboxop(P,q), S_i \in \Laxboxop, S = S_0 \circ (\nth{S}) \Big \} = \Big \{ (P,P_0,\ldots,P_n) : P = P_0 \otimes (\nth{P})  \Big \}$$
Next, we compute the signs. First, as $S = S_0 \circ (\nth{S})$, we have that 
$$S = (C_{\nth{q}}^{p_0,\ldots,p_n;p',r'} \circ_{n+1} S_0 \circ_n \ldots \circ_1 S_n)^\si$$
for an appropriate permutation $\si$. Note that it is the same permutation used in $\De$ and thus determines the sign $(-1)^{\De}$. Thus, we obtain
$$ (-1)^{S^{\dd}} = (-1)^{ \left( C_{\nth{q}}^{p_0,\ldots,p_n;p',r'} \right)^{\dd}  + \sum_{i=0}^n(S_{i})^{\dd} + \De} $$
This proves the equation.

Now, as shown in Example \ref{exLaxc}, we have that the suboperad $\ccc$ spanned by $m^c_{P,q}$ lies in the kernel of the projection  
$$\Tboxop(m^c,f^c,c^c) \twoheadrightarrow \Tboxop(m^c,f^c,c^c)^{(2)} / \left( (\ref{laxrel1}), (\ref{laxrel2}), (\ref{laxrel3}), (\ref{laxrel4}) \right).$$
 Moreover, using the above equation \eqref{decompLaxc} we easily show in a completely analogous way as Proposition \ref{propmorPkoszuldual} that $\ccc$ is the largest box subcooperad in this kernel and thus coincides with $\Lax^{\antishriek}$.
\end{proof}
\begin{opm}
Note in particular that $\Mor(\Assoc)^{\antishriek}$ is a box subcooperad of $\Lax^{\antishriek}$ and that $\Lax^{\antishriek}(0,q,0) = \Mor(\Assoc)^{\antishriek}(0,q,0) \cong \Assoc^{\antishriek}(q)$. We thus identify $m_q^c = m^c_{(0),q}$.
\end{opm}

\subsection{$\Lax$ is not Koszul} \label{parparnotkoszul}

\subsubsection{The Koszul complex $\Lax \smsquaresub{\ka} \Lax^{\antishriek}$ }
The twisting morphism $\ka: \sthin\inv \Lax^{\antishriek} \longrightarrow \Lax$ sends $(s\inv m^c,f^c,c^c)$ to $(m,f,c)$ and sends other elements to zero.
 The twisted complex $\Lax \smsquaresub{\ka} \Lax^{\antishriek}$ is generated as $k$-module by the elements
\begin{equation} \label{basiselement}
 \scalebox{0.85}{$\input{koslax_element.tikz}$} 
 \end{equation}
for $\nth{q} \in \N$ and $P_i \in \Part(p_{i-1},p_{i})$ and $P_0 \in \Part(p',r')$. We call them \emph{basis elements}. 

We describe the differential $d_\ka = d_\ka^r$ on a basis element as above. For $r'\neq 1$, we have
\begin{equation} \label{diffKos_1}
d_\ka\left( \quad \scalebox{0.85}{$\input{koslax_element.tikz}$}  \quad \right) = \sum_{i=1}^{n-1} (-1)^{i-1+\Thin(P_0)} \quad \scalebox{0.85}{$\input{d_ka_decomplax.tikz}$} .
\end{equation}
where $(-1)^{ \Thin(P_0)}= -1$ if $P_0= (0)$, and $1$ otherwise. For $r'=1$, we have 
\begin{align}
d_\ka\left( \quad \scalebox{0.85}{$\input{koslax_element.tikz}$}  \quad \right) &= \sum_{i=1}^{n-1} (-1)^{i-1} \quad \scalebox{0.85}{$\input{d_ka_decomplax.tikz}$}  \label{diffKos_2}\\
&\quad +  \sum_{\substack{ \be \in \Ss_{p'-1,n} \\ \underline{a} \in \Path(p') }} (-1)^{\be + \underline{a} +n(p'-1)+p'} \quad \input{d_ka_decomplax_2.tikz} \notag
\end{align}
where we define the auxiliary element $\be(\underline{a},\underline{P},\underline{q})$ as follows. Let us first give an example to sharpen the intuition.
\begin{vb}

\end{vb}
We now define $\be(\underline{a},\underline{P},\underline{q})$.
\begin{itemize}
\item define a set of \emph{paths} as
$$\Path(p') := \{ (\eps^{i_1,p'}, \ldots, \eps^{i_{p'-1},2}) : 1 \leq i_j < p-j \}$$
Note in particular that $i_{p'-1} = 1$ always holds, and thus $\eps^{i_{p'-1},2} = c$. Given such a path $\underline{a}=(\fromto{a}{p'-1})$ we define its sign
$$(-1)^{\underline{a}} = \prod_{i=1}^{p'-1}(-1)^{a_{i}} \text{ where } (-1)^{\eps^{j,t}}= (-1)^{j}.$$
In particular, $(-1)^{c} = -1$.
\item Let $\be \in \Ss_{p'-1,n}$ be a $(p'-1,n)$-shuffle, then given a path $\underline{a} = (\eps^{i_1,p'}, \ldots, \eps^{i_{p'-1},2})\in \Path(p')$, a sequence of natural numbers $\underline{q}=(\fromto{q}{n})$ and a sequence of composable partitions $\underline{P}= (P_1,\ldots,P_n)$, we define a sequence of elements $\be(\underline{a}, \underline{P},\underline{q}) = (m_{Q_1,q'_1},\ldots,m_{Q_{p'-1+n},q'_{p'-1+n}})$ as follows
$$ m_{Q_{\be(j)},q'_{\be(j)}} = \begin{cases} \eps^{i_j,p'-s_j+1+p_s} & \text{ if } 1 \leq j \leq p'-1 \\ m_{ P_j^{t_j},q_j} & \text{ if } p-1 < j \leq p'-1+n \end{cases}$$
where
\begin{align*}
s_j&= |\{ p'-1 < j' \leq p'-1+n \; | \; \be(j') < \be(j) \} |-1,\\
P_j^{t_j} &= \underbrace{(1,\ldots,1)}_{t_j} \otimes P_j\\
t_j &= p - | \{ 1 \leq j' \leq p-1 \;|\; \be(j') < \be(j) \} |
\end{align*}  
The result is a `path' from $p$ horizontal inputs (left) to $r$ horizontal outputs (right) which shuffles the path $\underline{a}$ with the top level elements $m_{P_i,q_i}$.
\item We write $\be(\underline{a},\underline{q})$ for $\be(\underline{a},\underline{P},\underline{q})$ where $\underline{P} = ((0),\ldots,(0))$. Analogously, denote $\be^c(\underline{a},\underline{q})$ by the corresponding tuple where instances of $(m,f,c)$ are replaced by $(m^c,f^c,c^c)$ respectively.
\end{itemize}
We show that the correct sign for this last set of terms is indeed $(-1)^{\be + a+1}$. We have that 
$$\De_{(1)}^{s\inv} ( m^c_{(p'),n} )  = - \sum_{\substack{ \be \in \Ss_{p'-1,n} \\ \underline{a} \in \Path(p') } }(-1)^{S^{\dd}} \quad \scalebox{0.85}{$\input{betabox.tikz}$} $$
 where $S$ denotes the underlying thin infinitesimal stacking.
 \begin{lemma}
 The sign $(-1)^{S^{\dd}}$ equals $(-1)^{\be + \underline{a}+1+n(p'-1)+p'}$.
 \end{lemma} 
\begin{proof} 
It is easy to compute that the desuspended sign of the stacking forming $\eps^{t,t}$ and $\eps^{i,t}$ is respectively $1$ and $(-1)^{i-1}$. Hence, 
$$(-1)^{ S^{\dd}} = (-1)^{(C^{p',\ldots,p',p'-1,\ldots,p'-1,\ldots,1,\ldots,1;0,0}_{1,\ldots,1,0,\ldots,0,\ldots,0,1,\ldots,1})^{\dd}} (-1)^{\sum_{j=1}^{p'-1}(i_j-1)}$$
where the place of the $0$'s is determined by $\be$. We compute the first factor as
$$  (-1)^{(C^{p',\ldots,p',p'-1,\ldots,p'-1,\ldots,1,\ldots,1;0,0}_{1,\ldots,1,0,\ldots,0,\ldots,0,1,\ldots,1})^{\dd}}  = (-1)^{ \sum_{j=1}^{p'-1} (n+p'-1-\be(j) + p'-p_{\be(j)-1})} = (-1)^{ n(p'-1)+\be}$$
as $p_{\be(j)-1} = p'-j+1$ and $(-1)^\be = (-1)^{  \sum_{j=1}^{p'-1}(\be(j)-j)}$. Hence, the result follows.
\end{proof}
Remark the similarities in notation and names with \cite[\S 3.3]{DVL} \cite[\S 3.1]{dinhvanhermanslowen2023}, which inspired this explicit description. 

\subsubsection{The filtration}

\begin{mydef}\label{defstackingnumber}
Let $P=(t_1,\ldots,t_r)\in \Part(p,r)$ and $q\in \N$ such that $p+q-r\geq 1$, then we define the \emph{stacking number of $(P,q)$} as
\begin{equation}\label{eqstackingnumber}
\SN(P,q) := \text{ maximum number of instances of } f \text{ and } c \text{ needed to compute } m_{P,q} \text{ from the generators of } \Lax.
\end{equation}
\end{mydef}
\begin{opm}
The stacking number can also be given by the following closed formula
\begin{equation}\label{formulasn}
\SN(P,q) = pq + \sum_{t_i >1} (t_i-1)( \frac{t_i}{2} + \sum_{\substack{ j<i \\ t_j > 1} } t_j ) + \sum_{t_i = 1} \sum_{j>i}(t_j-1)
\end{equation}
obtained by counting the occurrences of $f$ and $c$ in the stacking from Example \ref{vbmpq}.
\end{opm}

\begin{lemma}\label{lemfiltration}
The subspaces
$$
F_t(\Lax^{\antishriek}\smsquaresub{\ka}\Lax) := k \left\{ \; \input{koslax_element.tikz} \; | \; \sum_{i=1}^n \SN(P_i,q_i) \geq t \right\}
$$
define a convergent decreasing filtration computing the homology of $\Lax^{\antishriek}\smsquaresub{\ka}\Lax$. Moreover, it splits according to arity.
\end{lemma}
\begin{proof}
We verify that the differential $d_\ka$ preserves the filtration. In this respect, we first note that 
$$\SN(m_{P\circ P',q+q'}) \geq \SN(m_{P,q}) + \SN(m_{P,q}) $$
due to the maximality of the stacking number and the fact that 
$$ \input{XandX'.tikz} $$
computes $m_{P\circ P',q+q'}$. Secondly, it is clear from its definition that $\be(\underline{a},\underline{P},\underline{q})$ increases the sum $\sum_{i=1}^n \SN(P_i,q_i)$ by at least $1$. Hence, $d_{\ka}F_t \sub F_t$.

Further, we observe that
$$F_0 \Lax^{\antishriek}\smsquaresub{\ka}\Lax = \Lax^{\antishriek}\smsquaresub{\ka}\Lax \text{ and } \bigcap_{t \geq 0} F_t \Lax^{\antishriek}\smsquaresub{\ka}\Lax = 0$$
as we have that $0 \leq \SN(P,q) \leq pq + (r-1)(p-r) + \frac{(p-r)(p-r+1)}{2}$ for every partition $P \in \Part(p,r)$. Hence, it converges.
\end{proof}

\begin{vb}\label{exadditive}
We easily compute that 
\begin{itemize}
\item $\SN((0),q) = 0$,
\item $\SN((1,\ldots,1,2,1\ldots,1),0) = i$ as it is computed by $\eps^{i,p}$,
\item $\SN((1,\ldots,1),q)= pq$,
\item $\SN((p),0)= \frac{p(p-1)}{2}$. 
\end{itemize}
\end{vb}

\begin{mydef}\label{defstackingdifference}
For $P \in \Part(p,p')$ and $P' \in \Part(p',p'')$, their \emph{stacking difference} is defined as 
$$\SD((P,q),(P',q')) := \SN(P \circ P',q+q') - \SN(P,q) - \SN(P',q').$$
\end{mydef}

\subsubsection{Computing $E^1$}
The zeroth page of the spectral sequence
$$E^0_{st} := F_s(\Lax^{\antishriek}\smsquaresub{\ka}\Lax)_{s+t} / F_{s+1}(\Lax^{\antishriek}\smsquaresub{\ka}\Lax)_{s+t}$$
consists of those elements of degree $s+t$ and stacking number $s$, with differential
\begin{equation} \label{d0}
d^0\left(\; \scalebox{0.85}{$\input{decomplax_nc.tikz}$}  \; \right) = \sum_{\SD((P_i,q_i),(P_{i+1},q_{i+1}))=0}  (-1)^{i-1+ \Thin(P_0)} \; \scalebox{0.85}{$\input{d_ka_decomplax.tikz}$}  \;
\end{equation}
(see \eqref{diffKos_1} and \eqref{diffKos_2}). Note that $E^0$ splits into subcomplexes
$$ E^0  = \overline{E}^0 \oplus \widetilde{E}^0$$
where $\widetilde{E}^0 := k \left\{ \; \input{mcP.tikz}   \;|\;   P \in \Part \right\}$. Moreover, $\widetilde{E}^0$ has a zero differential.
\begin{mydef}
For $P \in \Part(p,r)$ and $q\in \N$ such that $p+q-r \geq 1$, define the subcomplex
$$\overline{E}^0_{s,-}(P,q) := k \left \lbrace \quad \scalebox{0.85}{$\input{basiselement_bottomthin.tikz}$}	\in E^0_{s,-} \quad \Big\vert \quad \begin{aligned} &q= \sum_{i=1}^n q_i, P= P_1 \circ \ldots \circ P_n\\
 &\SD((P_i,q_i),(P_{i+1},q_{i+1}))= 0 \end{aligned} \quad \right \rbrace \sub \overline{E}^0_{s,-}$$ 
\end{mydef}

\begin{lemma}\label{lemmaE1_1}
We have an isomorphism of complexes
$$\overline{E}^0_{s,-} \cong \bigoplus_{ \substack{ P_0 \in \Part(p',r') \\ P_i \in \Part(p_{i-1},p_i)\\ \SD((P_i,q_i),(P_{i+1},q_{i+1}))>0} }  s^{p'-r'+2k+1-\Thin(P_0)}\bigotimes_{i=1}^k \overline{E}^0_{s,-}(P_i,q_i) $$
where $\Thin(P_0)$ equals $1$ if $P_0 = (0)$ and $0$ otherwise.
\end{lemma}
\begin{proof}
Set $A:=p'-r'+2k+1-\Thin(P_0)$, then we construct an explicit isomorphism
$$F: \overline{E}^0_{s,-} \longrightarrow \bigoplus_{ \substack{ P_0 \in \Part(p',r') \\ P_i \in \Part(p_{i-1},p_i)\\ \SD((P_i,q_i),(P_{i+1},q_{i+1}))>0} } s^{A}\bigotimes_{i=1}^k \overline{E}^0_{s,-}(P_i,q_i)$$
as follows: let $x$ be a basis element as in \eqref{basiselement} and let $(i_1,\ldots,i_k)$ be a sequence of indices such that for $i_{j-1} < t < i_{j}$ holds
$$\SD((P_t,q_t),(P_{t+1},q_{t+1})) = 0 \text{ and } \SD((P_{i_j},q_{i_j}), (P_{i_j+1},q_{i_j+1})) > 0$$
where $i_0 =0$ and $i_{k+1} = n$. Define the following basis elements
$$ x_j := \; \input{ComputeE1_1.tikz}\; \in \overline{E}^0_{s,-}(P_{i_{j-1} +1}\circ \ldots \circ P_{i_{j}}, q_{i_{j-1}+1}+ \ldots + q_{i_{j}})$$
and set $F(x) := (-1)^{(p'-r')n} s^{A}( x_1 \otimes \ldots \otimes x_k)$. For a given partition $P_0 \in \Part(p',r')$, it is clear that $F$ is bijective. We verify that $F$ is a chain map. We compute
\begin{align*}
&F d^0 (x) \\
&= \sum_{j=1}^{k+1} \sum_{ i_{j-1}< l < i_{j}} (-1)^{l- 1+\Thin(P_0)} F\left(\quad 
\scalebox{0.85}{$ \input{Fdka_0_nc.tikz}$} \quad \right)\\
&= \sum_{j=1}^{k+1} \sum_{i_{j-1}< l < i_j } (-1)^{l- 1+\Thin(P_0)+(n-1)(p'-r')}s^{A}(
x_1 \otimes \ldots \otimes \; 
\scalebox{0.85}{$ \input{ComputeE1_2.tikz}$} \; \otimes \ldots \otimes x_k) \\
&= \sum_{j=1}^{k+1} (-1)^{i_{j-1}-1+\Thin(P_0)+(n-1)(p'-r')} s^{A} (x_1 \otimes \ldots \otimes d^0(x_j) \otimes \ldots \otimes x_k)\\
&= (-1)^{1-\Thin(P_0)+(n-1)(p'-r')}s^{A} d^0(x_1 \otimes \ldots \otimes x_k) =  d^0 F(x)
\end{align*}
which proves the claim. 
\end{proof}

Let $\Assoc^{\antishriek}\circ_{\ka'} \Assoc$ be the Koszul complex of the operad $\Assoc$ with twisting morphism $\ka':\Assoc^{\antishriek} \longrightarrow \Assoc$. Observe that it corresponds to the subcomplex of $\Lax^{\antishriek} \smsquaresub{\ka} \Lax$ consisting of thin boxes. As $\Assoc$ is Koszul, we have
$$H(\Assoc^{\antishriek}\circ_{\ka'} \Assoc) \cong k \; \input{unitKosAssoc.tikz}$$
concentrated in degree $-1$. 

\begin{lemma}\label{lemmaE1_2}
Let $P\in \Part(p,r)$ and $q\in \N$ such that $p+q-r\geq 1$ and $\SN(P,q)=t$, then we have an isomorphism of complexes
$$\overline{E}^0_{t,-}(P,q) \cong s^t \left(\Assoc^{\antishriek} \circ_{\ka'} \Assoc \right)(p+q-r)$$	
\end{lemma}
\begin{proof}
We construct an explicit isomorphism 
$$F: \overline{E}^0_{t,-}(P,q) \longrightarrow s^t(\Assoc^{\antishriek} \circ_{\ka'} \Assoc)(p+q-r) $$ 
defined as 
$$F\left(\; \input{basiselement_bottomthin.tikz} \; \right) := (-1)^{tn} s^t\; \input{elementKosAssoc.tikz}$$
Upon inspection of the differential $d^0$ \eqref{d0}, $F$ is clearly a chain map. We show that $F$ is bijective by induction on $n\geq 1$. It suffices to prove the case for $n=2$: for every 
$$ y := \; \input{bijection.tikz}$$
we show there exists a unique pair of stackings $P_1\in \Part(p,a)$ and $P_2 \in \Part(a,r)$ such that $P = P_1 \circ P_2$ and $\SD((P_1,q'_1-p+a),(P_2,q'_2-a+r)) = 0$.

Upon inspection of the unique stacking computing the stacking number $\SN(P,q)$ (Example \eqref{vbmpq}) (or equivalently its formula \eqref{formulasn}), we observe that in order for the stacking difference to be zero, $m_{P_1,q_1+p-a}$ needs to first `incorporate' the elements $\eps^{r,r}$ and then the elements $\eps^{b,c}$ with $b<c$ with the highest $b$, which form the last part of the partition $P=(t_1,\ldots,t_r)$. Hence, if $q'_1 \leq q$, we set 
$$a:=p, \;  P_1:=(1,\ldots,1) \text{ and } P_2:=P.$$
If $q'_1 > q$, then let $i$ be the unique number such that $q_1'= q+ \sum_{j>i}(t_j-1) +d$ for some $0<d<t_i$. In this case, set $a:= \sum_{j>i}t_j +d +r-i$ and define
$$ P_1 :=   (\underbrace{1,\ldots,1}_{(\sum_{j<i}t_j+d)\text{-times}}, t_i - d+1, t_{i+1},\ldots,t_r ) \text{ and } P_2:= ( t_1, \ldots ,t_{i-1}, d, \underbrace{1,\ldots,1}_{(r-i)\text{-times}} )$$
It is then easy to see that $P_1 \circ P_2 = P$ and $\SN(P_1,0) + \SN(P_2,q) = \SN(P,q)$.
\end{proof}
For convenience, we identify $\eps^{a,b}= m_{P_i,q_i}$ with its pair $(P_i,q_i)$.  
\begin{lemma}\label{lemmaE1_3}
For every stacking $P\in \Part(p,r)$, the following are equivalent
\begin{enumerate}
\item $q\leq 1$,
\item there exists a unique sequence of elements $(\eps^{a_1,b_1}, \ldots, \eps^{a_k,b_k})$ having strictly positive stacking differences.
\end{enumerate}
Moreover, in this case $n=p-r+q$ and
$$ \eps^{a_i,b_i} = \eps^{j,p-i+1} \text{ if } \sum_{s<j} (t_s-1) < i \leq \sum_{s\leq j} (t_s-1) $$
If $q=1$, then $\eps^{a_{p-r+1},b_{p-r+1}} = \eps^{r,r}$.
\end{lemma}
\begin{proof}
First, we observe that 
$$\SD(\eps^{a,b}, \eps^{c,d} ) = \begin{cases} 0 & \text{ if } a>c \neq d \text{ or } a=b=c=d \\
1 & \text{ otherwise } \end{cases}$$
This is immediate upon inspection of the stackings in \eqref{towers}. Hence, in order to have strictly positive stacking numbers, the elements $\eps^{a_i,a_i}$ can only appear at the end of the sequence. In case, $q\geq 2$, this means that the sequence at least ends with $(\eps^{r,r},\eps^{r,r})$ which have zero stacking difference. Now, suppose $q\leq 1$, then having strictly positive stacking numbers also implies $1 \leq a_1 \leq \ldots \leq a_{p-r} \leq r$. Let $P =(t_1,\ldots,t_r)$, then we have $t_j-1$ is equal to the number of indices $i$ such that $a_i= j$. As a result, $P$ fully determines $a_1,\ldots,a_{p-r}$. If $q=0$, this suffices. If $q=1$, we again have that the last element of the sequence is $\eps^{r,r}$.
\end{proof}
\begin{opm}
Note that the total stacking number of the above sequence is $\sum_{i=1}^r i(t_i-1)+qr$.
\end{opm}

\begin{prop} \label{E1}
\begin{align*}
E^1 \cong \quad & \bigoplus_{\substack{P_0 \in \Part(p',r'),r'>0 \\ (t_1,\ldots,t_r) \in \Part(p,r) } } \quad k\; \input{E1_1.tikz} \quad \oplus \\
&\bigoplus_{\substack{P_0 \in \Part(p',r') \\ (t_1,\ldots,t_r) \in \Part(p,r) } } \quad k\; \input{E1_0.tikz} \quad \oplus \quad \bigoplus_{P_0 \in \Part(p',r')} \quad k\; \input{mcP.tikz}
\end{align*}
\end{prop}
\begin{opm}
Note that $E^1$ also splits according to arity $E^1(p,q,r)$ and thus $E^1(p,q,r) = 0$ if $q>1$.
\end{opm}
\begin{proof}
First, we observe that
$$ \widetilde{E}^0 \cong  \quad \bigoplus_{P_0 \in \Part(p',r')} \quad k\; \input{mcP.tikz}$$
with zero differential. Hence, it suffices to compute 
\begin{align*}
H(\overline{E}^0_{s,-},d^0) \cong \bigoplus_{ \substack{ P_0 \in \Part(p',r') \\ P_i \in \Part(p_{i-1},p_i)\\ \SD((P_i,q_i),(P_{i+1},q_{i+1}))>0} } H\left( s^{p'-r'+2k+1-\Thin(P_0)}\bigotimes_{i=1}^k \overline{E}^0_{s,-}(P_i,q_i)\right)
\end{align*}
due to Lemma \ref{lemmaE1_1}. From Lemma \ref{lemmaE1_2}, we deduce that $H(\overline{E}^0_{s,-}(P_i,q_i))$ vanishes, except if $p_{i-1}+q_i-p_i =1$. This splits in two cases: either $q_i= 1$ and $p_{i-1}=p_i$ or $q_i=0$ and $p_{i-1} = p_i +1$. Thus, the sole non-vanishing elements of $\overline{E}^0_{s,-}$ in cohomology are of the form
$$ \input{E1_2.tikz} $$
with strictly positive stacking numbers. The result then follows from Lemma \ref{lemmaE1_3}.
\end{proof}

\subsubsection{Computing $E^\infty$}

\begin{theorem}\label{laxnotkoszul}
The box operad $\Lax$ is not koszul.
\end{theorem}
\begin{proof}
For any $l\geq 0$, the differential $d^l$ of $E^l$ applies the differential $d_{\ka}$ and keeps only those terms whose total stacking number has increased by exactly $l$. Hence, for $l\geq 2$, $d^l$ is zero, except for
$$d^l\left(\;  \input{E1_11.tikz} \; \right ) = (-1)^{ \eps} \quad \input{Einfty.tikz} $$
when $l=p+p'-r$ with $\eps = (p-r)(p'-1)+p'$. As a result, both elements vanish in cohomology. As we can take $l$ arbitrarily large, every element with bottom box decorated by $m^c_{P_0,n}$ for $P_0 \in \Part(p',r')$ with $r'\leq 1$ vanishes in $E^\infty$.

We then obtain
\begin{align*}
H(\Lax^{\antishriek} \smsquaresub{\ka} \Lax) \cong  E^{\infty} \cong  \quad & k\; \input{unitKosAssoc.tikz} \quad \oplus  \bigoplus_{\substack{P_0 \in \Part(p',r'),r'> 1 \\ (t_1,\ldots,t_r) \in \Part(p,r) } } \quad k\; \input{E1_1.tikz} \quad \oplus \\
&\bigoplus_{\substack{P_0 \in \Part(p',r'),r'>1 \\ (t_1,\ldots,t_r) \in \Part(p,r) } } \quad k\; \input{E1_0.tikz} \quad \oplus \quad \bigoplus_{P_0 \in \Part(p',r'),r'>1} \quad k\; \input{mcP.tikz}
\end{align*}
Hence, $\Lax^{\antishriek} \smsquaresub{\ka} \Lax$ is not acyclic.
\end{proof}
%
\begin{opm}\label{notquasiiso}
Note that the connected twisted complex $\Lax^{\antishriek} \smsquareconnsub{\ka} \Lax$ is neither acyclic, nor quasi-isomorphic to $\Lax^{\antishriek} \smsquaresub{\ka} \Lax$.
\end{opm}

\subsubsection{Acyclic subcomplex}
Upon inspection of the complex, we observe that we have a splitting
\begin{equation} \label{splitting}
\Lax^{\antishriek} \smsquaresub{\overline{\iota}} \Lax\cong \Lax^{\antishriek}_{\leq 1} \smsquaresub{\overline{\iota}} \Lax\oplus \Lax^{\antishriek}_{> 1} \smsquaresub{\overline{\iota}} \Lax
\end{equation}
where
$$ \Lax^{\antishriek}_{\leq 1} \smsquaresub{\overline{\iota}} \Lax := \left \lbrace \quad \input{koslax_element.tikz} \quad \Big \vert \quad P_0 \in \Part(p',r') \text{ such that }r'\leq 1 \right \rbrace \sub \Lax^{\antishriek} \smsquaresub{\overline{\iota}} \Lax $$
and 
$$ \Lax^{\antishriek}_{> 1} \smsquaresub{\overline{\iota}} \Lax := \left \lbrace \quad \input{koslax_element.tikz} \quad \Big \vert \quad P_0 \in \Part(p',r') \text{ such that }r'> 1 \right \rbrace \sub \Lax^{\antishriek} \smsquaresub{\overline{\iota}} \Lax.$$

\begin{cor}\label{laxacyclic}
The complex $\Lax^{\antishriek}_{\leq 1} \smsquaresub{\overline{\iota}} \Lax$ is acyclic.
\end{cor}
\begin{proof}
Upon inspection of the proof of Theorem \ref{laxnotkoszul}.
\end{proof}
\subsection{The minimal model $\Lax_{\infty}$} \label{parparminimal}

\subsubsection{The dg box operad $\Om\Lax^{\antishriek}$}

For $P \in \Part(p,r)$ and $q\in \N$, define
$$\boxop^{\Om\Lax^{\antishriek}}_{(1)} (P,q) := \left \lbrace (S,P_1,\ldots,P_n) \quad \Big \vert \quad  \begin{aligned} &S \text{ thin-quadratic stacking consisting of } n \text{ boxes,} \\
 &\text{resulting box of } S \text{ has arity } (p,q,r) \\
& P_1,\ldots,P_n \in \Part \text{ such that } \mu_{\Hh_S}(P_1,\ldots,P_n) = P \end{aligned} \right \rbrace $$
where $\mu_{\Hh_S}$ denotes the composition along the horizontal composite graph $\Hh_S$ of $S$. 

The cobar dg box operad $\Om\Lax^{\antishriek}$ is generated by the elements
$$ m_{P,q} \in \Om \Lax \sh(p,q,r) \text{ of degree } \begin{cases} p-r+q-1 & r>0 \\ q-2 & r= 0 \end{cases}$$
for every partition $P \in \Part(p,r)$, with differential
\begin{equation} \label{diff_cobarlax}
d(m_{P,q}) = \sum_{\substack{(S,P_1,\ldots,P_n) \in \boxop^{\Om\Lax^{\antishriek}}_{(1)}(P,q)}} (-1)^{\eps} \mu_S(m_{P_1,q_1} \otimes \ldots \otimes m_{P_n,q_n} )
\end{equation}
where $(-1)^{ \eps} = \begin{cases} (-1)^{ S^{\dd} + |m_{P_1,q_1}^c| +1 } & \text{ if } S \text{ of type }I \\
(-1)^{ S^{\dd} +|m^c_{P_1,q_1}| } & \text{ if } S \text{ of type }II \\
(-1)^{S^{\dd} + 1 } &\text{ if } S \text{ of type }III \end{cases}$.

We have an associated morphism of dg box operads
$$f_{\ka}: \Om \Lax^{\antishriek} \longrightarrow \Lax : m_{P,q} \longmapsto \begin{cases} m & \text{ if } (P,q) = ((0),2) \\
f & \text{ if } (P,q) = ((1),1) \\
c & \text{ if } (P,q)= ((2),0) \\
0 & \text{ otherwise } \end{cases} $$
\begin{opm}
Theorem \ref{laxnotkoszul} shows that $f_\ka$ is not a quasi-isomorphism. This is already apparent in lower arities, for instance 
$$H_0( \Om \Lax^{\antishriek}(2,1,2))  = k \quad \input{x11.tikz} \quad \oplus k\quad \input{x1_x1.tikz} \quad  \text{, while } \Lax(2,1,1) = k \quad \input{m11.tikz} \quad.$$
\end{opm}

\subsubsection{The dg box operad $\Laxinf$}

For $p+q \geq 2$, define
$$\boxop^{\Laxinf}_{(1)} (p,q) := \left \lbrace S \in \boxop((p_1,q_1,r_1),\ldots,(p_n,q_n,r_n);(p,q,r))  \quad \Big \vert \quad  \begin{aligned} 
&S \text{ thin-quadratic,} \\
& r_i = 1-\delta_{0p_i} \text{ and } r = 1-\delta_{0p}
 \end{aligned} \right \rbrace $$
where $\delta$ denotes the Kronecker delta.
\begin{mydef}
Let $\Laxinf$ be the dg box operad generated by the elements
$$m_{p,q} \in \Laxinf(p,q,r) \text{ of degree } p+q-2$$
for $p+q \geq 2$, with differential
\begin{equation}\label{difflaxinf}
d(m_{p,q}) = \sum_{S \in \boxop^{\Laxinf}_{(1)}(p,q)} (-1)^{\eps} \mu_S(m_{p_1,q_1} \otimes \ldots \otimes m_{p_n,q_n})
\end{equation}
where $(-1)^\eps$ as in \eqref{diff_cobarlax}.
\end{mydef}

\begin{lemma}
\begin{enumerate}
\item $\Laxinf \hookrightarrow \Om \Lax^{\antishriek}: m_{p,q} \longmapsto m_{(p),q}$ is a monomorphism of graded box operads.\label{laxinf1}
\item $\Laxinf$ is a dg box operad, that is, $d^2 =0$. \label{laxinf2}
\item $p:\Om \Lax^{\antishriek} \twoheadrightarrow \Laxinf : m_{P,q} \longmapsto \begin{cases} m_{p,q} & r\leq 1 \\
0 & r >1 \end{cases}$ is an epimorphism of dg box operads.\label{laxinf3}
\end{enumerate}
\end{lemma}
\begin{proof}
As both operads are free as graded box operads, \eqref{laxinf1} is immediate.

In order to prove both \eqref{laxinf2} and \eqref{laxinf3}, it suffices to show $p$ commutes with $d^{\Laxinf}$ and $d^{\Om \Lax^{\antishriek}}$. Indeed, we then have that
$$(d^{\Laxinf})^2(m_{p,q}) = (d^{\Laxinf})^2p(m_{(p),q}) = p (d^{\Om \Lax^{\antishriek}})^2(m_{(p),q}) = 0.$$

In order to show that $pd^{\Om \Lax^{\antishriek}} = d^{\Laxinf}p$, we first observe that $d^{\Laxinf}(m_{p,q})= p d^{\Om \Lax^{\antishriek}}(m_{(p),q})$. Hence, it suffices to show the equality for $P \in \Part(p,r)$ such that $r>1$, that is,
$$pd^{\Om \Lax^{\antishriek}}(m_{P,q}) = 0 = d^{\Laxinf} p(m_{P,q})$$ 
Now, let $(S,P_1,\ldots,P_n) \in \boxop^{\Om \Lax^{\antishriek}}_{(1)}(P,q)$, then it suffices to show that $P_i \in \Part(p_i,r_i)$ with $r_i>1$ for some $i$. If $S$ is of type $II$, it is clear that $P_1= P$. If $S$ is of type $III$, then $\Hh_S$ consists of two connected components $\Hh_S^{D}$ and $\Hh_S^{U}$: $\Hh_S^{D}$ consists of the thin bottom box labelled $1$, and $\Hh_S^{U}$ consists of the other $n-1$ non-thin boxes. Assume every non-thin box $i$ has a single horizontal output, i.e. $r_i =1$. Then $\Hh_S^U$ is a tree and thus $P = \mu_{\Hh_S}(P_1,\ldots,P_n) \in \Part(p,1)$, a contradiction.
\end{proof}
\begin{opm}
Writing out equation \eqref{difflaxinf} for $m_{1,q}$, we see that $\Mor(\Assoc)_\infty$ is a dg box suboperad of $\Laxinf$. 
\end{opm}

\subsubsection{The complex $\Lax^{\antishriek} \smsquaresub{\overline{\iota}} \Laxinf$}

Let $\overline{\iota}: \Lax^{\antishriek} \longrightarrow \Laxinf$ be the twisting morphism corresponding to $p$. The associated complex $\Lax^{\antishriek} \smsquaresub{\overline{\iota}} \Laxinf$ has differential $d_{\overline{\iota}}= \sthin\inv \Lax^{\antishriek} \smsquare d_{\Laxinf} + d^r_{\overline{\iota}}$ and it is generated as $k$-module by elements of the form
\begin{equation}
\input{basiselementlaxinf.tikz}
\end{equation}
where $P_0 \in \Part(p',r')$ and $x_i \in \Laxinf(p_{i-1},q_i,p_i)$. Upon inspection of the differential $d_{\overline{\iota}}$, we have a splitting 
$$\Lax^{\antishriek} \smsquaresub{\overline{\iota}} \Laxinf\cong \Lax^{\antishriek}_{\leq 1} \smsquaresub{\overline{\iota}} \Laxinf\oplus \Lax^{\antishriek}_{> 1} \smsquaresub{\overline{\iota}} \Laxinf$$
analogous to \eqref{splitting}.

\begin{lemma}\label{laxinfacyclic}
The complex $\Lax^{\antishriek}_{\leq 1} \smsquaresub{\overline{\iota}} \Laxinf$ is acyclic.
\end{lemma}
\begin{proof}
We consider again the homotopy \eqref{homotopy1} from Lemma \ref{acyclic}. Indeed, it is still well-defined as every generator of $\Laxinf$ has a `representative' in $\Lax^{\antishriek}_{r\leq 1}$.
\end{proof}
\begin{opm}
First, observe that the candidate homotopy from \eqref{homotopy1} is also well-defined for $\Lax^{\antishriek} \smsquaresub{\overline{\iota}} \Laxinf$. However, it does not define a null homotopy.

Second, observe that $\Lax^{\antishriek}_{\leq 1} \smsquaresub{\iota} \Om \Lax^{\antishriek}$ is a subcomplex of $\Lax^{\antishriek} \smsquaresub{\iota} \Om \Lax^{\antishriek}$, but not a summand. In fact, Theorem \ref{laxnotkoszul} and the proof of Theorem \ref{thmlaxinf} show that it will not be acyclic. Indeed, the candidate homotopy \eqref{homotopy1} from Lemma \ref{acyclic} is ill-defined.
\end{opm}

\subsubsection{The quasi-isomorphism}

\begin{lemma}
We have a morphism of dg box operads
$$f:= f_{\ka}i: \Laxinf \longrightarrow \Lax : m_{p,q} \longmapsto \begin{cases} m & (p,q)= (0,2) \\ f & (p,q) = (1,1) \\ c & (p,q)= (2,0) \\
0 & \text{ otherwise } \end{cases}.$$
\end{lemma}
\begin{proof}
Observe that the sole thin-quadratic stackings are made up of the elements $m,f$ and $c$ are those featured in the relations of $\Lax$ (see Definition \ref{lax}). It is then easy to compute that the image of the boundary of the elements $m_{0,3},m_{1,2}, m_{2,1}$ and $m_{3,0}$ corresponds exactly to the relations \eqref{laxrel1}, \eqref{laxrel2}, \eqref{laxrel3} and \eqref{laxrel4} respectively.
\end{proof}

\begin{theorem}\label{thmlaxinf}
The dg box operad $\Laxinf$ is the minimal model of $\Lax$.
\end{theorem}
\begin{proof}
First, observe that $\Laxinf$ is free as a graded box operad and the projection of the differential to its module of generators $s_{\thin}^{-1}\Lax^{\antishriek} \overset{d}{\longrightarrow} \Laxinf \twoheadrightarrow s_{\thin}^{-1}\Lax^{\antishriek}$ is zero. It thus suffices to show that $f$ is a quasi-isomorphism. 

We observe that $ \Lax^{\antishriek}_{\leq 1} \smsquare f: \Lax^{\antishriek}_{\leq 1} \smsquaresub{\overline{\iota}} \Laxinf \longrightarrow \Lax^{\antishriek}_{\leq 1} \smsquaresub{\overline{\iota}} \Lax$ is a quasi-isomorphism of acyclic complexes (due to Corollary \ref{laxacyclic} and Lemma \ref{laxinfacyclic}). In order to conclude that $f$ is a quasi-isomorphism, we verify that the proof of the Inclined Box Operadic Lemma (Proposition \ref{operadiclemmainclined}) restricts well to the subcomplexes $\Lax^{\antishriek}_{\leq 1} \smsquaresub{\overline{\iota}} \Laxinf$ and $\Lax^{\antishriek}_{\leq 1} \smsquaresub{\overline{\iota}} \Lax$. This follows directly from the observation that the Lemmas \ref{aritylemma}, \ref{filtration} and \ref{lemspectral} still hold for the above mentioned subcomplexes. The proof then holds mutatis mutandis. We conclude that $f$ is indeed a quasi-isomorphism.
\end{proof}
Following \cite{markl2004}, we call $\Laxinf$-algebras \emph{strongly homotopy lax prestacks} or \emph{lax prestacks up to homotopy}. 

\begin{opm}
Observe that we have a commuting diagram
$$
\begin{tikzcd}
\Laxinf \arrow[r, "\sim"]                             & \Lax                         \\
\Mor(\Assoc)_\infty \arrow[r, "\sim"] \arrow[u, hook] & \Mor(\Assoc) \arrow[u, hook]
\end{tikzcd}$$
where the horizontal morphisms are quasi-isomorphism and the vertical morphisms injections.
\end{opm}

Let $\uuu$ be a small category whose morphisms have finite factorisations (see $\S \ref{parparalgebras}$).
\begin{mydef}\label{alglaxinf}
A lax prestack up to homotopy $\A$ over $\uuu$ consists of an element $m_{p,q}$ for every $p+q \geq 2$ such that 
$$d(m_{p,q}) = \sum_{S \in \boxop^{\Laxinf}_{(1)}(p,q)} (-1)^{\eps} \mu_S^{\End(\A)}(m_{p_1,q_1} \otimes \ldots \otimes m_{p_n,q_n}).$$
where $(-1)^\eps$ as in \eqref{diff_cobarlax}. 

The element $m_{p,q}$ consists of a collection of maps 
$$m_{p,q}^{\si,A} : \A(U_0)(A_0,A_1) \otimes \ldots \otimes \A(U_0)(A_{q-1},A_q) \longrightarrow \A(U_p)(u_p\st \ldots u_1\st A_0, (u_p \ldots u_1)\st A_q)$$
of degree $p+q-2$ for every $p$-simplex $\si = (U_0 \overset{u_1}{\rightarrow} \ldots \overset{u_p}{\rightarrow} U_p) \in N_{p}(\uuu)$ and $q$-tuple of objects $A= (A_0,\ldots,A_q)\in \A(U_0)$.
\end{mydef}
\begin{opm}
For a lax prestack up to homotopy $\A$, for every $U \in \uuu$ the data 
$$(\A(U), m_{0,q}^{U,-})_{q \geq 2}$$
defines a $\Ainf$-category, and for every morphism $U \overset{u}{\rightarrow} V$ in $\uuu$ the data 
$$(m_{1,q}^{u,-})_{q \geq 0}: \A(U) \longrightarrow \A(V)$$
defines a $\Ainf$-functor.
\end{opm}

%% file: main.bbl
\def\cprime{$'$}
\providecommand{\bysame}{\leavevmode\hbox to3em{\hrulefill}\thinspace}
\providecommand{\MR}{\relax\ifhmode\unskip\space\fi MR }
\providecommand{\MRhref}[2]{%
  \href{http://www.ams.org/mathscinet-getitem?mr=#1}{#2}
}
\providecommand{\href}[2]{#2}
\begin{thebibliography}{10}

\bibitem{artintatevandenbergh}
M.~Artin, J.~Tate, and M.~Van~den Bergh, \emph{Some algebras associated to
  automorphisms of elliptic curves}, The {G}rothendieck {F}estschrift, {V}ol.\
  {I}, Progr. Math., vol.~86, Birkh\"auser Boston, Boston, MA, 1990,
  pp.~33--85. \MR{MR1086882 (92e:14002)}

\bibitem{aurouxorlov2008}
D.~Auroux, L.~Katzarkov, and D.~Orlov, \emph{Mirror symmetry for weighted
  projective planes and their noncommutative deformations}, Ann. of Math. (2)
  \textbf{167} (2008), no.~3, 867--943. \MR{2415388}

\bibitem{balchin2021}
S.~Balchin, \emph{A handbook of model categories}, Algebra and Applications,
  vol.~27, Springer, Cham, [2021] \copyright 2021. \MR{4385504}

\bibitem{barannikov2007}
S.~Barannikov, \emph{Modular operads and {B}atalin-{V}ilkovisky geometry}, Int.
  Math. Res. Not. IMRN (2007), no.~19, Art. ID rnm075, 31. \MR{2359547}

\bibitem{bataninmarkl2015}
M.~Batanin and M.~Markl, \emph{Operadic categories and duoidal {D}eligne's
  conjecture}, Adv. Math. \textbf{285} (2015), 1630--1687. \MR{3406537}

\bibitem{bataninmarkl2023a}
\bysame, \emph{Koszul duality for operadic categories}, Compositionality
  \textbf{5} (2023), no.~4, 56. \MR{4611064}

\bibitem{bataninmarkl2023}
\bysame, \emph{Operadic categories as a natural environment for {K}oszul
  duality}, Compositionality \textbf{5} (2023), no.~3, 46. \MR{4599796}

\bibitem{bataninweber2011}
M.~Batanin and M.~Weber, \emph{Algebras of higher operads as enriched
  categories}, Appl. Categ. Structures \textbf{19} (2011), no.~1, 93--135.
  \MR{2772603}

\bibitem{batanin2007}
M.~A. Batanin, \emph{Symmetrisation of {$n$}-operads and compactification of
  real configuration spaces}, Adv. Math. \textbf{211} (2007), no.~2, 684--725.
  \MR{2323542}

\bibitem{batanin2008}
\bysame, \emph{The {E}ckmann-{H}ilton argument and higher operads}, Adv. Math.
  \textbf{217} (2008), no.~1, 334--385. \MR{2365200}

\bibitem{bergermoerdijk2006}
C.~Berger and I.~Moerdijk, \emph{The {B}oardman-{V}ogt resolution of operads in
  monoidal model categories}, Topology \textbf{45} (2006), no.~5, 807--849.
  \MR{2248514}

\bibitem{bergermoerdijk2007}
\bysame, \emph{Resolution of coloured operads and rectification of homotopy
  algebras}, Categories in algebra, geometry and mathematical physics, Contemp.
  Math., vol. 431, Amer. Math. Soc., Providence, RI, 2007, pp.~31--58.
  \MR{2342815}

\bibitem{burroni1971}
A.~Burroni, \emph{{$T$}-cat\'{e}gories (cat\'{e}gories dans un triple)},
  Cahiers Topologie G\'{e}om. Diff\'{e}rentielle \textbf{12} (1971), 215--321.
  \MR{308236}

\bibitem{CruttwellSchulman2010}
G.~S.~H. Cruttwell and M.~A. Shulman, \emph{A unified framework for generalized
  multicategories}, Theory Appl. Categ. \textbf{24} (2010), No. 21, 580--655.
  \MR{2770076}

\bibitem{vanhermanslowen2022}
H.~Dinh~Van, L.~Hermans, and W.~Lowen, \emph{Operadic structure on the
  {G}erstenhaber-{S}chack complex for prestacks}, Selecta Math. (N.S.)
  \textbf{28} (2022), no.~3, Paper No. 47, 63. \MR{4388764}

\bibitem{dinhvanhermanslowen2023}
\bysame, \emph{Box operads and higher {G}erstenhaber brackets}, arXiv
  Mathematics e-prints (2023).

\bibitem{dinhvanliulowen2017}
H.~Dinh~Van, L.~Liu, and W.~Lowen, \emph{Non-commutative deformations and
  quasi-coherent modules}, Selecta Math. (N.S.) \textbf{23} (2017), no.~2,
  1061--1119. \MR{3624905}

\bibitem{DVL}
H.~Dinh~Van and W.~Lowen, \emph{The {G}erstenhaber-{S}chack complex for
  prestacks}, Adv. Math. \textbf{330} (2018), 173--228.

\bibitem{doubek2011}
M.~Doubek, \emph{On resolutions of diagrams of algebras}, arXiv Mathematics
  e-prints (2011).

\bibitem{doubek2012}
\bysame, \emph{Gerstenhaber-{S}chack diagram cohomology from the operadic point
  of view}, J. Homotopy Relat. Struct. \textbf{7} (2012), no.~2, 165--206.
  \MR{2988945}

\bibitem{foxmarkl1997}
T.~F. Fox and M.~Markl, \emph{Distributive laws, bialgebras, and cohomology},
  Operads: {P}roceedings of {R}enaissance {C}onferences ({H}artford,
  {CT}/{L}uminy, 1995), Contemp. Math., vol. 202, Amer. Math. Soc., Providence,
  RI, 1997, pp.~167--205. \MR{1436921}

\bibitem{fresse1967}
B.~Fresse, \emph{Modules over operads and functors}, Lecture Notes in
  Mathematics, vol. 1967, Springer-Verlag, Berlin, 2009. \MR{2494775}

\bibitem{fresse2010}
\bysame, \emph{Props in model categories and homotopy invariance of
  structures}, Georgian Math. J. \textbf{17} (2010), no.~1, 79--160.
  \MR{2640648}

\bibitem{fukaya1993}
K.~Fukaya, \emph{Morse homotopy, {$A^\infty$}-category, and {F}loer
  homologies}, Proceedings of {GARC} {W}orkshop on {G}eometry and {T}opology
  '93 ({S}eoul, 1993), Lecture Notes Ser., vol.~18, Seoul Nat. Univ., Seoul,
  1993, pp.~1--102. \MR{1270931}

\bibitem{gan2003}
W.~L. Gan, \emph{Koszul duality for dioperads}, Math. Res. Lett. \textbf{10}
  (2003), no.~1, 109--124. \MR{1960128}

\bibitem{gerstenhaberschack}
M.~Gerstenhaber and S.~D. Schack, \emph{On the deformation of algebra morphisms
  and diagrams}, Trans. Amer. Math. Soc. \textbf{279} (1983), no.~1, 1--50.
  \MR{MR704600 (85d:16021)}

\bibitem{gerstenhaberschack1}
\bysame, \emph{Algebraic cohomology and deformation theory}, Deformation theory
  of algebras and structures and applications (Il Ciocco, 1986), NATO Adv. Sci.
  Inst. Ser. C Math. Phys. Sci., vol. 247, Kluwer Acad. Publ., Dordrecht, 1988,
  pp.~11--264. \MR{MR981619 (90c:16016)}

\bibitem{gerstenhaberschack2}
\bysame, \emph{The cohomology of presheaves of algebras. {I}. {P}resheaves over
  a partially ordered set}, Trans. Amer. Math. Soc. \textbf{310} (1988), no.~1,
  135--165. \MR{MR965749 (89k:16052)}

\bibitem{getzler2009}
E.~Getzler, \emph{Operads revisited}, Algebra, arithmetic, and geometry: in
  honor of {Y}u. {I}. {M}anin. {V}ol. {I}, Progr. Math., vol. 269,
  Birkh\"{a}user Boston, Boston, MA, 2009, pp.~675--698. \MR{2641184}

\bibitem{getzlerjones}
E.~Getzler and J.~D.~S. Jones, \emph{Operads, homotopy algebra and iterated
  integrals for double loop spaces}, arXiv Mathematics e-prints (1994).

\bibitem{getzlerkapranov1998}
E.~Getzler and M.~Kapranov, \emph{Modular operads}, Compositio Math.
  \textbf{110} (1998), no.~1, 65--126. \MR{1601666}

\bibitem{getzlerkapranov1995}
E.~Getzler and M.~M. Kapranov, \emph{Cyclic operads and cyclic homology},
  Geometry, topology, \& physics, Conf. Proc. Lecture Notes Geom. Topology, IV,
  Int. Press, Cambridge, MA, 1995, pp.~167--201. \MR{1358617}

\bibitem{ginzburgkapranov1994}
V.~Ginzburg and M.~Kapranov, \emph{Koszul duality for operads}, Duke Math. J.
  \textbf{76} (1994), no.~1, 203--272. \MR{1301191}

\bibitem{hawkins}
E.~Hawkins, \emph{Operations on the {H}ochschild bicomplex of a diagram of
  algebras}, Adv. Math. \textbf{428} (2023), Paper No. 109156, 80. \MR{4603781}

\bibitem{hermanslowenOC}
L.~Hermans, \emph{Box operadic cohomology for prestacks}, in preparation.

\bibitem{hinich}
V.~Hinich, \emph{Homological algebra of homotopy algebras}, Comm. Algebra
  \textbf{25} (1997), no.~10, 3291--3323. \MR{MR1465117 (99b:18017)}

\bibitem{hoffbecklerayvallette2021}
E.~Hoffbeck, J.~Leray, and B.~Vallette, \emph{Properadic homotopical calculus},
  Int. Math. Res. Not. IMRN (2021), no.~5, 3866--3926. \MR{4227587}

\bibitem{kaufmannward2017}
R.~M. Kaufmann and B.~C. Ward, \emph{Feynman categories}, Ast\'{e}risque
  (2017), no.~387, vii+161. \MR{3636409}

\bibitem{kaufmannward2023}
\bysame, \emph{Koszul {F}eynman categories}, Proc. Amer. Math. Soc.
  \textbf{151} (2023), no.~8, 3253--3267. \MR{4591764}

\bibitem{kontsevich2}
M.~Kontsevich, \emph{Homological algebra of mirror symmetry}, Proceedings of
  the International Congress of Mathematicians, Vol.\ 1, 2 (Z\"urich, 1994)
  (Basel), Birkh\"auser, 1995, pp.~120--139. \MR{MR1403918 (97f:32040)}

\bibitem{kontsevich}
\bysame, \emph{Deformation quantization of {P}oisson manifolds}, Lett. Math.
  Phys. \textbf{66} (2003), no.~3, 157--216. \MR{MR2062626}

\bibitem{kontsevichsoibelman2001}
M.~Kontsevich and Y.~Soibelman, \emph{Homological mirror symmetry and torus
  fibrations}, Symplectic geometry and mirror symmetry ({S}eoul, 2000), World
  Sci. Publ., River Edge, NJ, 2001, pp.~203--263. \MR{1882331}

\bibitem{Koudenburg2020}
S.~R. Koudenburg, \emph{Augmented virtual double categories}, Theory Appl.
  Categ. \textbf{35} (2020), Paper No. 10, 261--325. \MR{4087662}

\bibitem{lack2018}
S.~Lack, \emph{Operadic categories and their skew monoidal categories of
  collections}, High. Struct. \textbf{2} (2018), no.~1, 1--29. \MR{3917424}

\bibitem{lackstreet2014}
S.~Lack and R.~Street, \emph{Triangulations, orientals, and skew monoidal
  categories}, Adv. Math. \textbf{258} (2014), 351--396. \MR{3190430}

\bibitem{leinsterFc}
T.~{Leinster}, \emph{{fc-multicategories}}, arXiv Mathematics e-prints (1999),
  math/9903004.

\bibitem{Leinster2004}
\bysame, \emph{Higher operads, higher categories}, London Mathematical Society
  Lecture Note Series, vol. 298, Cambridge University Press, Cambridge, 2004.
  \MR{2094071}

\bibitem{lodayvallette}
J.-L. Loday and B.~Vallette, \emph{Algebraic operads}, Grundlehren der
  mathematischen Wissenschaften [Fundamental Principles of Mathematical
  Sciences], vol. 346, Springer, Heidelberg, 2012. \MR{2954392}

\bibitem{lowenvandenberghhoch}
W.~Lowen and M.~Van~den Bergh, \emph{Hochschild cohomology of abelian
  categories and ringed spaces}, Advances in Math. \textbf{198} (2005), no.~1,
  172--221, preprint math.KT/0405227.

\bibitem{lowenvandenberghCCT}
W.~Lowen and M.~van~den Bergh, \emph{A {H}ochschild cohomology comparison
  theorem for prestacks}, Trans. Amer. Math. Soc. \textbf{363} (2011), no.~2,
  969--986. \MR{2728592 (2012c:16033)}

\bibitem{markl2002}
M.~Markl, \emph{Homotopy diagrams of algebras}, Proceedings of the 21st
  {W}inter {S}chool ``{G}eometry and {P}hysics'' ({S}rn\'{\i}, 2001), no.~69,
  2002, pp.~161--180. \MR{1972432}

\bibitem{markl2004}
\bysame, \emph{Homotopy algebras are homotopy algebras}, Forum Math.
  \textbf{16} (2004), no.~1, 129--160. \MR{2034546}

\bibitem{markl2010}
\bysame, \emph{Intrinsic brackets and the {$L_\infty$}-deformation theory of
  bialgebras}, J. Homotopy Relat. Struct. \textbf{5} (2010), no.~1, 177--212.
  \MR{2812919}

\bibitem{marklmerkulovshadrin2009}
M.~Markl, S.~Merkulov, and S.~Shadrin, \emph{Wheeled {PROP}s, graph complexes
  and the master equation}, J. Pure Appl. Algebra \textbf{213} (2009), no.~4,
  496--535. \MR{2483835}

\bibitem{merkulovwillwacher2018}
S.~Merkulov and T.~Willwacher, \emph{Deformation theory of {L}ie bialgebra
  properads}, Geometry and physics. {V}ol. {I}, Oxford Univ. Press, Oxford,
  2018, pp.~219--247. \MR{3932263}

\bibitem{milles2011}
J.~Mill\`es, \emph{Andr\'{e}-{Q}uillen cohomology of algebras over an operad},
  Adv. Math. \textbf{226} (2011), no.~6, 5120--5164. \MR{2775896}

\bibitem{muro2014}
F.~Muro, \emph{Homotopy theory of non-symmetric operads, {II}: {C}hange of base
  category and left properness}, Algebr. Geom. Topol. \textbf{14} (2014),
  no.~1, 229--281. \MR{3158759}

\bibitem{petersen2013}
D.~Petersen, \emph{The operad structure of admissible {$G$}-covers}, Algebra
  Number Theory \textbf{7} (2013), no.~8, 1953--1975. \MR{3134040}

\bibitem{rizzardovandenbergh}
A.~Rizzardo and M.~Van~den Bergh, \emph{Scalar extensions of derived categories
  and non-{F}ourier-{M}ukai functors}, Adv. Math. \textbf{281} (2015),
  1100--1144. \MR{3366860}

\bibitem{rizzardovandenberghneeman2019}
A.~Rizzardo, M.~Van~den Bergh, and A.~Neeman, \emph{An example of a
  non-{F}ourier-{M}ukai functor between derived categories of coherent
  sheaves}, Invent. Math. \textbf{216} (2019), no.~3, 927--1004. \MR{3955712}

\bibitem{seidel2002}
P.~Seidel, \emph{Fukaya categories and deformations}, Proceedings of the
  {I}nternational {C}ongress of {M}athematicians, {V}ol. {II} ({B}eijing,
  2002), Higher Ed. Press, Beijing, 2002, pp.~351--360. \MR{1957046}

\bibitem{stasheff1963}
J.~D. Stasheff, \emph{Homotopy associativity of {$H$}-spaces. {I}, {II}},
  Trans. Amer. Math. Soc. \textbf{108} (1963), 293--312, {\bf 108} (1963),
  275-292; ibid. \MR{158400}

\bibitem{szlachanyi2012}
K.~Szlach\'{a}nyi, \emph{Skew-monoidal categories and bialgebroids}, Adv. Math.
  \textbf{231} (2012), no.~3-4, 1694--1730. \MR{2964621}

\bibitem{vallette2007}
B.~Vallette, \emph{A {K}oszul duality for {PROP}s}, Trans. Amer. Math. Soc.
  \textbf{359} (2007), no.~10, 4865--4943. \MR{2320654}

\bibitem{vallette2020}
\bysame, \emph{Homotopy theory of homotopy algebras}, Ann. Inst. Fourier
  (Grenoble) \textbf{70} (2020), no.~2, 683--738. \MR{4105949}

\bibitem{vandenbergh2}
M.~Van~den Bergh, \emph{Noncommutative quadrics}, Int. Math. Res. Not. IMRN
  (2011), no.~17, 3983--4026. \MR{2836401}

\bibitem{vanderlaan2004}
P.~van~der Laan, \emph{Coloured koszul duality and strongly homotopy operads},
  arXiv Mathematics e-prints (2004).

\bibitem{ward2022}
B.~C. Ward, \emph{Massey products for graph homology}, Int. Math. Res. Not.
  IMRN (2022), no.~11, 8086--8161. \MR{4425832}

\bibitem{whiteyau2018}
D.~White and D.~Yau, \emph{Bousfield localization and algebras over colored
  operads}, Appl. Categ. Structures \textbf{26} (2018), no.~1, 153--203.
  \MR{3749666}

\end{thebibliography}
